\documentclass{article}

\usepackage{subcaption}
\usepackage{graphicx}
\usepackage{svg}
\usepackage{mathtools,amsthm,amsmath,amsfonts,amssymb}
\usepackage{algorithm}
\usepackage{xcolor}
\usepackage{hyperref}
\hypersetup{
    colorlinks,
    linkcolor={blue!70!black},
    citecolor={green!60!black},
    urlcolor={magenta!50!black}
}

\usepackage{geometry}[margin=1in]
\usepackage{algorithm}
\usepackage{algpseudocode}
\usepackage[title]{appendix}

\newcommand{\N}{\mathbb{N}}
\newcommand{\R}{\mathbb{R}}
\newcommand{\Y}{\mathcal{Y}}
\newcommand{\X}{\mathcal{X}}
\newcommand{\B}{\mathcal{B}}
\newcommand{\1}{\mathbf{1}}
\newcommand{\sgn}{\mathrm{sgn}}
\newcommand{\DGW}{\mathrm{DGW}}
\newcommand{\CGW}{\mathrm{CGW}}
\newcommand{\IGW}{\mathrm{IGW}}
\newcommand{\GW}{\mathrm{GW}}
\newcommand{\eig}{\mathrm{eig}}
\newcommand{\Leb}{\mathrm{Leb}}

\renewcommand{\P}{\mathcal{P}}

\DeclareMathOperator*{\argmax}{argmax}
\DeclareMathOperator*{\argmin}{argmin}

\newtheorem{theorem}{Theorem}
\newtheorem{lemma}{Lemma}
\newtheorem{example}{Example}
\newtheorem{corollary}{Corollary}
\newtheorem{proposition}{Proposition}
\theoremstyle{definition}
\newtheorem{definition}{Definition}
\theoremstyle{remark}
\newtheorem*{remark}{Remark}
\newtheorem{assumption}{Assumption/Notation}

\begin{document}

\begin{center}
\textbf{\large Beyond Procrustes distances: a multilinear Gromov-Wasserstein distance capturing chirality}

\vspace{7mm}

 Cl{\'e}ment Soubrier\textsuperscript{1}, Geoffrey Woollard\textsuperscript{2},
 Andrew Warren\textsuperscript{1,3},
 Khanh Dao Duc\textsuperscript{1}
\end{center}
\vspace{5mm}
{\small
$^{1}$ Department of Mathematics, University of British Columbia, Vancouver, BC V6T 1Z4, Canada\\
$^{2}$ Department of Computer Science, University of British Columbia, Vancouver, BC V6T 1Z4, Canada\\
$^{3}$ Mathematical Institute, Utrecht University, Utrecht, 3584 CD, The Netherlands}
\vspace{5mm}

\begin{abstract}
Efficiently and robustly analyzing shape data is critical across many scientific disciplines. While chirality is a fundamental property in numerous applications---most notably in molecular science---existing shape analysis metrics either fail to distinguish enantiomers, or rely on costly alignment. To address this gap, we introduce a multilinear generalization of the Gromov-Wasserstein objective, which yields an alignment-free distance between shapes, represented as probability distributions quotiented by a Lie group $G$. In  particular, for $G = SO(d)$, we introduce the Chiral Gromov-Wasserstein ($\mathrm{CGW}$) distance, sensitive to chirality. We establish robustness and sample complexity properties for the multilinear Gromov-Wasserstein distances and develop efficient algorithms to compute them, reformulating the underlying  optimization problem by projecting couplings onto a low-dimensional space. We derive algorithms for both local and approximate global solutions, yielding an approximation scheme, which is polynomial in the sample size and precision. We validate the framework through numerical experiments that demonstrate the effectiveness of $\mathrm{CGW}$ as a shape metric for chiral objects.
\end{abstract}
\vspace{7mm}


\section{Introduction}

Choosing appropriate metrics is key to efficiently and robustly analyze shape data, a need shared by disciplines ranging from archaeology and medicine to biology and chemistry \cite{dryden2016statistical}. In the context of pharmacology and chemistry, the handedness of molecules influences their interactions with other species: for example,  D-penicillamine is a drug used to treat rheumatoid arthritis, whereas its reflected version (L-penicillamine) is a toxic compound \cite{solomons2008organic}. In this context, the 2026 Nobel Prize in Chemistry was awarded to Henri Kagan and Kenso Soai for their work on synthesizing molecules with a given handedness \cite{Hargittai2026}, providing deep insights on the homochirality of biomolecules, such as amino acids \cite{jamroz2012chirality}. Yet, standard rotation invariant shape metrics (Procrustes-\cite{adamo2025depth} or Gromov-Wasserstein \cite{memoli2011gromov}) fail to distinguish  mirrored chiral shapes, which led us to ask: how to build a metric that can efficiently distinguish such chiral objects, or more generally, be sensitive to relevant classes of shape transforms?

To mathematically formalize this problem, we are representing shapes using probability measures on $\X \subset \R^d$ ($d=2$ or 3) with finite second moment, i.e. elements of $\P_2(\X)$. A natural distance on $\P_2(\X)$ is the 2-\emph{Wasserstein distance} coming from the theory of \emph{Optimal Transport}:
\begin{equation}\label{eq:W2}
    \mathrm{W}_2(\mu,\nu) \coloneqq \inf_{\pi\in\Pi(\mu,\nu)}\left(\int\|x-y\|^2d\pi\right)^{1/2}, \tag{$\mathrm{W}_2$}
\end{equation}  
where $\Pi(\mu,\nu)$ is the set of couplings between  $\mu$ and $\nu$, two elements of $\P_2(\X)$. To define a proper shape metric, we can make this distance invariant under orthogonal transformations, by quotienting $\mathrm{W}_2$ under the action of $O(d)$. This yields the  \emph{standard Procrustes-Wasserstein distance}, which has applications in image analysis, natural language processing and computer vision \cite{adamo2025depth,grave2019unsupervised,zhang2017earth} as well as protein alignment \cite{jin2021two}. More generally, we can define a  \emph{Procrustes-Wasserstein} distance on $\P_2(\X)/G$ for any group $G$ of isometries (e.g. $SO(d)$, $E(d)$, $SE(d)$...) as 
\begin{equation}\label{eq:PW}
    \mathrm{PW}_G (\mu,\nu)=\inf_{g\in G} \mathrm{W}_2(g_\sharp\mu,\nu), \tag{$\mathrm{PW}_G$}
\end{equation}
where $g_\sharp\mu$ is the pushforward of $\mu$ by the map $g$. 
The standard Procrustes-Wasserstein problem can notably be solved for a fixed coupling or a fixed orthogonal transformation, but global solutions are hard to obtain in practice, since they require solving a non-convex optimization problem \cite{grave2019unsupervised}. Local optimization algorithms have been proposed but their performance is sensitive to initial values \cite{adamo2025depth,zhang2017earth}.    
Globally solving the problem has also been tackled using stochastic optimization, convex relaxation \cite{grave2019unsupervised} or for a finite set of invariances $G$ \cite{pal2025wasserstein}. However, to our knowledge  no  guarantee of convergence, or convergence rate for global optimization, has been proved for this problem. In this framework, distinguishing chiral shapes also requires using $G=SO(d)$ instead of $O(d)$, leading to similar computational challenges.

As an alternative to Procrustes-type distances, one can use a distance that is by construction directly invariant by group action, bypassing the need to optimize on it. This is achieved for $G=E(d)$ with the 2-\emph{Gromov-Wasserstein} distance:
\begin{equation}\label{eq:GW2}
    \GW_2 (\mu,\nu) \coloneqq  \inf_{\pi\in\Pi(\mu,\nu)}\left(\int\Big(\|x_1-x_2\|^2-\|y_1-y_2\|^2\Big)^2d\pi(x_1,y_1)d\pi(x_2,y_2)\right)^{1/2}. \tag{$\GW_2$}
\end{equation}
This distance was introduced as a way to compare measure metric spaces up to isometries \cite{memoli2011gromov,sturm2006geometry}, and its variants have recently been used in various applications, including aligning brain images \cite{thual2022aligning}, proteins \cite{tajmir2025alignment}, or comparing graphs \cite{alvarez2018gromov,chowdhury2019gromov} and structured objects \cite{vayer2020fused}. From a computational standpoint, computing the GW distance amounts to solving a concave quadratic programming problem, and is  NP-hard in general \cite{rioux2024entropic}. 
Additionally, the \ref{eq:GW2} distance can be generalized by replacing the Euclidean norm $\|\cdot\| $ with any similarity/cost function (e.g. inner product).  Recently, the \emph{Inner Gromov-Wasserstein} (IGW) distance was hence introduced using the inner product as similarity \cite{alvarez2018gromov,vayer2020contribution}. Using this $O(d)$ invariant pseudo-distance, we can  equip $\P_2(\X)$ with a Riemannian metric \cite{zhang2025gradient}, allowing one to define gradient flows, interpolate between shapes, and perform statistical analysis. 

In this paper, we propose a general framework, which extends the Gromov-Wasserstein distance, to define and compute a class of distances over shapes that are sensitive to specific   transforms. This includes, importantly, $SE(d)$, with a resulting distance capturing chirality that we call the \emph{Chiral Gromov-Wasserstein} (CGW) distance. To define such distances, we first  introduce in Section \ref{sec:theory} a family of multilinear Gromov-Wasserstein discrepancies, and demonstrate that under some assumptions, they yield proper distances on shapes. Furthermore, we characterize in this section their low-dimensional structure, study their stability with respect to variation of their argument and prove the existence of an optimal correspondence map between shapes.  In Section \ref{sec:algo} we leverage this low-dimensional structure to propose an algorithm that approximates a global solution of the $\GW_m$ problem. We show that this algorithm returns a solution with additive accuracy $\epsilon$ is polynomial time in sample size $N$ and $1/\epsilon$, providing, to our knowledge, the first estimation of complexity for globally solving a Gromov-Wasserstein type of problem in fixed dimension $d$. As our algorithm makes $\GW_m$ distances suitable for practical shape comparison, we run various numerical experiments in Section \ref{sec:numeric}, including comparisons of chiral molecules that highlight the potential use of $\CGW$ as a shape metric.

\section{Multilinear Gromov-Wasserstein distances and their fundamental properties}\label{sec:theory}
We are interested in studying a class of (pseudo-)distances between probability measures that are invariant by the action of a matrix Lie group $G\subset GL(d)$. Proofs of this section can be found in Supplementary \ref{sec:supp_theory_gw}. Let $\X,\Y\subset\R^d$ with non-empty interior and let us denote as $\P_2(\X)$, the set of probability measures on $\X$ with finite second moment. We also consider the action of $G$ on $\X$ as the map $(g,x)\in G\times\R^d\to g(x)\in \R^d$.
\begin{definition} 
    Given a map $g:\X\to\Y$, we denote $g^{\times n}: (x_1,\dots, x_n)\in \X^n\to (g(x_1),\dots, g(x_n))\in \Y^n$. A function $f$ on $(\R^d)^n$ is  invariant by the action of $G$  (or \emph{$G$-invariant}) if: $\forall  g\in G,\quad f = f\circ g^{\times n}$.
\end{definition}

\begin{definition}
    We define a \emph{$G$-shape} as a probability distribution (with finite second moment) up to the action of $G$, i.e. an element of the $G$-\emph{shape space} $\P_2(\R^d)/G$, associated with the action $(g,\mu)\in G\times\P_2(\R^d)\mapsto g_{\sharp}\mu$.
\end{definition}
Note that this definition differs from the classical definition of shapes \cite{ dryden2016statistical,kendall2009shape} which are invariant by  rigid transformation, scaling and re-parametrization. Indeed, our definition is not scaling invariant for a matrix Lie group $G$, but it generalizes shapes to other possible types of invariance. For simplicity, we now refer to $G$-shapes simply as shapes, where $G$ is specified when relevant.

\subsection{A distance between shapes}\label{sec:GW_dist}  We equip the shape space with a discrepancy as follows: Assume that there exists  a multilinear $n$-form (or $n$-tensor) $m$ over $\X$, that is invariant by $G$, and such that $G$ can be written as 
 \[G=G_{m} \coloneqq \{T:\X\to\R^d,\;  m  = m\circ T^{\times n}\},\]
 the group of transformations preserving the form $m$.  In Lemma \ref{lem:degeneracy}, we give a sufficient non-degeneracy condition on $m$ such that $G_{m}\subset GL(d,\R)$, guaranteeing that we have a matrix Lie group. Examples of such Lie groups and multilinear forms are: $O(d)$   with $m=\langle\cdot,\cdot\rangle$ the standard inner product and $SL(d)$ with $m=\det$. Now, we define the $G$-invariant discrepancy:

\begin{definition}
    Given a multilinear $n$-form $m$ on $\R^d$, $\X,\Y\subset\R^d$, and $\mu\in\P_2(\X), \nu\in\P_2(\Y)$, we define the \emph{multilinear Gromov-Wasserstein} discrepancy as:
\begin{equation}\label{eq:GW_m}
    \GW_{m}(\mu,\nu) \coloneqq  \inf_{\pi\in\Pi(\mu,\nu)}\left(\int_{(\X\times\Y)^{n}}\Big(m(\mathbf{x})-m(\mathbf{y})\Big)^2d\pi^{\otimes n}(\mathbf{x},\mathbf{y})\right)^{1/2},\tag{$\mathrm{GW}_m$}
\end{equation}
where $\pi^{\otimes n}(\mathbf{x},\mathbf{y})=\prod_{i=1}^n\pi(x_i,y_i)$ and $m(\mathbf{x})=m(x_1,\dots, x_n)$ with $\mathbf{x}=(x_1,\dots, x_n)\in (\R^d)^n$. By construction $\GW_{m}$ is invariant under the action of the Lie group $G$.
\end{definition}
Here, the finite second moment implies $\GW_{m}(\mu,\nu)<+\infty$. A classical optimal transport fact, that we will use extensively, is the compactness of the set of couplings, implying that there exists an optimal transport plan $\pi$ for the $\GW_m$ problem (see the proof of Proposition \ref{prop:stability-general}).
\begin{remark}
    For a finite family of multilinear forms $(m)$, we can generalize the definition of \ref{eq:GW_m} as:
    \begin{equation*}
        \GW_{(m)}(\mu,\nu) \coloneqq  \inf_{\pi\in\Pi(\mu,\nu)}\left(\sum_{m_i\in(m)}\int_{(\X\times\Y)^{n(m_i)}}\Big(m_i(\mathbf{x})-m_i(\mathbf{y})\Big)^2d\pi^{\otimes n(m_i)}(\mathbf{x},\mathbf{y})\right)^{1/2},
    \end{equation*}
    where  $n(m_i)$ is the degree of the form $m_i$, defining an objective invariant by $G = \bigcap_{m_i\in(m)} G_{m_i}$, and 
     all the results shown below can be generalized to this case, except the scaling Proposition \ref{proposition:scaling}. In the following example, we introduce a chiral distance invariant by $SO(d) = O(d)\cap SL(d)= G_{\langle\cdot,\cdot\rangle}\cap G_{\det}$.
\end{remark}

\begin{example}\label{example:costs_def}
    We define  the Inner Gromov-Wasserstein \emph{(IGW)}, Determinant Gromov-Wasserstein \emph{(DGW)} and Chiral Gromov-Wasserstein \emph{(CGW)} $\GW_m$ discrepancies, that are respectively associated with $G= O(d)$, $SL(d)$ and $SO(d)$; as

\begin{align}
\IGW(\mu,\nu) &= \inf_{\pi\in\Pi(\mu,\nu)}\left(\int_{(\X\times\Y)^2}|\langle x_1,x_2\rangle-\langle y_1,y_2\rangle|^2d\pi^{\otimes 2}\right)^{1/2}, \label{eq:IGW} \tag{$\IGW$}\\
\DGW(\mu,\nu)&=\inf_{\pi\in\Pi(\mu,\nu)}\left(\int_{(\X\times\Y)^d}|\det(\mathbf{x})-\det(\mathbf{y})|^2d\pi^{\otimes d}\right)^{1/2},\label{eq:DGW} \tag{$\DGW$}\\
\CGW_t(\mu,\nu)& =  \inf_{\pi\in\Pi(\mu,\nu)}\left(t\int_{(\X\times\Y)^2}|\langle x_1,x_2\rangle-\langle y_1,y_2\rangle|^2d\pi^{\otimes 2} \right.\nonumber \\ &\left.+ (1-t)\int_{(\X\times\Y)^d}|\det(\mathbf{x})-\det(\mathbf{y})|^2d\pi^{\otimes d}\right)^{1/2}.\label{eq:CGW} \tag{$\CGW$}
\end{align}
Here, the determinant is taken over the $d$ vectors of $\mathbf{x}=(x_1,\dots,x_d)$, with $x_i\in \R^d$, and $0<t<1$ is a fixed parameter.
\end{example}
\begin{lemma}\label{lemma:distance} 
 \ref{eq:GW_m}$:(\mu,\nu)\in 
 \P_2(\X)^2\to\R^+$ is a pseudo-distance.
\end{lemma}
Note that $\GW_m$ is not a distance on $\P_2(\X)$, since it does not separate points. For example $\GW_m(\mu,g_\sharp\mu)=0$ for $g\in G$ by construction. Under some assumptions, $\GW_m$ is actually a distance between $G$-shapes (see Corollary \ref{coro:dist_shape}).
By definition, $\GW_m$ is also not translation invariant, but we can ensure this by centering the marginals. More precisely,  for $z\in\R^d$, let us denote the shift map  $S_z:x\in\R^d\mapsto x+z$ and consider the following objective: 
\begin{equation}\label{eq:translation}
        \inf_{\underset{t,u\in\R^d}{\pi\in\Pi(\mu,\nu)}} C(\pi,t,u) = \inf_{\underset{\pi\in\Pi(S_{t\,\sharp}\mu ,S_{u\,\sharp}\nu)}{t,u\in\R^d}}\int(m(\mathbf{x})-m(\mathbf{y}))^2d\pi^{\otimes n}(\mathbf{x},\mathbf{y}) ,
\end{equation}
with $C(\pi,t,u)=\int[m(\mathbf{x}+\mathbf{i}_n(t))-m(\mathbf{y}+\mathbf{i}_n(u))]^2d\pi^{\otimes n}$ and $\mathbf{i}_n:x\in\R^d\mapsto (x,\dots,x)\in(\R^d)^n$.

\begin{proposition}\label{proposition:translation}
     Let  $\mu,\nu\in \P_2(\X)$ with respective mean $\bar{\mu}, \bar{\nu}$, and  $\pi$ an optimal solution to the problem \ref{eq:GW_m}$(S_{-\bar{\mu}\,\sharp}\mu ,S_{-\bar{\nu}\,\sharp}\nu)$ (i.e. with centered marginal). Then $( -\bar{\mu},-\bar{\nu}, \pi)$ is a solution of the right side problem (\ref{eq:translation}).
\end{proposition}

A consequence of Proposition \ref{proposition:translation}, is that centering the marginal is a natural way of defining a translation invariant metric, since it corresponds to a Procrustes distance with respect to the translations. Let's now consider the effect of marginal scaling on the \ref{eq:GW_m} problem. For $\alpha>0,$ we define the scaling-by-$\alpha$ function  $U_{\alpha}:x\in\R^d\mapsto\alpha x\in\R^d$. We also define the projections $p_\X:(x,y)\in\X\times\Y\mapsto x\in \X$ and similarly for $p_\Y$. The $\GW_m$ cost is not scaling invariant, but the optimal transport plan scales with the marginals in the following sense:  

\begin{proposition}\label{proposition:scaling}
    Let $\mu\in \P_2(\X),\nu\in \P_2(\Y)$ and $\pi^*$ be an optimal coupling for the $\GW_m(\mu,\nu)$ problem. A solution of the $\GW_m(\mu,U_{\alpha\,\sharp}\nu)$ problem is $(p_\X,U_{\alpha}\circ p_\Y)_{\sharp}\pi^*$. More generally, the pushforward by $(p_\X,U_{\alpha}\circ p_\Y)$ is a one to one correspondence between solutions of $\GW_m(\mu,\nu)$ and $\GW_m(\mu,U_{\alpha\,\sharp}\nu)$.
\end{proposition}
 This proposition has some practical implications, avoiding the need to rescale or normalize shapes when doing registration or matching. Indeed, once an optimal coupling between two distributions is known, then we can easily deduce an optimal coupling between rescaled shapes. For instance, in a discrete setting where the indexing of the space ($x_j$ vs. $\alpha x_j$) is the same, the optimal correspondence is identical, and the optimal objective value for any $\alpha$ can be evaluated with the invariant optimal correspondence. Note that this proposition does not apply to composite costs like \ref{eq:CGW}, as two multilinear forms with different degree scale differently.
We now provide sufficient conditions for $\GW_m$ to define a proper distance. We first introduce a condition on $m$, which we term non-degeneracy:

\begin{definition}
    Let $m$ be a multilinear $n$-form on $\R^d$.
    We say that $m$ is non-degenerate at slot  $i$ if, for all  $x\in \R^d\backslash \{0\},$ $ m(\dots,\underset{i}{x}, \dots)\neq0$ as an $n-1$ form.
    We say that $m$ is non-degenerate if there exists an index $i$ such that $m$ is non-degenerate at slot  $i$.
\end{definition}

\begin{remark}
    The inner product and the determinant are  non-degenerate multilinear forms. Indeed,  given $x\in \R^d\backslash \{0\}$, $\langle x,x\rangle = \|x\|^2>0$, and  we can complete $\{x\}$ to a basis $\{x, x_1,\dots,x_{d-1}\}$ of $\R^d$, such that $\det(x, x_1,\dots,x_{d-1})\neq0$.
\end{remark}

The non-degeneracy condition then allows us to characterize the mapss $T: \mathbb{R}^d \rightarrow \mathbb{R}^d$, such that $m$ is $T-$invariant. $T$ is in fact a linear invertible map:

\begin{lemma}\label{lem:degeneracy}
    Let $m$ be a non-degenerate $n$-form on $\R^d$ and $T:\R^d\to\R^d$. Then $m=m\circ T^{\times n}\implies T\in GL(d).$
\end{lemma}

\begin{proposition}\label{prop:GLD_transform}
    Let $m$ be a non-degenerate $n$-form on $\R^d$,  and $\mu,\nu\in \P_2(\R^d)$, We note $\mathrm{supp}(\mu)$ the support of $\mu$ and assume that there exists $(x_1,\dots,x_d)\in \mathrm{supp}(\mu)^d$ with $(x_1,\dots,x_d))$ basis of $\R^d$. Then we have the following equivalence: $ \GW_{m}(\mu,\nu)= 0 \Leftrightarrow \exists g\in G, \mu = g_\sharp\nu .$
\end{proposition}

The assumption on the support is necessary in the case $m=\det$, as $\DGW(\mu,\nu)=0$ if of the supports of $\mu$ and $\nu$ are both contained in hyperplanes of $\R^d$. 
Now we can prove that \ref{eq:GW_m} is actually a distance  (up to the action of $G$), using Lemma \ref{lemma:distance}, and Propositions \ref{proposition:translation} and \ref{prop:GLD_transform}. First we define $\P_2^{\mathrm{h}}(\R^d)$ be the set of elements of $\P_2(\R^d)$ whose support are not contained in an hyperplane of $\R^d$. Let $\P_2^{\mathrm{a}}(\R^d)$ be the set elements of $\P_2(\R^d)$ whose support are not contained in respectively an affine hyperplane.

\begin{corollary}\label{coro:dist_shape}
      If $m$ is a non-degenerate $n$-form on $\R^d$, then \ref{eq:GW_m} is a distance on $\P_2^{\mathrm{h}}(\R^d)/G_m$. 
      
      Moreover, the centered \ref{eq:IGW} is a distance on $\P_2^{\mathrm{a}}(\R^d)/E(d)$, the centered \ref{eq:DGW} is a distance on $\P_2^{\mathrm{a}}(\R^d)/( T_r(d)\rtimes SL(d))$ and the centered \ref{eq:CGW} is a distance on $\P_2^{\mathrm{a}}(\R^d)/SE(d)$.
    Here $\rtimes$ denotes the group semidirect product and $T_r(d)$ denots the group of translations in $\R^d$.
\end{corollary}

As a result of Corollary \ref{coro:dist_shape}, the \ref{eq:GW_m} distance can be used to compare shapes in practical applications, specifically when they are represented as point clouds. Note that the multilinear form $m$, and thus, the group $G$ can be appropriately chosen for a specific application, e.g. the \ref{eq:CGW} distance for distinguishing chirality between proteins and molecules. Next, we address the challenge of computing the shape metric, by revealing the low-dimensional structure of the \ref{eq:GW_m} optimization problem, that will then be leveraged to derive an efficient algorithm in Section \ref{sec:algo}.

\subsection{Characterization of the multilinear Gromov-Wasserstein problem with the second-moment matrix}\label{sec:low_dim_carac}
We are now interested in solving the multilinear Gromov-Wasserstein problem i.e., computing the distance $\GW_m(\mu,\nu)$ given $\mu,\nu$. As for the classical Gromov-Wasserstein distance, it is a non-convex polynomial programming problem, NP-hard in general. However, the $\GW_2$ problem admits a low-dimensional structure that can be leveraged to compute approximations of a global solution, as done in \cite{ryner2023globally}. Similarly, the next results reformulate the $\GW_m$ problem as  depending only on the projection of the couplings on a low-dimensional space, as represented in Figure \ref{fig:low_dimen}A. Given $\eta\in \P_2(\X)$, we define its second-moment matrix $\tilde{\Sigma}^\eta = \int_{\X}xx^Td\eta(x)$. Given $\xi\in \P_2(\X\times\Y)$  we define its cross second-moment matrix $\Sigma^\xi=\int_{\X\times \Y}xy^Td\xi(x,y)$.

\begin{proposition}\label{proposition:GWm_problem} 
Let $(\mu,\nu)\in\P_2(\X)\times\P_2(\Y)$ and $m$ be an $n$-form. The \ref{eq:GW_m} cost can be written as a multivariate polynomial only depending on the second-moment matrices $\tilde{\Sigma}^\mu,\tilde{\Sigma}^\nu$  of $\mu,\nu$, and the cross second-moment $\Sigma^{\pi}$ of $\pi$ as
\begin{equation*}
     \GW_m(\mu,\nu)^2= Q_m(\tilde{\Sigma}^\mu)+Q_m(\tilde{\Sigma}^\nu)+2\inf_{\pi\in \Pi(\mu,\nu)}-Q_m(\Sigma^{\pi}),
\end{equation*}
with $Q_m:\R^{d^2}\to \R$ a degree $n$ multi-variate polynomial with $d^2$ variables, depending on the expansion of $m$ on the canonical basis only. Computation of the polynomial $Q_m$ for specific forms $m$ can be found in Supplementary \ref{sec:supp-low_dim_carac}. Define the $W_2$-continuous map $\tilde{p}_0:\pi\in\P_2(\X\times\Y)\mapsto \Sigma^\pi\in\R^{d\times d}$  and $\tilde{P}_\Pi\coloneqq \tilde{p}_0(\Pi(\mu,\nu))$ a compact convex set. $\tilde{p}_0$ is the restriction to  $\P_2(\X\times\Y)$ of a linear map defined on signed measure. Then:
\begin{equation}\label{eq:low_dim_mini}
    \inf_{\pi\in \Pi(\mu,\nu)}-Q_m(\Sigma^{\pi}) = \inf_{x\in \tilde{P}_\Pi} -Q_m((x^{i,j})_{i,j}),
\end{equation}
with $x^{i,j}$ the coordinates of $x$ on the canonical basis of $\R^{d\times d}$.
\end{proposition}
We now assume that $\X,\Y$ are compact with non-empty interior. We define $f_{i,j}:(x,y)\in\X\times\Y\mapsto x^iy^j\in \R$, where $x^i$ is the $i$-th component of $x$, and $V=\mathrm{span}(f_{i,j})\subset \mathcal{C}_b(\X\times\Y) $ continuous bounded functions. Note that $\int f_{i,j}\;d\pi=\Sigma^\pi_{i,j}$ for $\pi\in\P(\X\times\Y)$. We equip $V$ with $\langle\cdot,\cdot\rangle$ the $L^2(\X\times\Y)$ inner product. Let's define $(e_i)$ an orthonormal basis of $V$. Then $p_0:\pi\in\P(\X\times\Y)\mapsto \sum_{i=1}^{d^2}e_i\int_{\X\times\Y}e_i\;d\pi \in V,$ is the natural extension from $L^2(\X\times\Y)$ to $\P(\X\times\Y)$ of the orthogonal projection onto $V$. In the compact case, we have an equivalent reformulation of Proposition \ref{proposition:GWm_problem}, which is useful for the numerical implementation of the problem:
\begin{corollary}\label{coro:reformulation_GW_m}
    For all $v\in V$ and $\pi\in \P(\X\times\Y)$, we have $\int v\;d\pi=\langle v,p_0(\pi)\rangle$. If $\mathcal{R}_f:V\to\R^{d^2}$ is the representation of the $(f_{i,j})_{i,j}$ in the $(e_{i,j})_{i,j}$ basis, then $\tilde{p}_0= \mathcal{R}_f\circ p_0$. Finally, we can compute the distance $\GW_m(\mu,\nu)$ by solving problem (\ref{eq:low_dim_mini}), which is equivalent to minimizing $Q_m\circ \mathcal{R}_f$ over the compact convex set $P_\Pi=p_0(\Pi(\mu,\nu))$.

    \begin{equation*}
    \inf_{\pi\in \Pi(\mu,\nu)}-Q_m(\Sigma^{\pi}) = \inf_{x\in P_\Pi} -Q_m\circ\mathcal{R}_f(x).
\end{equation*}
    
\end{corollary}
In practical applications, the spaces $\X,\Y$ are compact. In Section \ref{sec:algo} which covers algorithms to solve the $\GW_m$ problem, we will identify $Q_m$ and $Q_m\circ\mathcal{R}_f$, since they only differ by the vector representation in two bases. The inner product used on $V$ is the one described in Corollary \ref{coro:reformulation_GW_m}.
\begin{example}\label{ex:3_csots_poly}
    For the \ref{eq:IGW} problem, $Q_m= \|\cdot\|^2_F$ the squared Frobenius norm, for the \ref{eq:DGW} problem, $Q_m= d!\det$ and for the \ref{eq:CGW} problem with parameter $t$, $Q_m= t\|\cdot\|^2_F+ (1-t)d!\det$.
\end{example}

\begin{remark}
    A similar low-dimensional description holds for the classical 2-Gromov-Wasserstein problem (with squared Euclidean distance), with a more complex term that does not depend on the coupling (see Supplementary \ref{sec:classical_gw}). This fact is used in \cite{ryner2023globally} to design an algorithm that approximates the 2-$\GW$ distance between measures. In this case, when $\mu\in \P_4(\X), \nu\in \P_4(\Y)$ the polynomial $-Q_m$ to be optimized is concave and depends on the coordinate of vectors of $V$ a vector space of dimension $d^2+1$. We have $V=\mathrm{span}(g, (f_{i,j})_{i,j})$, $g:x,y\in \X\times\Y\mapsto \|x\|^2_2\|y\|^2_2$ and $f_{i,j}:x,y\in \X\times\Y\mapsto x^iy^j$. The analog of the projection $p_0$ is still continuous and we can solve the optimization problem on $P_\Pi$, a compact convex subset of dimension $d^2+1$. All the results of this paper, except Lemma \ref{lem:Monge_DGW}, have variants applying to the classical 2-Gromov-Wasserstein problem. For example the 2-Gromov-Wasserstein is a distance on $\P_4(\R^d)/E(d)$ \cite[Theorem 5.1]{memoli2011gromov}.
\end{remark}
\begin{figure}
    \centering
    \includegraphics[width=\linewidth]{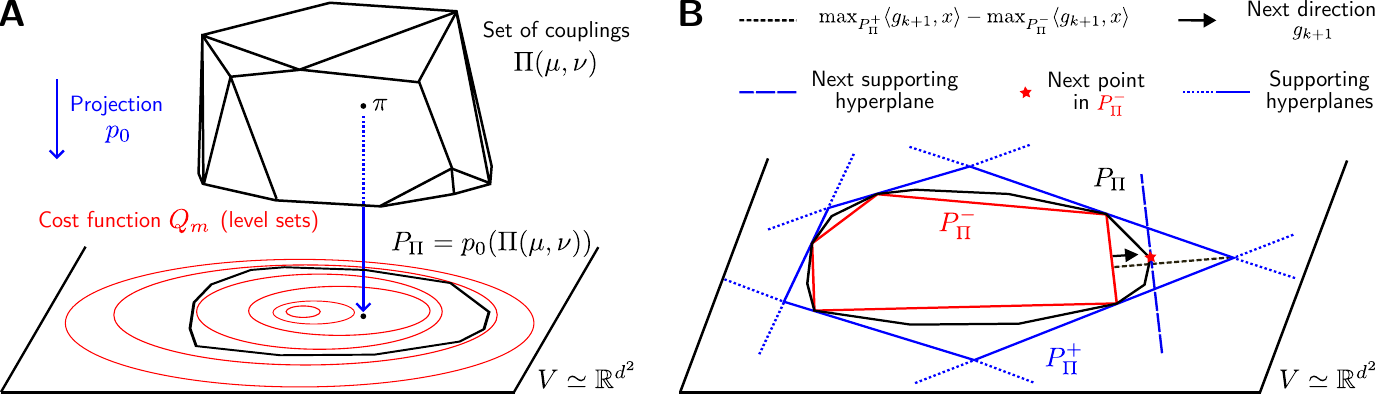}
    \caption{The low-dimensional structure of the \ref{eq:GW_m} problem and how it is approximately solved using a sandwich algorithm. A) Low-dimensional structure of the $\GW_m$ problem described in Corollary \ref{coro:reformulation_GW_m}. Given a coupling $\pi\in\Pi(\mu,\nu)$, the optimization problem for computing $\GW_m(\mu,\nu)$ only depends on the projection of $\pi$ on a $d^2$ dimensional space. We reformulate the computation of $\GW_m(\mu,\nu)$, as the global optimization over a compact convex set $P_\Pi$, of a  multi-variable polynomial function $Q_m$. B) Representation of Kamenev's sandwich algorithm \cite{kamenev1996algorithm,lammel2023convergence} to geometrically approximate $P_\Pi$, as described in Section \ref{sec:bounding_boxes}. The algorithm recursively updates two bounding boxes $P_\Pi^+$ and $P_\Pi^-$, which are converging to $P_\Pi$. At each iteration, the direction (black arrow) maximizing the gap between the two boxes (black dashed line) is selected. Then a new supporting hyperplane (blue dotted line) and vertex (red star) are computed to update the boxes.}
    \label{fig:low_dimen}
\end{figure}

\subsection{Properties of the optimal coupling}\label{sec:opt_coupl}

Now we state existence and stability results for the  $\GW_m$ cost. We are considering the $W_{2}$ distance between couplings.
To wit, recall that when $\mu\in\mathcal{P}(\mathcal{X})$ and $\nu\in\mathcal{P}(\mathcal{Y})$,
a coupling $\pi\in\Pi(\mu,\nu)$ is itself an element of $\mathcal{P}(\mathcal{X}\times\mathcal{Y})$;
and when both $\mu$ and $\nu$ have finite second moments, so does
$\pi$. Assuming $\mathcal{X}$ and $\mathcal{Y}$ are each
endowed with metrics $D_{\mathcal{X}}$ and $D_{\mathcal{Y}}$, by
convention we equip the product space $\mathcal{X}\times\mathcal{Y}$
with the metric $D_{\mathcal{X}}\oplus D_{\mathcal{Y}}$, defined
by 
\[
D_{\mathcal{X}}\oplus D_{\mathcal{Y}}((x,y),(x^{\prime},y^{\prime})):=\sqrt{D_{\mathcal{X}}^{2}(x,x^{\prime})+D_{\mathcal{Y}}^{2}(y,y^{\prime})}.
\]
Accordingly, the $W_{2}$ metric on $\mathcal{P}_{2}(\mathcal{X}\times\mathcal{Y})$
is defined using $D_{\mathcal{X}}\oplus D_{\mathcal{Y}}$. 
In other words, given $\pi,\pi^{\prime}\in\mathcal{P}_{2}(\mathcal{X}\times\mathcal{Y})$,
\[
W_{2}^{2}(\pi,\pi^{\prime}):=\inf_{\Theta\in\Pi(\pi,\pi^{\prime})}\int_{(\mathcal{X}\times\mathcal{Y})^{2}}\left(D_{\mathcal{X}}^{2}(x,x^{\prime})+D_{\mathcal{Y}}^{2}(y,y^{\prime})\right)d\Theta((x,y),(x^{\prime},y^{\prime})).
\] 

The following proposition establishes existence, and  qualitative stability, of  minimizers of the multilinear
Gromov-Wasserstein problem with respect to $W_2$ perturbations of the marginals. The proof is a straightforward application of $\Gamma$-convergence techniques (see \cite{braides2002gamma}), but we
could not find a suitable reference: in particular, the corresponding
proof of stability for the $p$-Gromov-Wasserstein problem in the
original paper of Memoli \cite{memoli2011gromov} is less direct and
does not apply in our case.
\begin{proposition}
\label{prop:stability-general}
Let $\X,\Y\subset \R^d$ and $Q:\R^{d^2}\to\R$ be a polynomial with $d^2$ variables, $\mu\in\mathcal{P}_2(\X)$ and $\nu\in\mathcal{P}_2(\Y)$, and consider $(\mu_k)$ and $(\nu_k)$  respectively converging to $\mu$ and $\nu$, such that 
$W_2 (\mu_{k},\mu)\rightarrow 0$ and $W_2(\nu_{k},\nu)\rightarrow 0$. Denote $\mathcal{C}(\pi):=Q(\Sigma^\pi)$; then $\inf_{\pi \in \Pi(\mu,\nu)} \mathcal{C}(\pi)$ is attained, and $\min_{\pi_k \in \Pi(\mu_k,\nu_k)} \mathcal{C}(\pi_k)\rightarrow \min_{\pi \in \Pi(\mu,\nu)} \mathcal{C}(\pi)$
as $k\rightarrow\infty$. 

Furthermore, letting $\pi_{k}^{*}\in\argmin_{\pi\in\Pi(\mu_{k},\nu_{k})}\mathcal{C}(\pi_k),$ it holds that $\pi_{k}^{*}$ is precompact in $(\mathcal{P}_2(\X\times \Y),W_2)$,
and any limit point $\pi^{*}$ is contained in $\argmin_{\pi\in\Pi(\mu,\nu)}\mathcal{C}(\pi)$.
\end{proposition}
We now quantify the stability of the $\GW_m$ distance with respect to the marginals.
\begin{theorem}\label{theo:sample_complexity}
 Let $\mu,\mu_0\in\mathcal{P}_2(\X)$ and $\nu,\nu_0\in\mathcal{P}_2(\Y)$, and assume that their  second moment $\mathfrak{m}_2(\mu),$ $\mathfrak{m}_2(\mu_0),\mathfrak{m}_2(\nu),\mathfrak{m}_2(\nu_0)\leq \alpha$ for $\alpha\geq1$. If $n$ is the degree of the multilinear form $m$, then we have:
    \begin{equation}\label{eq:GW_conver_W_2}
        |\GW_m(\mu_0,\nu_0)^2-\GW_m(\mu,\nu)^2|\leq (2+4\sqrt{2})n c_1 \alpha^{n-1/2} (W_2 (\mu_{0},\mu)+W_2(\nu_{0},\nu)),
    \end{equation} 
with $c_1$ the $\ell^1$ norm of the coefficients of the polynomial $Q_m$.
\end{theorem}
This theorem can be used in practical cases to approximate the $\GW_m$ distance given a specific tolerance, for example when a downsampling of the marginals is needed. Typically, if $\mu$ and $\nu$ are discrete marginals with a large number of points, one can approximate the marginals with fewer points, then solve the \ref{eq:GW_m} problem on the approximation, and estimate the resulting error in cost.
\begin{remark} Note that, when $\mu,\nu$ have finite $q$ moments for $q$ large enough ($q>4$ if $d\leq 4$ else $q>2d/(d-2)$), we can apply Theorem \ref{theo:sample_complexity} to classical $W_2$ convergence rates of the empirical distributions  with $k$ elements $\hat{\mu}_k,\hat{\nu}_k$ \cite{fournier2015rate} to get 
\[\mathbb{E}[|\GW_m(\hat{\mu}_k,\hat{\nu}_k)^2-\GW_m(\mu,\nu)^2|]\leq c \alpha^{n-1/2} \phi_k,\]
with $\phi_k = k^{-\frac{1}{\max(d,4)}}\ln(1+k)^{\frac{1}{2}\1_{\{d=4\}}}$. Here the constant $c$ only depends on the $q$-th moment of the marginals, the multilinear form $m$ and the dimension $d$. Similar result holds with different rates under weaker assumptions on $q$. Our rate is worse than the ones derived for the $\GW_2$ distance \cite[Theorem 4.2]{zhang2024gromov} and \cite[Theorem 3.1]{kato2025convergence}, which are leveraging the existence of a variational formulation for the concave $\GW_2$ cost \cite[Corollary 4.1]{zhang2024gromov}. However, our results apply to more general costs, which may be non-concave (they are never convex) and may not possess such variational formulation. For example the \ref{eq:DGW} cost is not concave.
\end{remark}

\begin{definition}
    Given an optimal transport problem, we call a \emph{Monge map}, a measurable map $T:\X\to\Y$ such that the coupling $\pi=(\mathrm{id},T)_\sharp\mu$ is  a solution of the problem. In this case,  we call $\pi$ a Monge plan. The marginal constraints on $\pi$ imply $\nu=T_\sharp\mu$, and $T$ represents how to move the mass of $\mu$ to $\nu$ in an optimal way.
\end{definition}

The existence of Monge maps has been studied in different settings, but is not guaranteed in full generality: characterizing this existence for the squared Euclidean Gromov-Wasserstein distance is still an open problem \cite{dumont2025existence}. Its existence for the $W_2$ problem, when the first marginal $\mu$ has a density has been proved \cite{brenier1987decomposition}, and it is a fundamental result in optimal transport theory. 
We now extend the existence of Monge map to $\GW_m$ problems by adapting an argument from \cite[Theorem 3.2]{dumont2025existence}.
\begin{lemma}[Existence of a Monge map for $\GW_m$] \label{lem:Monge_DGW}
    Suppose that $\mu,\nu \in \P(\R^d)$ have compact support and that $\mu\ll \Leb$, i.e. $\mu$ has a density, and that -$Q_m$ is concave. Then there exists an optimal correspondence map $T:\R^d\to\R^d$, such that $(id,T)_\sharp\mu$ is an  optimal transport plan for the $\GW_{m}$ problem. 
\end{lemma}
\begin{lemma}[Optimal transport plans for \ref{eq:DGW} are Monge plans]\label{lem:Monge_DGW_2}
    Suppose that $\mu,\nu \in \P(\R^d)$ have compact support, are centered and that $\mu,\nu\ll \Leb$. For the  \ref{eq:DGW} problem, any optimal transport plan with marginals $\mu,\nu$ is a Monge plan.
\end{lemma}
Proposition \ref{prop:stability-general} and Theorem \ref{theo:sample_complexity} show the stability of the $\GW_m$ cost and optimal coupling, when the marginals $\mu,\nu$ are approximated. In particular, this applies when the marginals have a density and  are approximated by a discrete measure (in numerical applications for example): The discrete cost converges towards the continuous cost, at a controlled rate. The existence of a Monge map establishes an optimal correspondence between marginals in the continuous setting. When numerically approximated, such a map enables shape alignment and registration. Note that, for general marginals  $\mu,\nu$, the optimal plans may not be maps, and that the convergence of the optimal plans happens up to a subsequence, since the $\GW_m$ may have several solutions.

\section{Solving multilinear Gromov-Wasserstein problems in low-dimension}\label{sec:algo}
We present in this section an algorithm that approximates the global solution of the \ref{eq:GW_m} problem, leveraging its low-dimensional structure (Corollary \ref{coro:reformulation_GW_m}).  We consider the case where the marginals $\mu,\nu$ have finite support and we assume that their dimension $d$ is low, typically 2 or 3.   Proofs for this section and additional algorithms can be found in Supplementary \ref{sec:supp_algo}. Our algorithm is structured as follows: 

(i) [Section \ref{sec:bounding_boxes}] We approximate the projection $P_\Pi$ of the marginal constraint space $\Pi(\mu,\nu)$ on the low-dimensional space $V$. This is done with a sequential refinement of outer and inner polytopes for $P_\Pi$, called bounding boxes. They converge to $P_\Pi$ at a known rate (see Theorem \ref{theo:convergence_rate}).

(ii) [Section \ref{sec:opti_cost}] Once bounding boxes have converged to a given accuracy, we solve the polynomial optimization problem $\inf -Q_m(x)$ on the bounding boxes. We obtain two values that are lower and upper bounds for the  optimal cost on $P_\Pi $.  They can also be computed at each iteration and be used as a certificate of global optimality. 

(iii) [Section \ref{sec:algo_loc}] In order to converge to a local optimum of our problem, we propose a reformulation of the Frank-Wolfe scheme (Algorithm \ref{alg:Frank_Wolfe}) that also leverages the low-dimensional structure of the problem. This local optimization is initialized using the approximation of a global minimum computed in step (ii).

In Section \ref{sec:algo_glob_concave}, we study the case of a concave cost. This enables us to perform the optimization step (ii) efficiently using the polytope representation of the bounding boxes, and to compute the optimal coupling as a solution to an optimal transport problem. When the cost is concave, our approximation algorithm is tractable in the data size (see Theorem \ref{theo:FPTAS}) for a given accuracy. We also show that the complexity arising from the data size only corresponds to solving classical optimal transport problems.

Our algorithm shares similarities with a recent approach for globally solving the 2-Gromov-Wasserstein problem \cite{ryner2023globally}, and our algorithm can also solve this classical GW problem. One major difference in our approach resides in the choice of direction of the cuts for updating the bounding boxes. In our case, the directions derive from solving an optimization problem with respect to both inner and outer bounding boxes, while in \cite{ryner2023globally} they were computed from the gradient of the functional at an optimum of the outer bounding box. Note that this allows us to derive a convergence rate for our concave algorithm (see Theorem \ref{theo:convergence_rate}), whereas only the convergence of the algorithm could previously be obtained. In addition, our algorithm applies in a more general setting to weighted point clouds, and includes a local optimization step that is key for the convergence of the transport plan (see Section \ref{sec:algo_loc}). Numerical comparison of Ryner and colleague's  algorithm \cite{ryner2023globally} and ours can be found in Figure \ref{fig:bench_ryner}.

\subsection{A Sandwich algorithm to approximate polytopes}\label{sec:bounding_boxes}
We now present a method to approximate $P_\Pi$.  When the marginals are discretized, a direct computation of $P_\Pi$ naturally suffers combinatorial explosion, because it would require projecting all extremal points of $\Pi(\mu,\nu)$ onto $V$. Instead, we are iteratively approximating $P_\Pi$ with supporting hyperplanes by applying Kamenev's sandwich algorithm \cite{kamenev1996algorithm,lammel2023convergence}, as represented in Figure \ref{fig:low_dimen}B. This algorithm iteratively creates both a decreasing outer bounding box $P_{\Pi}^+$ and an increasing inner bounding box $P_{\Pi}^-$. More precisely, at iteration $k$, $P_{\Pi,k}^-\subset P_\Pi\subset P_{\Pi,k}^+$ and $(P_{\Pi}^\pm)_k$ are monotone families of convex sets converging to $P_\Pi$ as $k\to+\infty$. The bounding boxes are built as follows:
\begin{definition}[Bounding boxes]\label{def:bounding_box}
Let $P_\Pi$ be a polytope $\subset V$ with boundary $\partial P_\Pi$. Then $P_{\Pi}^-$ is an inner bounding box of $P_\Pi$ if it is the convex hull of points in $\partial P_\Pi$ and $P_{\Pi}^+$ is an outer bounding box of $P_\Pi$ if it is the intersection of supporting half spaces of $P_\Pi$.
\end{definition}
In our algorithm, we are iteratively appending points $g^*\coloneqq\argmax_{h\in P_\Pi} \langle h,g\rangle\in\partial P_\Pi$ to $P_{\Pi}^-$, for given directions $g$. Similarly, we compute $ \hat{g}=\max_{h\in P_\Pi}\langle h,g\rangle$ and update the outer bounding box as follows: $P_{\Pi}^+\gets P_{\Pi}^+\cap H(g)$, with $H(g) \coloneqq \{h\in V, \langle h,g\rangle\leq \hat{g}\},$ the supporting hyperplane of $P_\Pi$ in direction $g$.

\begin{definition}[Representation of a polytope]\label{def:polytope}
    A polytope $P\subset\R^l$ can be equivalently represented by its vertices or its facets. We define the $V-$representation of $P$ as the convex hull of the set of its vertices, i.e. $P=\mathrm{conv}(\{x \text{ vertex of P}\})$. We define the $H-$representation of $P$ as the intersection of halfspaces: $P=\bigcap_{g \in F(P)}H(g)$ where $F(P)$ is the set of $P$-facets unit normal vectors pointing outwards. Each facet is a linear constraint for the feasible set $P$. The $H-$representation of a polytope with $s$ facets can be written as $\{x\in \R^l, Ax\leq b\}$, with $A$ a $s\times l$ matrix and $b$ a vector of size $s$, where each row represents a constraint.
\end{definition}

At each iteration $k$ of the algorithm, a unit vector $g_k$ is chosen (see below). We then solve the linear program: 
\begin{equation} \label{eq:lin_problem}
\hat{g}_k, g^*_k = \max, \argmax_{h\in P_\Pi}\langle h,g_k\rangle, \qquad \textrm{with} \quad \max_{h\in P_\Pi} \langle h,g_k\rangle =\max_{\pi\in\Pi(\mu,\nu)}\int_{\X\times\Y} g_k(x,y)d\pi(x,y),
\end{equation}
and the bounding boxes are updated by adding a vertex to $P_\Pi^-$ and a constraint to $P_\Pi^+$, as described in Definition \ref{def:bounding_box} and below.
Note that problem (\ref{eq:lin_problem}) is an optimal transport problem with cost $g$ as a result of Corollary \ref{coro:reformulation_GW_m}. This problem can be solved or efficiently approximated using various methods \cite{cuturi2013sinkhorn,peyre2019computational,flamary2021pot,flamary2024pot,cuturi2022optimal}. To choose $g_{k+1}$, we follow Kamenev's strategy \cite{lammel2023convergence,kamenev1996algorithm} and maximize the gap between $P_{\Pi,k}^-$ and $P_{\Pi,k}^+$, as:
\begin{equation}\label{eq:Hausdorff_direc}
    c,g_{k+1} = \max,\argmax_{g\in F(P_{\Pi,k}^-)} [\max_{h\in P_{\Pi,k}^+}\langle h,g\rangle - \max_{h\in P_{\Pi,k}^-}\langle h,g\rangle].
\end{equation}
The value $c$ represents this maximal gap, or a distance between the bounding boxes \cite{lammel2023convergence} (black dashed line in Figure \ref{fig:low_dimen}B). Note that $g_{k+1}$ is an element of $ F(P_{\Pi,k}^-)$ (see Definition \ref{def:polytope}), and the value $\max_{h\in P_{\Pi,k}^-}\langle h,g\rangle$ is known in the $H$-representation of $P_{\Pi,k}^-$. We know the $H$-representation of $P_{\Pi,k}^+$ and the $V$-representation of $P_{\Pi,k}^-$ by construction as we append to them the new face and vertex computed at each iteration, hence they have size $O(k)$. Since a solution of $\argmax_{h\in P_{\Pi,k}^+}\langle h,g\rangle$ is a vertex of $P_{\Pi,k}^+$, problem \eqref{eq:Hausdorff_direc} can be efficiently solved using the $V-$representation of $P_{\Pi,k}^+$ and the $H-$representation of $P_{\Pi,k}^-$. Thus, the algorithm requires computing the $H$-representation of $P_{\Pi,k}^-$ and the $V$-representation of $P_{\Pi,k}^+$, knowing the other representation in each case.  By the Upper Bound Theorem \cite[Section 5.5]{matouvsek2002lectures}, the newly computed representations have size at most $O(k^{\lfloor d^2/2\rfloor})$, so we can control the size of the bounding boxes. The strategy to approximate $P_\Pi$ is summarized in Algorithm \ref{alg:main_GW_m}, where we initialize the bounding boxes using Algorithm \ref{alg:bounding_box_cvx_2} (the impact of the initialization on convergence will be discussed next). In Figure \ref{fig:low_dimen}B we also give a schematic of how the bounding boxes get updated upon selecting the new direction $g_{k+1}$.

\begin{algorithm}
\caption{Approximation of $P_\Pi$ \cite{kamenev1996algorithm}}\label{alg:main_GW_m} 
\begin{algorithmic}[1]
\Require  $\epsilon>0,\mu,\nu$ probability measure with finite support.
\State Define the gap between the bounding boxes $c=+\infty$
\State \textbf{Initialize}: Create bounding boxes $P_\Pi^-,P_\Pi^+$ such that $P_\Pi^+$ is bounded.
\While{$c>\epsilon$}

    \State Solve \eqref{eq:Hausdorff_direc} and compute $c,g$
    \State Solve \eqref{eq:lin_problem} with direction $g$ and compute $\hat{g},g^*$
    \State $P_\Pi^-\gets \mathrm{conv} (P_\Pi^-\cup g^*)$  
    \State $P_\Pi^+\gets P_\Pi^+\cap H(g)$ 
\EndWhile

\State \Return $P_\Pi^-,P_\Pi^+$

\end{algorithmic}
\end{algorithm}
We now study the convergence rate of Algorithm \ref{alg:main_GW_m}. We first introduce the Hausdorff distance between two polytopes $P,Q\subset\R^{d^2}$: 
\begin{equation}
    \label{eq:hauss_distance}
    \mathcal{H}_2(P,Q)= \max\left(\sup_{x\in Q}\inf_{y\in P} \|x-y\|,\ \sup_{y\in P}\inf_{x\in Q} \|x-y\|\right).
\end{equation}
Our convergence rate estimation relies on the fact that we can control the volume of the convex set $P_\Pi$, assuming that the marginals second-moment matrices $\tilde{\Sigma}^\mu$ and $\tilde{\Sigma}^\nu$ are positive definite, with the following notations
\begin{assumption}\label{assump_full_dim}
    We assume $\mu$ and $\nu$ to be compactly supported with mean $ \bar{\mu}, \bar{\nu}$. We assume $\tilde{\Sigma}^\mu- \bar{\mu}\bar{\mu}^T$ and $\tilde{\Sigma}^\nu-\bar{\nu}\bar{\nu}^T$ to be definite positive, with $R, \lambda >0$ such that $\mu,\nu\in\P(\mathcal{B}^{\R^d}_0(R))$ and $\lambda\leq\min(\mathrm{Spec}(\tilde{\Sigma}^\mu- \bar{\mu}\bar{\mu}^T)\cup \mathrm{Spec}(\tilde{\Sigma}^\nu)- \bar{\nu}\bar{\nu}^T)$. 
\end{assumption}
Under this assumption, the discrepancy defined in \eqref{eq:Hausdorff_direc} is actually strongly equivalent to $\mathcal{H}_2$ (from  \cite[Lemma 4.4]{lammel2023convergence} and Lemma \ref{lemma:P_Pi_volume}). Several initializations of the bounding boxes for Algorithm \ref{alg:main_GW_m} are possible, for example using a $d^2$ dimensional hyper-rectangle  or a $d^2$ simplex. However, for  Theorem \ref{theo:convergence_rate} to hold, we need to create an initial bounding box $P_\Pi^-$ with a controlled volume and asphericity (more details in Supplementary \ref{sec:supp_algo_conv}). Algorithm \ref{alg:bounding_box} creates such a box.

\begin{theorem}
    \label{theo:convergence_rate}
    Let $\epsilon>0$. Under Assumption \ref{assump_full_dim}, suppose that we initialize Algorithm \ref{alg:main_GW_m} with a bounding box as described in Algorithm \ref{alg:bounding_box}. Then there exists a constant $c_0$ depending on $d$, $R$ and $\lambda$ only, such that:
         \[\mathcal{H}_2(P_{\Pi,k}^+,P_{\Pi,k}^-)\leq c_0k^{-1/(d^2-1)},\;\mathrm{for}\; k\geq d^2+1.\]
    In other words the bounding boxes $P_{\Pi,k}^+$ and $P_{\Pi,k}^-$ converge to the set $P_\Pi$ at a rate $\frac{1}{d^2-1}$.
\end{theorem}

\subsection{Optimization of the cost}\label{sec:opti_cost}
As Corollary \ref{coro:reformulation_GW_m} shows, solving the $\GW_m$ problem is equivalent to  minimizing the polynomial $-Q_m$ on $P_\Pi$, a $d^2$ dimensional compact convex set. To approximate this solution, we minimize $-Q_m$ on the inner bounding box $P_\Pi^-$ as:
\begin{equation} \label{eq:non_conv_pol_opt}
    C^*,x^*=\min,\argmin_{x\in P_{\Pi}^-}-Q_m((x^{i,j})_{i,j}).
\end{equation}
Note that $-Q_m$ cannot be convex. For a general cost $-Q_m$, the low-dimensional problem \eqref{eq:non_conv_pol_opt} can be approximated and solved using standard non-convex algorithms such as simulated annealing or branch and bound. Here, as $Q_m$ is polynomial, the global solution on $P_\Pi^-$ can be approximated using the Lasserre hierarchy leading to a semidefinite programming (SDP) relaxation \cite{lasserre2001global}. The size of the variable for the SDP problems grows polynomially in the relaxation level, and solving the SDP problem faces scalability issues \cite{majumdar2020recent}. Upon solving the optimization problem (\ref{eq:non_conv_pol_opt}), an optimal transport plan $\pi^*$ can be derived from the solution $x^*\in V$, by solving the following convex problem (with standard convex optimization techniques \cite{boyd2004convex}):
\begin{equation}\label{eq:final_coupling}
    \pi^*=\argmin_{\pi\in \Pi(\mu,\nu)} \|x^*-p_0(\pi)\|^2.
\end{equation}
Note that $\pi^*$ yields an approximation of the true optimal coupling for the $\GW_m$ problem, since we approximated $P_\Pi$ with $P_\Pi^-$. To refine the cost and the optimal coupling, we run a local optimization algorithm, such as Frank-Wolfe, with initial value $\pi^*$ (see section \ref{sec:algo_loc}).

\subsection{Global optimization for a concave cost}\label{sec:algo_glob_concave}

In the specific case where  $-Q_m$ is concave, we can solve the $\GW_m$ problem by leveraging the fact that a solution is located at a vertex of the constraint polytope $P_{\Pi}^-$. This is the case for the \ref{eq:IGW} cost in all dimensions and for the \ref{eq:CGW} cost with parameter $1/2\leq t<1$ in dimension $d=2$. In dimension  $d=3$ the concavity of the $\CGW$ cost depends on the size of the marginals $\mu,\nu$. If  $\mathrm{supp}(\mu)\subset \B_0^{\R^d}(R)$ and $\mathrm{supp}(\nu)\subset \B_0^{\R^d}(R)$ for $R>0$, we can choose  $t \in  \left[\frac{2\sqrt{3}R^2}{1+2\sqrt{3}R^2},1\right)$ to ensure concavity (see Supplementary \ref{sec:supp_cost_convexity}). Then, approximating $\GW_m$ can be done by computing the cost $-Q_m$ at all vertices of  $P_{\Pi,k}^-$ using its $V-$representation, to solve  \eqref{eq:non_conv_pol_opt}. By construction the number of vertices of $P_{\Pi,k}^-$ is the number of iteration $k$ plus the number of initial vertices (at most $2d^2$ when initialized with Algorithm \ref{alg:bounding_box}). In the concave case, problem (\ref{eq:final_coupling}) can be reformulated as a classical optimal transport problem:
\begin{equation}\label{eq:final_coupling_cvx}
    \pi^*= \argmax_{\pi\in\Pi(\mu,\nu)}\int \bar{g}(x,y)d\pi(x,y),
\end{equation} where $\bar{g}\in V$ is such that the supporting hyperplane of $P_\Pi$ with normal $\bar{g}$ contains $x^-$, with $x^-\coloneqq \argmin_{x\in P_{\Pi,k}^-}-Q_m(x)$. At each iteration, if the vertex $x_0$ added to $P_{\Pi}^-$ is optimal when searching in direction $g_0$, we update $\bar{g} \gets g_0 $ and $x^-\gets x_0$. Even if problem (\ref{eq:final_coupling_cvx}) may have several solutions, using a deterministic algorithm to solve it ensures $x^-=p_0(\pi^*)$.  Algorithm \ref{alg:main_GW_m_concave} represents how to approximate the $\GW_m$ solution by adapting Algorithm \ref{alg:main_GW_m} when the cost is concave. We can now estimate the complexity of this algorithm:

\begin{algorithm}
\caption{Main algorithm for $\GW_m$ in dimension $d$, concave cost}\label{alg:main_GW_m_concave}
\begin{algorithmic}[1]
\Require $\epsilon>0, \mu,\nu$.
\State Compute second-moment matrix $\tilde{\Sigma}^{\mu}$, $\tilde{\Sigma}^{\nu}$
\State Define:$f_{i,j}\coloneqq (x,y)\in \R^{d}\times\R^{d}\mapsto x^iy^j\in \R$
\State Define bounding boxes $P_\Pi^+\coloneqq\{\}$,$P_\Pi^-\coloneqq\{\}$, box cost $c^- = +\infty$, $c^+ = -\infty$ and optimal direction $\bar{g}=\emptyset$
 \State  \textbf{Initialization of the bounding boxes: }  $P_\Pi^-,P_\Pi^+,c^-, \bar{g}$ (Algorithm \ref{alg:bounding_box_cvx})
\State $c^+,x^+\gets \min,\argmin_{P_\Pi^+} (-Q_m) $ done by enumeration

 \State  \textbf{Updating bounding box until convergence }

\While{$|c^+ -c^-|>\epsilon$}
    \State Solve problem \eqref{eq:Hausdorff_direc}  and compute optimal direction $g$ \label{alg_cv_dir}
    \State Solve problem \eqref{eq:lin_problem} and compute $\hat{g},g^*$ \label{alg_cv_lin_problem}
    \State $P_\Pi^-\gets \mathrm{conv} (P_\Pi^-\cup g^*)$  
    \State $P_\Pi^+\gets P_\Pi^+\cap H(g)$ 
    \State $c^+,x^+\gets \min,\argmin_{P_\Pi^+} (-Q_m) $ done by enumeration
    \If{$c^->-Q_m(g^*)$}
    \State $\bar{g}\gets g$
    \State $c^-\gets -Q_m(g^*)$
    \EndIf
\EndWhile

\State  \textbf{Optimal coupling} $\pi^*$ solving convex problem \eqref{eq:final_coupling_cvx} given $\bar{g}$
\State (Optional) $\pi^*, c^-$ solving $\text{Frank-Wolfe}(m,\mu,\nu,\pi^*)$
\State $\GW^2\gets Q_m(\tilde{\Sigma}^{\mu})+Q_m(\tilde{\Sigma}^{\nu})+2c^-$
\State\Return $\pi^*,\GW^2$
\end{algorithmic}
\end{algorithm}

\begin{theorem}\label{theo:FPTAS}
     Assume that $\mu,\nu$ are discrete, with at most $N$ points and that $-Q_m$ is concave on the initial bounding box $P_\Pi^+$ of Algorithm \ref{alg:bounding_box_cvx}. Under Assumption \ref{assump_full_dim}, Algorithm \ref{alg:main_GW_m_concave} is an approximation scheme for the \ref{eq:GW_m} problem whose running time is polynomial in $N$ and $1/\epsilon$ for fixed $d,R,\lambda$. More precisely, Algorithm \ref{alg:main_GW_m_concave} enables one to approximate $\GW_m(\mu,\nu)^2$ at additive error $\epsilon>0$:
    
    - with complexity $O\left(N^3\ln(N)\left(\frac{1}{\epsilon}\right)^{(d^2-1)}+\left(\frac{1}{\epsilon}\right)^{2(d^2-1)(\lfloor d^2/2\rfloor+1)}\right)$,

    - with  memory of size $O\left(N^2+\left(\frac{1}{\epsilon}\right)^{(d^2-1)(\lfloor d^2/2\rfloor+1)}\right)$,\\
    where the implied constants only depend on $d$,  $m$ and $R$, $\lambda$ specified in Assumption \ref{assump_full_dim}. In particular, they do not depend on the marginals $\mu,\nu$, as long as Assumption \ref{assump_full_dim} holds.
\end{theorem}  

Our evaluation of complexity Theorem \ref{theo:FPTAS} first results from estimating  the bounding boxes size and applying Theorem \ref{theo:convergence_rate}, which states that the number of iterations needed to compute an $\epsilon$-approximation of the $\GW_m$ cost only depends on $\epsilon$ and is independent of the sample size $N$. At iteration number $k$, the data needed to represent and efficiently update the bounding boxes has size $O(k^{\lfloor d^2/2\rfloor+1})$ and updating it has cost and memory size at most  $O(k^{2\lfloor d^2/2\rfloor+1})$.  The complexity in $N$ arises from solving an optimal transport problem, which happens once per iteration and  has complexity $O(N^3\ln(N))$\cite{peyre2019computational}.  Note that we can alternatively compute an $\epsilon$-approximation (in cost) of this problem using entropic regularization and the Sinkhorn algorithm, with complexity $O(N^2\ln(N)\epsilon^{-3})$ \cite[Remark 4.6]{peyre2019computational}. This has a significant practical advantage of enabling scaling the algorithm to large data size. But we note that in this case, the convergence rate is not known, as the noise due to Sinkhorn could impact the convergence of our algorithm. Interestingly, Algorithm \ref{alg:main_GW_m_concave} can be adapted to the 2-Gromov-Wasserstein cost (see Algorithm \ref{alg:classical_GW_2}), with convergence rate described in Theorem \ref{theo:FPTAS_GW_2}: Given $\mu,\nu$ discrete measures with bounded support we can approximate $\GW_2(\mu,\nu)^2$ at an accuracy $\epsilon>0$ with complexity $O(\epsilon^{-2(d^2)(\lfloor (d^2+1)/2\rfloor+1)})$, and with  memory of size $O(\epsilon^{-(d^2)(\lfloor (d^2+1)/2\rfloor+1)})$.
For technical reasons, detailed in Supplementary \ref{sec:sup_alg_theo}, the implied constant in this case depends on the marginals and their specific geometry. In contrast, the constant in the multilinear case is less constrained by the marginals as it only depends on $R$ and $\lambda$ (under assumption \ref{assump_full_dim}). 

In practice, there is also no need to approximate the whole convex set $P_\Pi$ to improve the convergence, but only regions with the lowest costs $-Q_m$. To do so, we use a different strategy than Kamenev \cite{kamenev1996algorithm} to choose the direction of the cuts. Instead of maximizing the gap between $P_{\Pi,k}^+$ and $P_{\Pi,k}^-$ in the direction $g_{k+1}$ as in \eqref{eq:Hausdorff_direc}, we set,
\begin{equation}\label{eq:concave_direction}
    g_{k+1} = \argmax_{g\in F(P_{\Pi,k}^-)} \left[\langle g,x^+\rangle - \max_{h\in P_{\Pi,k}^-}\langle g,h\rangle\right],
\end{equation}
with $x^+=\argmin_{x\in P_{\Pi,k}^+} -Q_m(x)$. In other words, we select the facet with the largest positive gap to $x^+$, and use its normal vector as direction, ensuring that $x^+$ is removed from $P_{\Pi}^+$ when updated (if not in $P_{\Pi}$). This direction selection is thus a slight variation of \eqref{eq:Hausdorff_direc}, and similarly, the term $\max_{h\in P_{\Pi}^-}\langle g,h\rangle$ is known in the $H$-representation of $P_{\Pi}^-$.
Some variants of the algorithm, that can be used to approximate solutions when the cost is concave are presented in Supplementary \ref{sec:supp_alg_cvx_ex}, notably Algorithm \ref{alg:parctical_concave} selecting the optimal direction based on \eqref{eq:concave_direction}. We are using it in practical applications. 

Since the optimization of  $-Q_m$ can be efficiently done at each iteration on both bounding boxes, the respective optimal costs  $c^+,c^-$ of $-Q_m$ on $P_{\Pi}^+$ and $P_{\Pi}^-$, can then be computed, yielding a certificate of optimality and convergence $|c^+-c^-|$. Note that $c^+$ can be computed by vertex enumeration under the assumption that if $-Q_m$ is concave on $P_{\Pi}^+$. When this is not the case, once $P_{\Pi}^+$ has converged close enough to $P_{\Pi}$, vertex enumeration yields a good estimation, and Algorithm \ref{alg:main_GW_m_concave} can easily be adapted.
The certificate enables one to upper and lower bound the global solution of $\GW_m$, and to stop the algorithm when the bounds reach a given relative tolerance $\epsilon$, i.e. if:
$\frac{|c^+-c^-|}{Q_m(\tilde{\Sigma}^{\mu})+Q_m(\tilde{\Sigma}^{\nu})+2c^-}<\epsilon$. Figure \ref{fig:conv_glo_2} represents the certificate of optimality dynamics in a practical case. Note that $c^-$ converges faster than $c^+$ in our simulations,  specifically during the first iterations, as represented in Figure \ref{fig:supp_conv_certif}. Thus, a low number of iterations may yield a good approximation of the optimal cost even though the certificate is quite large.

\subsection{Local optimization}\label{sec:algo_loc}

Given an initial coupling $\pi_0$ we can apply the Frank-Wolfe algorithm
\cite{frank1956algorithm} to compute a local optimum of the  $\GW_m$ cost. In a continuous setting, we obtain Algorithm \ref{alg:Frank_Wolfe}. Note that the cost for running it is alleviated by leveraging the low-dimensional structure of the problem (see Corollary \ref{coro:reformulation_GW_m}). Indeed, there is no need to compute the integral of the  \ref{eq:GW_m} cost (of dimension $2dn$, $n\geq 2$), since the optimization problem only depends on the cross second-moment matrix $\Sigma^{\pi_k}$ of the coupling $\pi_k$, i.e. an  integral of dimension $2d$. Moreover, we can analytically compute the gradient of the $\GW_m$ cost as the function $c_k:(x,y)\mapsto \langle  x, -\nabla Q_m(\Sigma^{\pi_k})  y\rangle $. Here the gradient $\nabla Q_m(\Sigma^{\pi_k})\in \R^{d^2}$ is reshaped as a $d\times d$ matrix. Note that  we can write  $c_k= \sum_{i,j}  [-\nabla Q_m(\Sigma^{\pi_k})]_{i,j}f_{i,j}(x,y)$, with the basis functions $f_{i,j}$ defined in Section \ref{sec:low_dim_carac} and Corollary \ref{coro:reformulation_GW_m}. For the \ref{eq:IGW} and \ref{eq:DGW} costs, $\nabla Q_m(\Sigma^{\pi_k})$ respectively correspond, up to a positive constant, to $\Sigma^{\pi_k}$ and the cofactor matrix of $\Sigma^{\pi_k}$. Overall, these calculations rely on the fact that we can represent our problem in
either the infinite dimensional  bounded convex set $\Pi(\mu,\nu)$,  or the $d^2$ dimensional compact convex set $P_\Pi\subset V$.

\begin{algorithm}
\caption{Frank-Wolfe algorithm for $\GW_m$}\label{alg:Frank_Wolfe}
\begin{algorithmic}[1]
\Require $\pi_0 \in\Pi(\mu,\nu)$, $ \epsilon>0 $, cost polynomial $Q_m$ (see Corollary \ref{coro:reformulation_GW_m}), error function $(\mathrm{err})$
\State Compute second-moment matrices $\tilde{\Sigma}^{\mu}$, $\tilde{\Sigma}^{\nu}$
    \While{$k=0$ or $\mathrm{err}(\pi_k,\pi_{k-1})>\epsilon$ or $\mathrm{err}(\Sigma^{\pi_k},\Sigma^{\pi_{k-1}})>\epsilon$}
    
    \State Compute the cross second-moment matrix $\Sigma^{\pi_k}$ and the gradient $\nabla Q_m(\Sigma^{\pi_k})$ 
    \State Define the cost function $c_k:(x,y)\mapsto \langle  x,-\nabla Q_m(\Sigma^{\pi_k})y\rangle$
    \State Compute $\hat{\pi}_{k+1}=\argmin_{\pi\in\Pi(\mu,\nu)}\int_{(\X\times\Y)} c_k(x,y) d\pi(x,y)$ (optimal transport problem) \label{alg_fw_ot}
        \State $T, \tau\gets\min,\argmin_{\tau\in[0,1]} -Q_m((\tau\Sigma_{i,j}^{\hat{\pi}_{k+1}}+(1-\tau)\Sigma_{i,j}^{\pi_k})_{i,j})$ ($1-D$ polynomial problem for the line search)
        
        \State $\pi_{k+1}\gets \tau\hat{\pi}_{k+1}+(1-\tau)\pi_{k}$
        \State $k \leftarrow k + 1$
    \EndWhile
\State $\GW^2=Q_m(\tilde{\Sigma}^{\mu})+Q_m(\tilde{\Sigma}^{\nu})+2T$
\State\Return $\pi_k, \GW^2$ or just $T$
\end{algorithmic}
\end{algorithm}

Since $Q_m$ is smooth on $P_\Pi$, the objective value in Algorithm \ref{alg:Frank_Wolfe} converges towards the value of an approximate stationary point of $\GW_m$ at a rate $O(\frac{1}{\sqrt{k}})$, with $k$ the number of iterations \cite{lacoste2016convergence}.
The complexity of each iteration corresponds to solving an optimal transport problem with the cost $c_k$ (step \ref{alg_fw_ot}) and can be efficiently tackled with linear programming algorithms or approximated using the Sinkhorn algorithm \cite{cuturi2013sinkhorn}. The line search step of the algorithm is a one dimensional polynomial optimization over $[0,1]$, which can be solved analytically for $m$ of low enough degree, and can be tackled in the general case using Newton methods and branch and bound. For example, we derived analytical formulas and implemented them for the line search in the \ref{eq:IGW} case (in any dimension), the \ref{eq:DGW} and \ref{eq:CGW} cases in dimension 2 or 3 (see Supplementary \ref{sec:supp_line_search}). When $Q_m$ is concave, Algorithm \ref{alg:Frank_Wolfe} corresponds to hopping between vertices of the feasible set. In this case, the concavity enables to better control the Frank-Wolfe gap, enabling a convergence with rate $O(\frac{1}{k})$ (see Supplementary \ref{sec:sup_algo_loc}).

\section{Numerical experiments}\label{sec:numeric} We implemented the global optimization algorithm for concave cost as described in Section \ref{sec:algo_glob_concave}, as well as the Frank-Wolfe algorithm of Section \ref{sec:algo_loc}. We also provide an implementation of the 2-Gromov-Wasserstein distance.

\begin{figure}[ht!]
    \centering
    \includegraphics[width=0.9\linewidth]{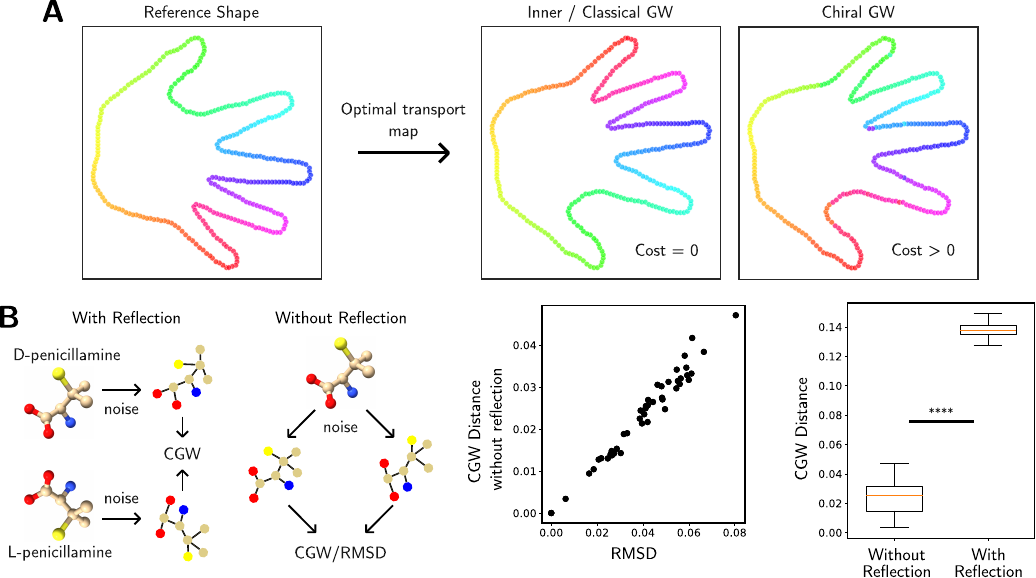}
    \caption{Chiral shape comparison using the \ref{eq:CGW} distance. A) Transport of a reference right hand to a left hand using the Classical, Inner and Chiral Gromov-Wasserstein optimal transport maps.  Only the $\CGW$ distance between the two hands is positive, capturing the change of handedness.  B) Comparison of penicillamine conformations using the $\CGW$ distance. We compared either the two enantiomers of penicillamine, each perturbed with noise (with reflection), or two noisy copies of the same enantiomer (without reflection). Without reflection, the $\CGW$ distance and Root-Mean-Square Deviation (RMSD) are linearly correlated. On the right side, the boxplots represent $\CGW$ distances between conformations of the same handedness (without reflection) and of opposite handedness (with reflection). Both the Mann–Whitney $U$ test and the Kolmogorov–Smirnov test comparing the two distributions give a p-value $p < 10^{-5}$.}
    \label{fig:hand}
\end{figure}

Figure \ref{fig:hand}A represents the optimal costs and transport plans between a right-hand reference shape (298 points) and a left-hand target shape. All the costs used are concave and have an optimal transport map. We used the parameter $t=0.501 $ for the $\CGW$ cost. The color scheme of the reference shape is transported using the $\GW_m$ optimal transport map. As expected, the \ref{eq:IGW} and  $\GW_2$ costs  do not capture the handedness (i.e. their chirality). Indeed, the optimal map is a reflection with respect to the horizontal axis and it yields a $0$ cost. On the contrary, the \ref{eq:CGW} distance is non-zero and its optimal map forces the transport between the pinky and the thumb. The evolution of $c^+$ and $c^-$, the bounds on the optimal cost, can be found in Figure \ref{fig:supp_conv_certif}. Note that the transported colormap for the $\CGW$ case is not continuous on the outline, which is typical for optimal transport between non-convex shapes \cite{nouri2025optimal}.

Figure \ref{fig:hand}B represents comparison of 3D penicillamine conformations. We are focusing on the two enantiomers: D- and L-penicillamine. On each structure, we add random Gaussian noise at a range of levels to simulate small thermal fluctuations around a ground state \cite{frenkel2023understanding}. We then compare conformations of molecules with the same chirality (without reflection) or with different chirality (with reflection through with respect to the $yz$-plane) using the \ref{eq:CGW} distance in 3D, with uniform weighting on each atom. We rescaled each molecule such that it is contained in a ball of radius 1, keeping the relative scale constant between each pair, and used a value of $t=0.8$. We plotted the  $\CGW$ distance between non-reflected conformations (middle panel), as a function of the \emph{Root-Mean-Square Deviation} (RMSD), which is a simple measure of how much the conformations differ. $\CGW$ and RMSD are linearly correlated, showing that $\CGW$ captures small deformations in shape as the RMSD does. Furthermore, we computed the $\CGW$ costs between molecules obtained without reflection and then with reflection and compared the two resulting distributions (right panel). The relatively larger CGW distances for molecules with reflection suggests that CGW captures the chirality of the enantiomers. As the molecule is small enough (9 heavy atoms) and the cost concave, we have also optimized the objective by enumeration over all permutation maps to confirm that our algorithm converged to the global solution.

To study how our global optimization algorithm performs in practice with respect to the theoretical bounds and complexity established earlier, we also ran experiments on standard Gaussian distributions in 2 and 3D (see Figure \ref{fig:conv_glo_2}). In particular, with respect to the number of iterations $k$, the number of vertices and constraints of the bounding box  was significantly less than the upper bound derived in Theorem \ref{theo:FPTAS}: In 2D, they are both linear while the upper bound is quadratic (Figure \ref{fig:conv_glo_2}A), enabling running several thousands of iterations without memory issues on a standard laptop. Similarly in 3D, we found a vertex and constraint number of
 $O(k^3)$ while the upper bound is $O(k^4)$ (Figure \ref{fig:conv_glo_2}B).
As our algorithm shares similarities with the one described by Ryner et al. \cite{ryner2023globally}, we also compared their algorithm to ours. Note that their superiority to generic non-linear global optimization algorithms was already established in benchmarking \cite{ryner2023globally}. Our algorithm performed similarly in the numerical experiments we ran, but with some differences related to the number of iterations and the marginal size (see Figure \ref{fig:bench_ryner} for more details). 

Still, computing $\GW_m$ global solutions  for problems that require many iterations in our current implementation reaches memory-time trade-off limits when the bounding boxes have $\sim 10^5$ vertices or constraints. This trade-off constitutes the main computational bottleneck, even when a small number of points in the shape allows for fast exact optimal transport computation (typically for shapes with 5000 points or less as represented in Figure \ref{fig:conv_glo_2}). 

As an alternative to our global optimization scheme, local optimization (Algorithm \ref{alg:Frank_Wolfe}) can be run in parallel over a range of initial plans or initial search directions. Even though local optimization has a relatively slow theoretical convergence rate (see Section \ref{sec:algo_loc}), our numerical experiments showed convergence in less than 30 iterations (Figure \ref{fig:conv_FW}). Finally, Algorithm \ref{alg:Frank_Wolfe} can be adapted to larger data size using entropy regularization to solve the optimal transport sub-problems (Figure \ref{fig:conv_FW_entrop}).

\section{Conclusion and discussion}
We introduced the multilinear Gromov-Wasserstein distance \ref{eq:GW_m}, as a powerful tool to compare shapes up to transformations of a matrix Lie group $G$, provided the existence of a non-degenerate multilinear form $m$ associated with it (Corollary \ref{coro:dist_shape}). This distance has theoretical properties that are relevant for practical use: it is robust to variations of the marginals in a quantifiable way (Theorem \ref{theo:sample_complexity}), and the $\GW_m$ problem admits optimal correspondence maps, enabling shape registration and alignment. We have leveraged the low-dimensional properties of the $\GW_m$ problem to propose a local optimization scheme (Algorithm \ref{alg:Frank_Wolfe}), as well as an algorithm  approximating the global solution in polynomial time (in $N$ and $1/\epsilon$) when the cost is concave, making the $\GW_m$ problem tractable. The \ref{eq:CGW} distance is a specific example, relevant to biological and chemical applications as it captures the chirality of objects by construction. For a parameter $t $ large enough, the $\CGW_t$ cost is concave, making the distance computable with our algorithm. 

Upon implementing and testing our concave algorithm in dimensions 2 and 3, we confirmed that the $\CGW$ distance captures change in shape chirality. While this algorithm is suitable to compare 2D shapes as is, some further optimization, that would for example take into account suitable data structures, should enable running more iterations in dimension 3. We can also control the accuracy of the algorithm at each iteration with a certificate of global optimality, that can potentially avoid unnecessary computing. As highlighted in Section \ref{sec:algo_glob_concave}, our algorithm can run on large data, as long as we can solve an optimal transport problem on it. In principle, this optimal transport problem can be approximated, using for example entropy-regularized optimal transport and the  Sinkhorn algorithm \cite{cuturi2013sinkhorn}, or using the low rank of the cost matrix to run  low rank optimal transport \cite{scetbon2021low,scetbon2022low}. These alternatives allow for larger data size and faster approximation of exact optimal transport solutions. As such approximations introduce noise in the cost, it would be interesting to assess their impact on the algorithm stability. We are currently pursuing these research directions.

\paragraph{Data availability:}
The implementation of our code in Python can be found in  \href{https://github.com/geoffwoollard/xgw}{this} and \href{https://doi.org/10.5281/zenodo.23930940}{this} repositories which also include extensive testing and the numerical experiments of this paper. The penicillamine structure originates from \href{https://pubchem.ncbi.nlm.nih.gov/}{PubChem} \cite{Kim2025_pubchem} with the compound CID: 4727.

\paragraph{Acknowledgement:}
The authors acknowledge support from the Pacific Institute of Mathematical Sciences (PIMS) and PIMS Kantorovich Initiative. GW is supported by the Canada Graduate Scholarship – Doctorate (CGS-D) from NSERC (Natural Sciences and Engineering Research Council of Canada). AW acknowledges support from: the Burroughs Wellcome Fund, NSERC via Discovery Grants RGPIN-2019-03926 and RGPIN-2025-06747, and an Exploration Grant (NFRFE-2019-00944) from the New Frontiers in Research Fund (NFRF).
KDD was supported by an NSERC Discovery Grant RGPIN-2020-05348. 

\bibliographystyle{plain}
\bibliography{bibliography}
\clearpage

\begin{appendices}

\section*{Notation}
\begin{itemize}
    \item $\R$ the set of real numbers,
    \item $d$  the dimension of the ambient space, typically $d=2,3$,
    \item $\X,\Y\subset \R^d$, typically compact in practical applications, equipped with the canonical inner product $\langle\cdot,\cdot\rangle$,
    \item $\|\cdot\|$ the 2-norm on $\R^d$, 
    \item $\|\cdot\|_F$ the matrix Frobenius norm,
    \item $\det$ the determinant,
    \item For $x\in\X$, we will write $x=(x^1,\dots,x^d)$ its coordinates,
    \item $\mathbf{x}\coloneqq (x_1,\dots,x_k)\in (\R^d)^k$ is a family of elements of $\R^d$. It can be seen as an element of $\R^{dk}$ as well,
    \item $\P(\X), $ space of probability measures on $\X$,
    \item For $\mu \in \P(\X)$, we denote its second moment $\mathfrak{m}_2(\mu)\coloneqq\int_\X \|x\|^2d\mu$, 
    \item $\P_2(\X)$ space of probability measures on $\X$ with finite second moment,
    \item $m$ multilinear $n$-form (or $n$-tensor) over $\R^d$, $n\in \N\backslash\{0,1\}$,
    \item $G$ matrix Lie group over $\R$. Examples are $O(d), SO(d)$ the orthogonal and special orthogonal groups, $SL(d)$ the special linear group,  or $E(d), SE(d)$ Euclidean and special Euclidean groups. 
    \item $\Pi(\mu,\nu)$ the set of couplings between $\mu\in\P(\X)$ and $\nu\in\P(\Y)$. $\Pi(\mu,\nu)\subset\P(\X\times\Y)$,
    \item Given $\mu\in\P(\X)$ and $\nu\in\P(\Y)$, we define the product measures $\mu\times\nu\in \P(\X\times \Y)$ and $\mu^k\in \P(\X^k)$,
    \item Given a coupling $\pi\in\Pi(\mu,\nu)$, we define $\pi^{\otimes k}(\mathbf{x},\mathbf{y})\coloneqq \prod_{i=1}^k\pi(x_i,y_i)$. 
    \item Given a map $g:\X\to\Y$, we define $g^{\times n}\coloneqq (x_1,\dots, x_n)\in \X^n\mapsto (g(x_1),\dots, g(x_n))\in \Y^n$,
    \item Given $\mu\in\P(\X)$ and $f:\X\to\Y$, $f_\sharp\mu\in \P(\Y)$ denotes the measure pushforward by $f$,
    \item Given $\eta\in \P_2(\X)$, we define its second-moment matrix $\tilde{\Sigma}^\eta = \int_{\X}xx^Td\eta(x)$, 
    \item Given $\xi\in \P_2(\X\times\Y)$  we define its cross second-moment matrix $\Sigma^\xi=\int_{\X\times \Y}xy^Td\xi(x,y)$,
    \item Given a symmetric matrix $\Sigma$, we define $\mathrm{Spec}$ its spectrum (the set of its eigenvalues),
    \item $\mathbb{E}[\cdot]$ refers to the expectation of a random variable,
    \item $\mathcal{H}_2$ refers to the Hausdorff distance between polytopes in a Euclidean space (using the 2-norm).
    
\end{itemize}

\section{Proofs of Section \ref{sec:theory}}\label{sec:supp_theory_gw}
\subsection{Proofs of Section \ref{sec:GW_dist}}
\subsubsection*{Proof of Lemma \ref{lemma:distance}}
\begin{proof}Let $(\mu,\nu)\in 
    \P_2(\X)^2$. We first show that $\GW_m(\mu,\nu)<\infty$. Take the coupling $\pi_0 = \mu\times\nu$. We have:
    \begin{equation*}
        \begin{split}
            \GW_m(\mu,\nu)^2&\leq\int_{(\X\times\Y)^{n}}(m(\mathbf{x})-m(\mathbf{y}))^2d\pi_0^{\otimes n}(\mathbf{x},\mathbf{y})\\
            &=\int_{\X^{n}}m(\mathbf{x})^2d\mu^n-2\int_{\X^{n}}m(\mathbf{x})d\mu^n\int_{\Y^{n}}m(\mathbf{y})d\nu^n+\int_{\Y^{n}}m(\mathbf{y})^2d\nu^n<+\infty,
        \end{split}
    \end{equation*}
    since the last term is polynomial function of the two first moments of $\mu,\nu$. In fact, if $m(\mathbf{x})=\sum_{(i_1,\dots,i_n)\in [1,d]^n}\alpha_{i_1,\dots,i_n}x_1^{i_1}\dots x_n^{i_n}$ is the expansion of the multilinear form on the canonical basis, then we have 
    \begin{equation*}
        \begin{split}
            \int_{\X^n}m(\mathbf{x})^2d\mu^{ n}
            &=\sum_{\substack{(i_1,\dots,i_n)\\(j_1,\dots,j_n)}}\alpha_{i_1,\dots,i_n}\alpha_{j_1,\dots,j_n}\int_{\X^n}x_1^{i_1}x_1^{j_1}\dots x_n^{i_n} x_n^{j_n}d\mu^n\\
            &=\sum_{\substack{(i_1,\dots,i_n)\\(j_1,\dots,j_n)}}\alpha_{i_1,\dots,i_n}\alpha_{j_1,\dots,j_n}\int_{\X}x_1^{i_1}x_1^{j_1}d\mu\dots \int_{\X}x_n^{i_n}x_n^{j_n}d\mu\\
            &\underset{\mathrm{(Cauchy-Schwarz)}}{\leq} \sum_{\substack{(i_1,\dots,i_n)\\(j_1,\dots,j_n)}}\alpha_{i_1,\dots,i_n}\alpha_{j_1,\dots,j_n}\int_{\X}\|x_1\|^2d\mu\dots \int_{\X}\|x_n\|^2d\mu<+\infty.\\
        \end{split}
    \end{equation*}
    Now, $\GW_m$ is clearly non-negative and symmetric; and $\GW_m(\mu,\mu)=0$ since\[\int_{(\X\times\X)^n}|m(\mathbf{x})-m(\mathbf{y})|^2d(\textrm{id,id})_\sharp\mu^{\otimes n} = 0.\]
    Finally, let's prove that the triangle inequality holds. Using the Gluing Lemma \cite[Theorem 1.1.10]{dudley2014uniform}, given $\mu,\gamma,\nu\in \P(\X)$ and transport plans in $\pi_{\mu,\gamma}\in \Pi(\mu,\gamma)$, $\pi_{\gamma,\nu}\in \Pi(\gamma,\nu)$, there exists $\pi\in \P(\X^3)$, with marginals $\pi_{\mu,\gamma}$ and $\pi_{\gamma,\nu}$ when projecting on the first and third copy of $\X$ respectively. Here we choose $\pi_{\mu,\gamma},\pi_{\gamma,\nu}$ optimal solutions of $\GW_m(\mu,\gamma),\GW_m(\gamma,\nu)$ respectively and define  $\pi_{\mu,\nu}=\int_\X \pi(\cdot,dx,\cdot)$, then $\pi_{\mu,\nu}\in \Pi(\mu,\nu)$. Now:
    \begin{equation*}
        \begin{split}
            \GW_m(\mu,\nu)\leq&\left(\int_{(\X^2)^n}(m(\mathbf{x})-m(\mathbf{z}))^2 d\pi_{\mu,\nu}^{\otimes n}(\mathbf{x},\mathbf{z})\right)^{\frac{1}{2}}\\
            =&\left(\int_{(\X^3)^n}(m(\mathbf{x})-m(\mathbf{z}))^2 d\pi^{\otimes n}(\mathbf{x},\mathbf{y},\mathbf{z})\right)^{\frac{1}{2}}\\
            =&\left(\int_{(\X^3)^n}(m(\mathbf{x})-m(\mathbf{y})-(m(\mathbf{z})-m(\mathbf{y})))^2 d\pi^{\otimes n}(\mathbf{x},\mathbf{y},\mathbf{z})\right)^{\frac{1}{2}}\\
            \underset{\mathrm{(Minkowski)}}{\leq}&\left(\int_{(\X^3)^n}(m(\mathbf{x})-m(\mathbf{y}))^2 d\pi^{\otimes n}(\mathbf{x},\mathbf{y},\mathbf{z})\right)^{\frac{1}{2}}+\left(\int_{(\X^3)^n}(m(\mathbf{z})-m(\mathbf{y}))^2 d\pi^{\otimes n}(\mathbf{x},\mathbf{y},\mathbf{z})\right)^{\frac{1}{2}}\\
            =&\left(\int_{(\X^2)^n}(m(\mathbf{x})-m(\mathbf{y}))^2 d\pi_{\mu,\gamma}^{\otimes n}(\mathbf{x},\mathbf{y})\right)^{\frac{1}{2}}+\left(\int_{(\X^2)^n}(m(\mathbf{z})-m(\mathbf{y}))^2 d\pi_{\gamma,\nu}^{\otimes n}(\mathbf{y},\mathbf{z})\right)^{\frac{1}{2}}\\
            =& \GW_m(\mu,\gamma)+\GW_m(\gamma,\nu).\\
        \end{split}
    \end{equation*}

Altogether, $\GW_m$ is a pseudo-distance on $\P_2(\X)$.    
\end{proof}

\subsubsection*{Proof of Proposition \ref{proposition:translation}}
\begin{proof}
    Let's fix $\pi_0\in\Pi(\mu,\nu)$. Now, let's define $\pi\coloneqq (S_{-\bar{\mu}} ,S_{-\bar{\nu}})_\sharp \pi_0$ which has centered marginals and study $\inf_{t,u\in\R^d}C(\pi,t,u)$. Expanding the term in the integral, we get:
     \begin{equation*}
        \begin{split}
            \inf_{t,u\in\R^d}C(\pi,t,u)&= \inf_{t,u\in\R^d}\int[m(\mathbf{x}+\mathbf{i}_n(t))-m(\mathbf{y}+\mathbf{i}_n(u))]^2d\pi^{\otimes n}(\mathbf{x},\mathbf{y})\\
            &=  C(\pi,0,0)+\inf_{t,u\in\R^d}\int[m(\mathbf{x}+\mathbf{i}_n(t))-m(\mathbf{x})-m(\mathbf{y}+\mathbf{i}_n(u))+m(\mathbf{y})]^2d\pi^{\otimes n}(\mathbf{x},\mathbf{y})\\
            &+2\int[m(\mathbf{x}+\mathbf{i}_n(t))-m(\mathbf{x})-m(\mathbf{y}+\mathbf{i}_n(u))+m(\mathbf{y})][m(\mathbf{x})-m(\mathbf{y})]d\pi^{\otimes n}(\mathbf{x},\mathbf{y}).
        \end{split}
    \end{equation*}
    The first integral is non-negative. Let's prove that the second integral is actually 0. We can expand $m(\mathbf{x}+\mathbf{i}_n(t))-m(\mathbf{x})$ into a sum involving $t$ in at least one entry of $m$, at entry $i$ let's say, such that we write $m(\dots,\underset{i}{t},\dots)$. Now integrating over the $i$-th copy of $\X\times\Y$ we obtain:
    \begin{equation*}
       \begin{split}
            \int m(\dots,\underset{i}{t},\dots)m(\mathbf{x})d\pi(x_i,y_i)
            &=\int m(\dots,\underset{i}{t},\dots)m(x_1,\dots,x_i,\dots,x_n)d\mu(x_i)\\
            &=m(\dots,\underset{i}{t},\dots)m(x_1,\dots,\int x_id\mu(x_i),\dots,x_n)\\
            &=m(\dots,\underset{i}{t},\dots)m(x_1,\dots,\underset{i}{0},\dots,x_n) = 0.
       \end{split}
    \end{equation*}
    We used here the multilinearity and that  $\mu$ is centered. Similarly, \[\int m(\dots,\underset{i}{t},\dots)m(\mathbf{y})d\pi(x_i,y_i)=0.\]
    Finally, we have $\inf_{t,u\in\R^d}C(\pi,t,u)= C(\pi,0,0)=C(\pi_0,-\bar{\mu},-\bar{\nu})$, so the claim holds.
\end{proof}

\subsubsection*{Proof of Proposition \ref{proposition:scaling}}
\begin{proof}
    Given $\alpha>0$ and $\pi\in\Pi(\mu,\nu)$, we define 
    \begin{equation*}
        \begin{split}
            C(\pi,\alpha)&=\int[m(\mathbf{x})-m(\alpha\mathbf{y})]^2d\pi^{\otimes n}\\
            &= \int[m(\mathbf{x})-\alpha^nm(\mathbf{y})]^2d\pi^{\otimes n}(\mathbf{x},\mathbf{y})\\
            &= \int m(\mathbf{x})^2d\mu^n+\alpha^{2n}\int m(\mathbf{y})^2d\nu^n-2\alpha^{n}\int m(\mathbf{x})m(\mathbf{y})d\pi^{\otimes n}(\mathbf{x},\mathbf{y}).
        \end{split}
    \end{equation*}
    Now  $\GW_m(\mu,U_{\alpha\;\sharp}\nu)= \inf_{\pi\in\Pi(\mu,\nu)} C(\pi,\alpha)$.
    Given a solution $\pi^*$ minimizing $C(\pi,\alpha)$, $\pi^*$ equivalently solves
    \[  \inf_{\pi\in\Pi(\mu,\nu)}-2\int_{(\X\times\Y)^n}m(\mathbf{x})m(\mathbf{y})d\pi^{\otimes n}(\mathbf{x},\mathbf{y})\quad\text{ and }\quad \inf_{\pi\in\Pi(\mu,\nu)}C(\pi,1).\]
    In particular, we deduce  that if $\pi^*$ solves the  $\GW_m(\mu,\nu)$ problem then ($p_\X,U_{\alpha}\circ p_\Y)_{\sharp}\pi$ solves the  $\GW_m(\mu,U_{\alpha\;\sharp}\nu)$ problem. Applying the same argument the other way around, we also deduce a one to one correspondence between the solutions of $\GW_m(\mu,U_{\alpha\;\sharp}\nu)$ and $\GW_m(\mu,\nu)$.
\end{proof}

\subsubsection*{Proof of Lemma \ref{lem:degeneracy}}

It is a direct corollary of the following lemma:

\begin{lemma}\label{lem:non_dege_small_set}
    Let $A\subset\R^d$ and $m$ be a non-degenerate $n$-form on $\R^d$. Assume that there exists $x_1,\dots,x_d\in A$ a basis of $\R^d$. Then, 
    \[\{T: A\to\R^d, m\circ T^{\times n}=m\}\subset GL(d),\] 
    and the unique linear extension of $T$ to $\R^d$, $\tilde{T}$, verifies $m\circ \tilde{T}^{\times n}=m$.
    
\end{lemma}

\begin{proof}[Proof of Lemma \ref{lem:non_dege_small_set}]
    Without loss of generality, we assume that $m$ is non-degenerate at slot 1. Using the multilinearity of $m$, we can write:
    \[\forall y, y_2\dots,y_n\in \R^d,\quad m(y, y_2,\dots,y_n) = \langle y, f(y_2,\dots,y_n)\rangle.\]
    Here $f:(\R^d)^{n-1}\to\R^d$ with its coordinates being multilinear $n-1$ forms. 
    
    We note $\mathbf{z} = (z_1,\dots,z_{n-1})\in (\R^d)^{n-1}$ First, let's prove that there exists $(\mathbf{z_1},\dots,\mathbf{z_d})\in (A^{n-1})^d$ such that  $(f(\mathbf{z}_1),\dots,f(\mathbf{z}_d))$ is a basis of $\R^d$. Let's assume this is not true, then there exists $v\in \R^d$, such that for all $y_1,\dots,y_{n-1}\in A$, $m(v,y_1,\dots,y_{n-1})=\langle v, f(y_1,\dots,y_{n-1})\rangle=0.$ Since a basis of $\R^d$ is contained in $A$, we deduce by linearity, that $m(v,y_1,\dots,y_{n-1}) = 0$ for all $y_1,\dots,y_{n-1}\in \R^d$. This contradicts the non-degeneracy assumption. From now on, we  fix the vectors $\mathbf{z_1},\dots,\mathbf{z_d}\in A^{n-1}$ such that  $(f(\mathbf{z}_1),\dots,f(\mathbf{z}_d))$ is a basis of $\R^d$.

    Using our assumption, let's choose $(x)=(x_1,\dots,x_d)\in A^d$ a basis of $\R^d$. Let's define  $e_i = f(\mathbf{z}_i)\in \R^d$ for $1\leq i\leq d$. Then  $(e)\coloneqq (e_1,\dots,e_d)$ is a basis of $\R^d$, as well. Now, for $1\leq i,j\leq d$, 
    \begin{equation}\label{eq:linearity_T}
        \langle x_i,e_j\rangle = m(x_i, \mathbf{z}_j) =m(T(x_i), T^{\times n-1}(\mathbf{z}_j)) = \left\langle T(x_i),f(T^{\times n-1}(\mathbf{z}_j))\right\rangle.
    \end{equation}
    Let's define the $d\times d$ matrix $M$, as $M_{i,j} = \langle x_i,e_j\rangle$. Since $(e)$ and $(x)$ are bases, $M$ is invertible. Let's rewrite $f(T^{\times n-1}(\mathbf{z}_i)) = g_i$. From \eqref{eq:linearity_T}, we also have $M_{i,j} = \langle T(x_i), g_j\rangle$. Since $M$ is invertible,  $(g)$ and $(T(x))$ are bases of $\R^d$ as well. In particular, $\mathrm{span}(T(A))=\R^d$. Moreover, $T$ is linear since  \eqref{eq:linearity_T} implies, 
    \[\forall y \in A \text{ with coordinates }(\alpha) \text{ on } (x), \;\forall1\leq i\leq d, \quad \langle T(y),g_i\rangle = \left\langle \sum_i\alpha_iT(x_i),g_i\right\rangle.\]
    Finally, $T$ is linear and surjective, thus invertible. Now, using the linearity of $T$ and $m$, $\tilde{T}$, the linear extension of $T$ to $\R^d$, is uniquely defined and we have $m\circ \tilde{T}^{\times n}=m$.
\end{proof}

\subsubsection*{Proof of Proposition \ref{prop:GLD_transform}}
\begin{proof}
    Let $m$ be a non-degenerate $n$-form on $\R^d$,  and $\mu,\nu\in \P_2(\R^d)$, We assume that there exists $(x_1,\dots,x_d)\in \mathrm{supp}(\mu)^d$ with $(x_1,\dots,x_d))$ basis of $\R^d$. We just need to prove that: \[  \GW_{m}(\mu,\nu)= 0 \implies \exists g\in G, \nu = g_\sharp\mu.\]
    Without loss of generality, we can suppose that $m$ is non-degenerate at slot 1.
    Using the multilinearity of $m$, we can write:
    \[\forall x, x_2\dots,x_n\in \R^d,\quad m(x, x_2,\dots,x_n) = \langle x, f(x_2,\dots,x_n)\rangle,\]
    with $f:(\R^d)^{n-1}\to\R^d$ whose components are  multilinear $(n-1)$-forms.
    We suppose that $\GW_{m}(\mu,\nu)= 0$, then 
    \begin{equation*}
        \begin{split}
             0&=\GW_{m}(\mu,\nu)^2\\
             &=\inf_{\pi\in\Pi(\mu,\nu)}\int_{(\X\times\X)^n}(m(\mathbf{x})-m(\mathbf{y}))^2d\pi^{\otimes n}(\mathbf{x},\mathbf{y}) \\
             &=\int_{(\X\times\X)^n}(m(\mathbf{x})-m(\mathbf{y}))^2d(\pi^*)^{\otimes n}(\mathbf{x},\mathbf{y}) \\
             &=\int_{(\X\times\X)^{n-1}}\left(\int_{\X^2}(\langle x, f(x_2,\dots,x_n)\rangle-\langle y, f(y_2,\dots,y_n)\rangle)^2 d\pi^*(x,y)\right) d(\pi^*)^{\otimes n-1}. \\
        \end{split}
    \end{equation*}
    We define \[B_1\coloneqq \{(\mathbf{x}_1,\dots,\mathbf{x}_{d}) \in (\X^{(n-1)})^{d}, (f(\mathbf{x}_1),\dots,f(\mathbf{x}_d))\text{ is a basis of }\R^d\} .\]  The assumption on the support of $\mu$ and the proof of Lemma \ref{lem:non_dege_small_set} ensure that there exist $\mathbf{x}_1,\dots,\mathbf{x}_{d} \in \mathrm{supp}(\mu^{n-1})$, such that $\det(f(\mathbf{x}_1),\dots,f(\mathbf{x}_d))>0$. Now we use the continuity of $f$ and the determinant to conclude that $\mu^{d(n-1)}(B_1)$ has non zero measure. But, given $(\mathbf{x}_1,\dots,\mathbf{x}_{d}) \in B_1$,  $(\mathbf{y}_1,\dots,\mathbf{y}_{d}) \in \R^{d\times(n-1)\times d}$:
    \[\forall 1\leq i,j\leq d,\;\langle x, f(\mathbf{x}_i)\rangle = \langle y, f(\mathbf{y}_j)\rangle \quad  \pi^*\text{-a.e.,}\]
    with $(f(\mathbf{x}_1),\dots,f(\mathbf{x}_{d}))$ a basis of $\R^d$. We reuse the assumption on the support of $\mu$ to ensure that the $x$ variable has to span $\R^d$, leading us to conclude from Lemma \ref{lem:non_dege_small_set}, that there exist $g$ invertible matrix such that $y=gx$ $\pi^*\text{-a.e.,}$ and $\nu = g_\sharp\mu$. We conclude that $g\in G$ by applying  Lemma \ref{lem:non_dege_small_set} on the support of $\mu$.

\end{proof}

\subsubsection*{Proof of Corollary \ref{coro:dist_shape}}

\begin{proof}
    This result holds since \[G_{\langle\cdot,\cdot\rangle}\coloneqq{\{T:\R^d\to\R^d,\langle\cdot,\cdot\rangle=\langle T(\cdot),T(\cdot)\rangle\}} = O(d),\]
    and 
    \[G_{\det}\coloneqq{\{T:\R^d\to\R^d,\det=\det\circ T^{\times d}\}} = SL(d).\]
\end{proof}
\subsection{Proofs of Section \ref{sec:low_dim_carac}} \label{sec:supp-low_dim_carac}
\subsubsection*{Proof of Proposition \ref{proposition:GWm_problem}}

\begin{proof}[Proof of Proposition \ref{proposition:GWm_problem}]
    Let $\X,\Y\subset \R^d$, $\mu\in \P_2(\X)$  and $\nu\in \P_2(\Y)$.   Let's first rewrite the $\GW_m(\mu,\nu)$ problem:
    \begin{equation*}
            \GW_m(\mu,\nu)^2
            =\int_{(\X\times\Y)^n}m(\mathbf{x})^2d\mu^{ n}+\int_{(\X\times\Y)^n}m(\mathbf{y})^2d\nu^{n}+2\inf_{\pi\in\Pi(\mu,\nu)} - T_\pi,
    \end{equation*}
    with 
    \begin{equation}\label{eq:T_pi}
        T_\pi\coloneqq \int_{(\X\times\Y)^n}m(\mathbf{x})m(\mathbf{y})d\pi^{\otimes n}(\mathbf{x},\mathbf{y}).
    \end{equation}
    Let $m(\mathbf{x})=\sum_{(i_1,\dots,i_n)\in [1,d]^n}\alpha_{i_1,\dots,i_n}x_1^{i_1}\dots x_n^{i_n},$ be the expansion of the multilinear form on the canonical basis. We can compute the first two terms as follows:
    \begin{equation*}
        \begin{split}
            \int_{\X^n}m(\mathbf{x})^2d\mu^{ n}
            &=\sum_{\substack{(i_1,\dots,i_n)\\(j_1,\dots,j_n)}}\alpha_{i_1,\dots,i_n}\alpha_{j_1,\dots,j_n}\int_{\X^n}x_1^{i_1}x_1^{j_1}\dots x_n^{i_n} x_n^{j_n}d\mu^n\\
            &=\sum_{\substack{(i_1,\dots,i_n)\\(j_1,\dots,j_n)}}\alpha_{i_1,\dots,i_n}\alpha_{j_1,\dots,j_n}\int_{\X}x_1^{i_1}x_1^{j_1}d\mu\dots \int_{\X}x_n^{i_n}x_n^{j_n}d\mu\\
            &= \sum_{\substack{(i_1,\dots,i_n)\\(j_1,\dots,j_n)}}\alpha_{i_1,\dots,i_n}\alpha_{j_1,\dots,j_n} \tilde{\Sigma}^\mu_{i_1,j_1}\dots \tilde{\Sigma}^\mu_{i_n,j_n}\eqqcolon Q_m(\tilde{\Sigma}^{\mu}).
        \end{split}
    \end{equation*}
    Note that this term is non-negative. Since we can choose a marginal $\mu$ to get any symmetric PSD matrix, the polynomial $Q_m$ is non-negative over all symmetric PSD matrices. $Q_m$ has $d^2$ variables and a degree at most $n$.
    Rewriting $T_\pi$ and integrating recursively, we obtain similarly:
        \begin{equation*}
               \begin{split}
                T_\pi
                =&\sum_{\substack{(i_1,\dots,i_n)\\(j_1,\dots,j_n)}}\alpha_{i_1,\dots,i_n}\alpha_{j_1,\dots,j_n}\int_{(\X\times\Y)^n}x_1^{i_1}y_1^{j_1}\dots x_n^{i_n} y_n^{j_n}d\pi^{\otimes n}(\mathbf{x},\mathbf{y})\\
                &=\sum_{\substack{(i_1,\dots,i_n)\\(j_1,\dots,j_n)}}\alpha_{i_1,\dots,i_n}\alpha_{j_1,\dots,j_n}\Sigma_{i_1,j_1}^{\pi}\dots \Sigma^{\pi}_{i_n,j_n} = Q_m(\Sigma^{\pi}) = Q_m(\tilde{p}_0(\pi)).
               \end{split}
        \end{equation*}
    Let's prove that $\tilde{p}_0$ is continuous. Let $\pi_n$ converge to $\pi$ with respect to $W_2$ in $\P_2(\X\times\Y)$. We define the functions $f_{i,j}:(x,y)\in\X\times\Y\mapsto x^iy^j\in \R$. Note that for $\eta\in \P_2(\X\times\Y)$, $\tilde p_0(\eta)_{i,j}=\int_{\X\times\Y}f_{i,j}\;d\eta$.  If we note $\mathbf{x}=(x,y)\in \R^{2d}$, then we have $|f_{i,j}(x,y)|\leq 2 (\|x\|^2+\|y\|^2)= 2\|(x,y)\|^2$. We conclude from Theorem \ref{thm:W2-topology} that $\tilde{p}_0(\pi_n)\to \tilde{p}_0(\pi)$ in $\R^{d^2}$. Since $\tilde{p}_0$ is linear and $\Pi(\mu,\nu)$ is convex, $\tilde{P}_\Pi$ is convex as well. Since $\tilde{p}_0$ is continuous and $\Pi(\mu,\nu)$ is compact, $\tilde{P}_\Pi$ is compact. 
    Finally, \[\inf_{\pi\in \Pi(\mu,\nu)}-Q_m(\Sigma^{\pi}) =\inf_{\pi\in \Pi(\mu,\nu)} -Q_m(\tilde{p}_0(\pi)^{i,j}_{i,j})=\inf_{x\in \tilde{p}_0(\Pi(\mu,\nu))} -Q_m((x^{i,j})_{i,j})=\inf_{x\in \tilde{P}_\Pi} -Q_m((x^{i,j})_{i,j}),\]
    where $(x^{i,j})$ is the coordinate in the $(f_{i,j})$ basis.
\end{proof}

\subsubsection*{Proof of Corollary \ref{coro:reformulation_GW_m}}
\begin{proof}
    This proof is based on the previous one. We assume that $\X,\Y$ are compact with non-empty interior, so is $\X\times\Y$. Let $f_{i,j}:(x,y)\in\X\times\Y\mapsto x^iy^j\in \R$ for $1\leq i,j\leq d$.
    Since $\X,\Y$ are full dimensional, $(f_{i,j})$ is a linearly independent family. Let's use the $L^2$ inner product on  $V=\mathrm{span}(f_{i,j})\subset \mathcal{C}(\X\times\Y)$ as $\langle f,g\rangle = \int_{\X\times\Y}fg$. We can now  choose $(e_{i})_{1\leq i\leq d^2}$ an orthonormal basis of $V$ and define $p_0:L^2(\X\times\Y)\to V$ the orthogonal projection over $V$. Now, as $\P(\X\times\Y)\cap L^2(\X\times\Y) $ is dense in $\P(\X\times\Y)$, we can naturally extend $p_0$ to $\P(\X\times\Y)$, such that  \[\forall v\in V, \pi\in \P(\X\times\Y),\quad \int_{\X\times\Y} v\;d\pi=\int_{\X\times\Y} \sum_{i} \langle v,e_i\rangle e_i\;d\pi=\sum_{i} \langle v,e_i\rangle\langle p_0(\pi),e_i\rangle =\langle v,p_0(\pi)\rangle.\]
    Let $\mathcal{R}_f:V\to\R^{d^2}$ be the representation of the family $(f_{i,j})_{i,j}$ in the basis $(e_{i,j})_{i,j}$, then $\mathcal{R}_f$ is linear invertible. Now $\tilde{p}_0 =\mathcal{R}_f\circ p_0$. Thus, $p_0$ is also continuous, and $P_{\Pi}$ is compact and convex. 
\end{proof}

\subsubsection*{Proof of Example \ref{ex:3_csots_poly}}
\begin{proof}
    
    Now let's apply the results of Proposition \ref{proposition:GWm_problem} to the 3 costs defined earlier. We compute the cost linearization, useful for numerical schemes (like the Frank-Wolfe algorithm as in Section \ref{sec:algo_loc}). Let's expand the term $T_\pi$ as defined in \eqref{eq:T_pi}.
\begin{enumerate}
    
        \item \ref{eq:IGW}.  
        \begin{equation*}
            \begin{split}
                T_\pi=&\int_{(\X\times\Y)^2}\langle x_1,x_2\rangle\langle y_1,y_2\rangle d\pi^{\otimes 2}\\
                =&\int_{\X\times\Y}\left\langle x_1, \int_{\X\times\Y}x_2y_2^Td\pi(x_2,y_2)y_1\right\rangle  d\pi(x_1,y_1)\\
                =&\int_{\X\times\Y}\left\langle x_1, \Sigma^\pi y_1\right\rangle  d\pi(x_1,y_1).\\
            \end{split}
        \end{equation*}
        We can also express this term under this form:
        \begin{equation*}
            \begin{split}
                 T_{\pi}=& \sum_{1\leq i,j\leq d} \int_{(\X\times\Y)^2}x_1^ix_2^iy_1^jy_2^jd\pi^{\otimes 2} \\
                 =&\sum_{1\leq i,j\leq d} (\Sigma_{i,j}^\pi)^2 \\
                =&\|\Sigma^\pi\|^2_F,
            \end{split}
        \end{equation*}
         with $\|\cdot\|_F$ the Frobenius norm. Thus,
        \begin{equation*}
            \begin{split}
                \IGW(\mu,\nu)^2=&\inf_{\pi\in\Pi(\mu,\nu)}\int_{(\X\times\Y)^2}|\langle x_1,x_2\rangle-\langle y_1,y_2\rangle|^2d\pi^{\otimes 2} \\ 
                =&\|\tilde{\Sigma}^{\mu}\|_F^2 + \|\tilde{\Sigma}^{\nu}\|_F^2 +2\inf_{\pi\in\Pi(\mu,\nu)}-\int_{(\X\times\Y)^2}\langle x,\Sigma^\pi y\rangle d\pi\\
                =&\|\tilde{\Sigma}^{\mu}\|_F^2 + \|\tilde{\Sigma}^{\nu}\|_F^2+2\inf_{\pi\in\Pi(\mu,\nu)}-\|\Sigma^\pi\|^2_F.\\
            \end{split}
        \end{equation*}
        Now $\nabla Q_m(\Sigma^\pi)=2\Sigma^\pi$.
        
        \item \ref{eq:DGW}. We rewrite $T_\pi$ as: 
        \begin{equation*}
            \begin{split}
                T_{\pi}=&\int_{(\X\times\Y)^d}\det(x_1,\dots,x_d)\det(y_1,\dots,y_d)d\pi^{\otimes d}\\
                =&\sum_{\sigma_x,\sigma_y\in S(d)}\sgn(\sigma_x)\sgn(\sigma_y)\\ & \int_{(\X\times\Y)^d}(x_1)_{\sigma_x(1)}(y_1)_{\sigma_y(1)}\dots (x_d)_{\sigma_x(d)}(y_d)_{\sigma_y(d)}d\pi^{\otimes d}\\
                =&\sum_{\sigma_x,\sigma_y\in S(d)} \sgn(\sigma_x)\sgn(\sigma_y)\Sigma^\pi_{\sigma_x(1),\sigma_y(1)}\dots\Sigma^\pi_{\sigma_x(d),\sigma_y(d)}\\
                =&\sum_{\sigma_x,\sigma_y\in S(d)} \sgn(\sigma_y)\Sigma^\pi_{1,\sigma_y(1)}\dots\Sigma^\pi_{d,\sigma_y(d)}\\
                =&d! \det(\Sigma^\pi).
            \end{split}
        \end{equation*}
        Using the Laplace expansion of the determinant along all the rows of $\Sigma^\pi$, we can express the determinant in terms of $\Delta$, the matrix of first minors of $\Sigma^\pi$;
        
        \begin{equation*}
            \begin{split}
                T_{\pi}&=  d! \det(\Sigma^\pi)\\
                & = (d-1)!\sum_{i,j}\Sigma^\pi_{i,j}(-1)^{i+j}\Delta_{i,j}\\
                & = (d-1)!\int \langle x,\textrm{com}(\Sigma^\pi) y\rangle d\pi,\\
            \end{split}
        \end{equation*}
         with $\mathrm{com}(M)$ the matrix of cofactors of $M$. Thus,
        \begin{equation}
        \begin{split}\label{eq:DGW_low_dim}
        \DGW(\mu,\nu)^2=&\inf_{\pi\in\Pi(\mu,\nu)}\int_{(\X\times\Y)^d}|\det(\mathbf{x})-\det(\mathbf{y})|^2d\pi^{\otimes d}\\ 
        =&d!\det(\tilde{\Sigma}^{\mu})+ d!\det(\tilde{\Sigma}^{\nu})+2(d-1)!\inf_{\pi\in\Pi(\mu,\nu)}-\int_{(\X\times\Y)^d}\langle x,\mathrm{com}(\Sigma^{\pi}) y\rangle d\pi\\
        =&d!\left[\det(\tilde{\Sigma}^{\mu})+ \det(\tilde{\Sigma}^{\nu})+2\inf_{\pi\in\Pi(\mu,\nu)}-\det(\Sigma^{\pi})\right].
        \end{split}
    \end{equation}
    Now $\nabla Q_m(\Sigma^\pi)=d!\mathrm{com}(\Sigma^{\pi})$.
    
    \item \ref{eq:CGW}.
    We get by linearity:
    \begin{equation}
        \begin{split}\label{eq:CGW_low_dim}
        \CGW(t,\mu,\nu)^2
        =&A+2\inf_{\pi\in\Pi(\mu,\nu)}-\int_{\X\times\Y}\langle x,[t\Sigma^{\pi}+(t-1)(d-1)!\;\mathrm{com}(\Sigma^{\pi})] y\rangle d\pi\\
        =&A+2\inf_{\pi\in\Pi(\mu,\nu)}-t\|\Sigma^\pi\|^2_F- (1-t)d!\det(\Sigma^{\pi}),\\
        \end{split}
    \end{equation}
    with \[A = t[\|\tilde{\Sigma}^{\mu}\|_F^2+\|\tilde{\Sigma}^{\nu}\|_F^2]+(1-t)d![\det(\tilde{\Sigma}^{\mu})+\det(\tilde{\Sigma}^{\nu})].\]
    \end{enumerate}
    Now $\nabla Q_m(\Sigma^\pi)=2t\Sigma^\pi+(1-t)d!\mathrm{com}(\Sigma^{\pi})$
\end{proof}

\subsection{Low-dimensional structure of the 2-Gromov-Wasserstein}\label{sec:classical_gw}
We go over the low-dimensional structure of the classical Gromov-Wasserstein problem, similar to Corollary \ref{coro:reformulation_GW_m}. Assume $\X,\Y\subset \R^d$ are bounded.
Here we suppose that the marginals $(\mu,\nu)\in \P(\X)\times\P(\Y)$ are centered. We can write:
\begin{equation*}
    \begin{split}
        \GW_2(\mu,\nu)^2 =&  \inf_{\pi\in\Pi(\mu,\nu)}\int_{(\X\times\Y)^{2}}(\|x_1-x_2\|^2-\|y_1-y_2\|^2)^2d\pi^{\otimes 2}\\
        =&\int\|x_1-x_2\|^4d\mu^{ 2}+\int\|y_1-y_2\|^4d\nu^{ 2} - 4 \int\|x\|^2\|y\|^2d\mu\times\nu\\
        & +\inf_{\pi\in\Pi(\mu,\nu)}-4\int\|x\|^2\|y\|^2d\pi-8\int\langle x_1,x_2\rangle\langle y_1,y_2\rangle d\pi^{\otimes 2}
    \end{split}
\end{equation*}
Here only the last line depends on the coupling. We can then derive a low-dimensional representation as in Corollary \ref{coro:reformulation_GW_m} with $V=\mathrm{span}(g, (f_{i,j})_{i,j})$, $g:x,y\in \X\times\Y\mapsto \|x\|^2_2\|y\|^2_2$ and $f_{i,j}:x,y\in \X\times\Y\mapsto x^iy^j$. We equip $V$ with the $L^2(\X,\Y)$ inner product and we define $p_0$ the extension to $\P(\X\times\Y)$ of the orthogonal projection on $V$ similarly to  Corollary \ref{coro:reformulation_GW_m}. $p_0$ is still continuous as a consequence of Theorem \ref{thm:W2-topology}, and $P_\Pi=p_0(\Pi(\mu,\nu))$ is still convex bounded. 
Now, the term depending on the coupling $\pi$ is equal to:
\[\inf_{\pi\in\Pi(\mu,\nu)} -4 \int gd\pi-8\|\Sigma^\pi\|_F^2= \inf_{\pi\in\Pi(\mu,\nu)} -4 \langle g,p_0(\pi)\rangle-8\sum_{1\leq i,j\leq d} \langle f_{i,j},p_0(\pi)\rangle^2\eqqcolon Q(p_0(\pi))\]
Thus, the cost only depends polynomially on the projection of the coupling on $V$ a vector space of dimension $d^2+1$. Note that the cost is concave in $\pi$, and we can use a variant of our Algorithm \ref{alg:main_GW_m} to compute its global optimal value. We can also run the Frank-Wolfe Algorithm \ref{alg:Frank_Wolfe} to compute a local solution, with the cost becoming:
\begin{equation*}
    c_k:(x,y)\mapsto 16\langle  x, \Sigma^{\pi_k}  y\rangle +4 \|x\|^2\|y\|^2= 4 g(x,y)+16\sum_{i,j}  \Sigma^{\pi_k}_{i,j}f_{i,j}(x,y).
\end{equation*}

\subsection{Proofs of Section \ref{sec:opt_coupl}}
\subsubsection*{Auxiliary results}

Let $(\X,D)$ be a Polish space and $p\in [1,+\infty)$. Let $\P_p(\X)$ be the set of probability distributions on $\X$ with finite $p$-th moment. We define the $p$-Wasserstein distance on $\P_p(\X)$ as 
\[W_{p}(\mu,\nu)=\left(\inf_{\pi\in\Pi(\mu,\nu)}\int D(x,y)^p d\pi\right)^{1/p}.\]
We quote the following characterization of the $W_{p}$ topology. We denote $\overset{*}{\to}$ the weak-$*$ convergence.

\begin{theorem}[Definition 6.6 and Theorem 6.7 of \cite{villani2008optimal}]
\label{thm:W2-topology}
Let $(\X,D)$ be a complete separable metric space, and $(\mu_{n})_{n\in\mathbb{N}},\mu\in \mathcal{P}_{p}(\X)$. The following are equivalent:
\begin{enumerate}
\item $W_{p}(\mu_{n},\mu)\rightarrow0$ as $n\rightarrow\infty$.
\item $\mu_{n}\overset{*}{\to}\mu$ in duality with $C_{b}(\X)$,
and for some $x_{0}\in \X$, $\int D^{p}(x_{0},x)d\mu_{n}(x)\rightarrow\int D^{p}(x_{0},x)d\mu(x)$.
\item $\mu_{n}\overset{*}{\to}\mu$ in duality with 
the continuous functions with growth of degree $p$: \[C_{p}(\X):=\left\{ \varphi\in C(\X):\exists x_{0}\in \X,\exists a>0,|\varphi(x)|\leq a(D^{p}(x,x_{0})+1)\right\}.\]

\end{enumerate}
\end{theorem}

We will be using the following technical lemma.
\begin{lemma}
\label{lem:w2-product-bound}
Let $\mu,\nu\in\mathcal{P}_{2}(\X)$. Then 
\[
W_{2}^{2}(\mu\times\mu,\nu\times\nu)\leq2W_{2}^{2}(\mu,\nu).
\]
\end{lemma}

\begin{proof}
Let $\pi^{*}$ be an optimal coupling between $\mu$ and $\nu$. Then,

\begin{align*}
\int d_{\mathcal{X}^{2}}^{2}\left((x,x^{\prime}),(y,y^{\prime})\right)d\pi^{*}(x,y)\times d\pi^{*}(x^{\prime},y^{\prime}) & =\int\left(d_{\mathcal{X}}^{2}(x,y)+d_{\mathcal{X}}^{2}(x^{\prime},y^{\prime})\right)d\pi^{*}(x,y)\times d\pi^{*}(x^{\prime},y^{\prime})\\
 & =2W_{2}^{2}(\mu,\nu).
\end{align*}
The left-hand side is an upper bound for $W_{2}^{2}(\mu\times\mu,\nu\times\nu)$
because the $(x,x^{\prime})$-marginal of $d\pi^{*}(x,y)\times d\pi^{*}(x^{\prime},y^{\prime})$
is equal to $d\mu(x)\times d\mu(x^{\prime})$, and similarly for
$\nu$. 
\end{proof}

We shall also use the following result, which says that the entire set of couplings is stable with respect to perturbations of the marginals.

\begin{lemma}[Theorem 1 of \cite{bogachev2024hausdorff}]
\label{lem:glueing-convergence}
Let $\mathcal{X}$ and $\mathcal{Y}$ be complete, separable metric
spaces, and $\mu_{1},\mu_{2}\in\mathcal{P}_{2}(\mathcal{X})$ and
$\nu_{1},\nu_{2}\in\mathcal{P}_{2}(\mathcal{Y})$. Then for every
measure $\pi_{1}\in\Pi(\mu_{1},\nu_{1})$ there exists a measure $\pi_{2}\in\Pi(\mu_{2},\nu_{2})$
such that $W_{2}(\pi_{1},\pi_{2})\leq W_{2}(\mu_{1},\mu_{2})+W_{2}(\nu_{1},\nu_{2})$.
\end{lemma}

\subsubsection*{Proof of Proposition \ref{prop:stability-general}}
\begin{proof}
Let's denote $|\cdot|$ the Euclidean norm. To show existence of minimizers, we first observe that, given $\mu\in\P_2(\X)$, $\nu\in\P_2(\Y)$ and a coupling $\pi\in\Pi(\mu,\nu)$,
\begin{align*}
    |\Sigma^{\pi}|^2 &=\sum_{i,j=1}^d\left|\int x^iy^jd\pi\right|^2\\ &\leq \sum_{i,j=1}^d\left(\int |x^iy^j|d\pi \right)^2 \leq\sum_{i,j}\int |x^i|^2d\pi \int|y^j|^2d\pi = \int |x|^2d\mu \int |y|^2d\nu<+\infty.
\end{align*}
Note that we also have from Proposition \ref{proposition:GWm_problem}:

\[0\leq\int (m(\mathbf{x})-m(\mathbf{y}))^2 d\pi^{\otimes n}= Q(\tilde{\Sigma}^\mu)+Q(\tilde{\Sigma}^\nu)-2Q(\Sigma^\pi),\]
so that \[2\inf_{\pi\in\Pi(\mu,\nu)}-Q(\Sigma^\pi)\geq -Q(\tilde{\Sigma}^\mu)-Q(\tilde{\Sigma}^\nu)>-\infty.\]
Let $\pi_{j}$ be a sequence
in $\Pi(\mu,\nu)$ satisfying $\mathcal{C}(\pi_{j})\searrow\inf_{\pi\in\Pi(\mu,\nu)}\mathcal{C}(\pi)$;
by the proof of \cite[Theorem 1.5]{ambrosio2013user}  it holds that
$\Pi(\mu,\nu$) is precompact with respect to the $W_{2}$ topology
on $\mathcal{P}_{2}(\mathcal{X}\times\mathcal{Y})$, meaning that
we can extract a convergent subsequence (not relabeled) such that
$W_{2}(\pi_{j},\pi^{*})\rightarrow0$ for some limit point $\pi^{*}$.
Furthermore, $\pi^{*}\in\Pi(\mu,\nu)$, as can be seen by integrating
$\pi^{*}$ against test functions depending only on one coordinate.
Thanks to Theorem \ref{thm:W2-topology}, the function $\pi\mapsto\Sigma^\pi$ is continuous, thus $\pi\mapsto C(\pi)=Q(\Sigma^\pi)$ is also continuous, and it holds that $\lim_{j\rightarrow\infty}\mathcal{C}(\pi_{j})=\mathcal{C}(\pi^{*})$.
Hence, $\pi^{*}$ is a minimizer.

Now, let $\pi_{k}^{*}\in\text{argmin}_{\pi_{k}\in\Pi(\mu_{k},\nu_{k})}\mathcal{C}(\pi_{k})$.
Since $\mu_{k}$ and $\nu_{k}$ are each convergent sequences with
respect to $W_{2}$, by Theorem \ref{thm:W2-topology}
it holds that the sequence $\pi_{k}^{*}$ is precompact with respect
to the $W_{2}$ topology on $\mathcal{P}_{2}(\mathcal{X}\times\mathcal{Y})$;
so as before, we can extract a convergent subsequence (not relabeled)
such that $W_{2}(\pi_{k}^{*},\pi^{*})\rightarrow0$ for some limit
point $\pi^{*}$; furthermore, $\pi^{*}\in\Pi(\mu,\nu)$.

We deduce that $\mathcal{C}(\pi_{k}^{*})\rightarrow\mathcal{C}(\pi^{*})$;
it remains to show that $\pi^{*}\in\text{argmin}_{\pi\in\Pi(\mu,\nu)}\mathcal{C}(\pi)$.
To this end, let $\pi\in\Pi(\mu,\nu)$ be arbitrary. In view of Lemma
\ref{lem:glueing-convergence}, we can produce a sequence $\pi_{k}\in\Pi(\mu_{k},\nu_{k})$
such that $W_{2}(\pi_{k},\pi)\rightarrow0$. Note that this means
that $\mathcal{C}(\pi_{k})\rightarrow\mathcal{C}(\pi)$, again by Theorem \ref{thm:W2-topology}. Of course,
$\mathcal{C}(\pi_{k}^{*})\leq\mathcal{C}(\pi_{k})$ since $\pi_{k}^{*}$
is a minimizer. Therefore, (passing to the same subsequence as for
$\pi_{k}^{*}$) we have 

\[
\mathcal{C}(\pi^{*})=\lim_{k\rightarrow\infty}\mathcal{C}(\pi_{k}^{*})\leq\lim_{k\rightarrow\infty}\mathcal{C}(\pi_{k})=\mathcal{C}(\pi).
\]
Since $\pi\in\Pi(\mu,\nu)$ was arbitrary, this shows that
$\pi^{*}$ is a minimizer as desired.
\end{proof}

\subsubsection*{Proof of Theorem \ref{theo:sample_complexity}}
Let's first prove this proposition:
\begin{proposition}\label{prop:conv_P_Pi} Let $\X,\Y\subset\R^d$, $\tilde{p}_0:\pi\in\P_2(\X\times\Y)\mapsto \Sigma^\pi\in \R^{d\times d}$ and  $\tilde{P}_\Pi^{\mu,\nu}\coloneqq \tilde{p}_0(\Pi(\mu,\nu))$.  Let $\mu,\mu_0\in\mathcal{P}_2(\X)$ and $\nu,\nu_0\in\mathcal{P}_2(\Y)$.
We assume $\mathfrak{m}_2(\mu_{0}),\mathfrak{m}_2(\nu_{0}),\mathfrak{m}_2(\mu),\mathfrak{m}_2(\nu)\leq \alpha$ for some $\alpha\geq1$. Then the Hausdorff distance $\mathcal{H}_2$ (as defined in \eqref{eq:hauss_distance}) satisfies:
\begin{equation}\label{eq:hauss_conver_W_2}
    \mathcal{H}_2(\tilde{P}_\Pi^{\mu_0,\nu_0},\tilde{P}_\Pi^{\mu,\nu})\leq 2\sqrt{2} \sqrt{\alpha }(W_2 (\mu_{0},\mu)+W_2(\nu_{0},\nu) ).
\end{equation}
\end{proposition}

\begin{proof}[Proof of Proposition \ref{prop:conv_P_Pi}]
    Let $f_{i,j}: \mathbf{z}=(x,y)\in\X\times\Y\mapsto  x^iy^j\in\R$ be the product of the $i$-th and $j$-th component and $\pi_1\in \Pi(\mu_0,\nu_0),\pi_2\in  \Pi(\mu,\nu)$. Let $\Theta^*\in \Pi(\pi_1,\pi_2)\subset\P_2((\X\times\Y)^2)$ be  the optimal coupling for the $W_2$ distance between $\pi_1$ and $\pi_2$. Then for  $\mathbf{z_1},\mathbf{z_2}\in \X\times\Y$, let's rewrite $\sum_{i,j}|f_{i,j}(\mathbf{z_1})-f_{i,j}(\mathbf{z_2})|$ as:
    \begin{equation*}
        \begin{split}
             \sum_{i,j}|x_1^iy_1^j-x_2^iy_2^j|\leq \sum_{i,j}|x_1^iy_1^j-x_1^iy_2^j|+|x_1^iy_2^j-x_2^iy_2^j|
            \leq \sum_{i,j} |x_1^i\|y_1^j-y_2^j|+|y_2^j\|x_1^i-x_2^i|.
        \end{split}
    \end{equation*}
Now :
    \begin{equation*}
        \begin{split}
            ||\Sigma^{\pi_1}-\Sigma^{\pi_2}||_F^2 \leq & \sum_{i,j}\left|\int_{\X\times\Y} f_{i,j} d\pi_1-\int_{\X\times\Y} f_{i,j}d\pi_2\right|^2\\
            =& \sum_{i,j}\left|\int_{(\X\times\Y)^2} f_{i,j}(\mathbf{z_1})-f_{i,j}(\mathbf{z_2}) d\Theta^*\right|^2\\
            \leq & \sum_{i,j}\left|\int_{(\X\times\Y)^2} \left|f_{i,j}(\mathbf{z_1})-f_{i,j}(\mathbf{z_2})\right| d\Theta^*\right|^2\\
            =&\sum_{i,j}\left|\int_{(\X\times\Y)^2} |x_1^i||y_1^j-y_2^j|+|y_2^j||x_1^i-x_2^i| d\Theta^*\right|^2\\
            \leq & 2 \sum_{i,j}\left|\int_{(\X\times\Y)^2} |x_1^i||y_1^j-y_2^j| d\Theta^*\right|^2+\left|\int_{(\X\times\Y)^2}|y_2^j||x_1^i-x_2^i| d\Theta^*\right|^2.
        \end{split}
    \end{equation*}
    We use the Cauchy Schwarz inequality, such that 
    \begin{equation*}
        \begin{split}
            ||\Sigma^{\pi_1}-\Sigma^{\pi_2}||_F^2 \leq &2 \sum_{i,j}\int_{(\X\times\Y)^2}|x_1^i|^2d\Theta^*\int_{(\X\times\Y)^2}|y_1^j-y_2^j|^2d\Theta^*\\ 
            &+\int_{(\X\times\Y)^2}|y_2^j|^2d\Theta^*\int_{(\X\times\Y)^2}|x_1^i-x_2^i|^2d\Theta^*\\
            \leq &  2 \int_{(\X\times\Y)^2}||\mathbf{z}||^2d\Theta^*\int_{(\X\times\Y)^2}||\mathbf{z}_1-\mathbf{z}_2||^2d\Theta^*\\
            \leq &  8\alpha\;W_2(\pi_1,\pi_2)^2,\\
        \end{split}
    \end{equation*}
    where in the last inequality, we have used the fact that $ \int \|\mathbf{z}\|_2^2 d\Theta^* =\mathfrak{m}_2 (\pi_1)+\mathfrak{m}_2(\pi_2) =   \mathfrak{m}_2(\mu_0)+\mathfrak{m}_2(\nu_{0})+\mathfrak{m}_2(\mu)+\mathfrak{m}_2(\nu)\leq 4 \alpha$.  
    
    Finally, for $x\in \tilde{P}_\Pi^{\mu_0,\nu_0}$, we can write $x=\Sigma^{\pi_1}$ for $\pi_1\in \Pi(\mu_0,\nu_0)$. Applying  Lemma \ref{lem:glueing-convergence} we can find $\pi_2\in  \Pi(\mu,\nu)$ such that $W_2(\pi_1,\pi_2)\leq W_2(\mu,\mu_0)+W_2(\nu,\nu_0)$. If we denote $y=\Sigma^{\pi_2}\in \tilde{P}_\Pi^{\mu,\nu}$, we have 
    \[\|x-y\|\leq 2\sqrt{2} \sqrt{\alpha }(W_2 (\mu_{0},\mu)+W_2(\nu_{0},\nu) ).\]
    Similarly, for $\tilde y\in \tilde{P}_\Pi^{\mu,\nu}$, we can find $\tilde  x\in \tilde{P}_\Pi^{\mu_0,\nu_0}$ such that 
    \[\|\tilde x-\tilde y\|\leq 2\sqrt{2} \sqrt{\alpha }(W_2 (\mu_{0},\mu)+W_2(\nu_{0},\nu) ),\]
    and \eqref{eq:hauss_conver_W_2} holds.
\end{proof}

\begin{proof}[Proof of Theorem \ref{theo:sample_complexity}]
    This proof follows from Proposition \ref{prop:conv_P_Pi}. Given  $\pi_1\in \Pi(\mu_0,\nu_0),\pi_2\in  \Pi(\mu,\nu)$ we have: 
    \begin{equation*}
            \|\Sigma^{\pi_1}\|^2_F =\sum_{i,j}\left(\int x^iy^jd\pi_1\right)^2  \leq\sum_{i,j}\int |x^i|^2d\pi_1\int |y^i|^2d\pi_1= \mathfrak{m}_2(\mu_0)\mathfrak{m}_2(\nu_0)\leq   \alpha^2.
    \end{equation*}
    Similarly, $\|\Sigma^{\pi_2}\|^2_F\leq   \alpha^2$. The Lipschitz constant of the polynomial $Q_m$ on the ball $\B_0^{\R^d}(\alpha)$ can be upper bounded by $c_0\alpha^{n-1}$, with $c_0$ only depending on the coefficients of $Q_m$ and $n$. As an example we can take $c_0=nc_1$, with $c_1$ the $\ell^1$ norm of  the coefficients of $Q_m$. Thus,   $|Q_m(\Sigma^{\pi_1})-Q_m(\Sigma^{\pi_2})|\leq 2\sqrt{2}  n c_1 \alpha^{n-1/2}(W_2 (\mu_0,\mu)+W_2(\nu_{0},\nu) ).$ 
    Similarly,
    \begin{equation*}
            \|\tilde{\Sigma}^{\mu}\|^2_F =\sum_{i,j}\left(\int x^ix^jd\mu\right)^2  \leq\sum_{i,j}\int |x^i|^2d\mu\int |x^j|^2d\mu=\left(\int \|x\|^2d\mu\right)^2\leq   \alpha^2.
    \end{equation*}
    We can also apply the same computations as in the proof of Proposition \ref{prop:conv_P_Pi} to get:
    \[\|\tilde{\Sigma}^{\mu}-\tilde{\Sigma}^{\mu_0}\|_F\leq \sqrt{2} \; W_2(\mu_0,\mu)\sqrt{\int \|x\|^2d\mu+\int \|y\|^2d\mu_0}\leq 2\sqrt{\alpha}  W_2(\mu_0,\mu).\]
    Thus: \[|Q_m(\tilde{\Sigma}^{\mu})-Q_m(\tilde{\Sigma}^{\mu_0})|\leq 2  n c_1 \alpha^{n-1/2}W_2 (\mu_{0},\mu) \quad\text{and}\quad |Q_m(\tilde{\Sigma}^{\nu})-Q_m(\tilde{\Sigma}^{\nu_0})|\leq 2  n c_1 \alpha^{n-1/2}W_2 (\nu_{0},\nu).\] 
     But we have
     \begin{align*}
         |\GW_m(\mu_0,\nu_0)^2-\GW_m(\mu,\nu)^2|\leq& 2|Q_m(\Sigma^{\pi_1})-Q_m(\Sigma^{\pi_2})|\\&+|Q_m(\tilde{\Sigma}^{\mu})-Q_m(\tilde{\Sigma}^{\mu_0})|+|Q_m(\tilde{\Sigma}^{\nu})-Q_m(\tilde{\Sigma}^{\nu_0})|\\
         \leq&  n c_1\alpha^{n-1/2}(4\sqrt{2}+2)\left[(W_2 (\mu_0,\mu)+W_2(\nu_{0},\nu)\right],
     \end{align*}
     and we deduce \eqref{eq:GW_conver_W_2}.
\end{proof}

We now cover the example given in Remark after Theorem \ref{theo:sample_complexity}. Under the assumption that $\mu,\nu$ have finite $q$ moments for $q$ high enough ($q>4$ when $d\leq 2$ or $q>2d/(d-2)$ otherwise), \cite[Theorem 1]{fournier2015rate} gives us :
\[\mathbb{E}[W_2(\hat{\mu}_k,\mu)]\leq \sqrt{\mathbb{E}[W_2^2(\hat{\mu}_k,\mu)]}\leq c_0 \phi_k,\]
with  $\phi_k = k^{-\frac{1}{\max(d,4)}}\ln(1+k)^{\frac{1}{2}\1_{\{d=4\}}}$ and $c_0$ depending on the $q$-th moment of $\mu$ and the dimension only. We refer the reader to \cite{weed2019sharp} for the convergence of the empirical distance in different settings. Using the result of Theorem \ref{theo:sample_complexity}, we conclude:
\[\mathbb{E}[|\GW_m(\hat{\mu}_k,\hat{\nu}_k)^2-\GW_m(\mu,\nu)^2|]\leq c \alpha^{n-1/2} \phi_k.\]
Here the constant $c$ only depends on the $q$-th moment of the marginals, the multilinear form $m$ and the dimension $d$.

\subsubsection*{Proof of Lemma \ref{lem:Monge_DGW}}
\begin{proof}
    As shown previously, $\GW_m(\mu,\nu)^2=Q_m(\tilde{\Sigma}^\mu)+Q_m(\tilde{\Sigma}^\nu)+2 \inf_{\pi\in\Pi(\mu,\nu)} -Q_m(\Sigma^\pi)$. 
    We now focus on the last term. Let $\pi^*$ be an optimal transport plan, then it minimizes $\pi\to-Q_m(\Sigma^\pi)$. Thus $\Sigma^{\pi^*}$ minimizes the concave function $-Q_m$ on the compact convex set $\tilde{P}_\Pi\subset V$. Without loss of generality, we can select $\Sigma\in \tilde{P}_\Pi$ that minimizes $-Q_m$ and that is an extreme point of $\tilde{P}_\Pi$. In particular, there exists a vector $M\in V$, such that $\Sigma$ is the unique solution of the problem (left side):
    \begin{equation}
        \label{eq:monge_simpli}
        \min_{X\in \tilde{P}_\Pi}\langle M,X\rangle=\inf_{\pi\in\Pi(\mu,\nu)}\int\langle x,M y\rangle d\pi(x,y).
    \end{equation}
    Now, we can follow the proof of  \cite[Theorem 3.2]{dumont2025existence}, adatped below, to ensure that there exists a Monge plan $\pi^M$ that solves the right side problem \ref{eq:monge_simpli}. But the uniqueness of the solution of the left side equation ensures $\Sigma=\Sigma^{\pi^M}$ and that $\pi^M$ is a solution of the problem $\inf_{\pi\in\Pi(\mu,\nu)} -Q_m(\Sigma^\pi)$.

    We now prove the existence of the Monge plan $\pi^M$. For the sake of completeness we include the adapted proof and used results below.
    We note \[\inf_{\pi\in\Pi(\mu,\nu)}\int-\langle x,M y\rangle d\pi(x,y)=\inf_{\pi\in\Pi(\mu,\nu)}\int C(x, y)d\pi(x,y).\]
    Using a singular value decomposition, we write $M=O_1^T \Sigma_0 O_2\in\R^{d\times d}$ with $(O_1,O_2)\in O_d(\R)^2$ orthogonal matrices and $\Sigma_0\in\R^{d\times d}$ diagonal with non-negative entries. The cost $C$ then becomes
        \begin{equation*}
        C(x, y)=-\langle O_2^T\Sigma_0O_1 x,y\rangle =-\langle \Sigma_0O_1 x, O_2 y\rangle.
        \end{equation*}
        
    Using Lemma \ref{lemma:reparam}, the problem becomes an optimal transportation problem between $\mu'\coloneqq O_{1\sharp}\mu$ and $\nu'\coloneqq O_{2\sharp}\nu$. Up to permutation of the rows and columns of $O_1$ and $O_2$, we can assume that the non-zero eigenvalues of $\Sigma_0$ are $\sigma_1 \geq \dots \geq \sigma_h>0$, with $h\coloneqq \mathrm{rank}(M^*)\leq d$. The problem therefore becomes $\min_{\tilde\pi} \langle c_{\Sigma_0}, \tilde\pi\rangle$ for $\tilde\pi \in \Pi(\mu', \nu')$, where $c_{\Sigma_0}(\tilde x,\tilde y)=-\sum_{i=1}^h \sigma_i \tilde x_i \tilde y_i\eqqcolon \tilde c( p(\tilde x), p(\tilde y))$, and $p$ is the orthogonal projection on $\R^h$.

    The existence of an optimal map $T_0$ between $\mu'$ and $\nu'$ follows from the application of Theorem \ref{theo:fibers-main} for $E\coloneqq E_0\coloneqq \R^d=\R^h\times\R^{d-h}\eqqcolon B_0\times F$ and $\varphi\coloneqq p$. Indeed, $B_0$ and $F$ are complete Riemannian manifolds, the cost $\tilde c$ satisfies the twist condition on $B_0\times B_0$, $p_\sharp\mu'\ll\Leb_h$, and $\mu'_u\ll\Leb_{d-h}$ as a conditional probability for a.e.~$u$. One can then induce an optimal map between $\mu$ and $\nu$ by composing with $O_1$ and $O_2^T$ (Lemma \ref{lemma:reparam}).\\
    
    Now, one has that $c_{\Sigma_0}(x,y)=-\langle \tilde\Sigma_0x,y\rangle$, where $\tilde\Sigma_0=\mathrm{diag}({\sigma_i})_{1\leq i\leq h}$. As $p_\sharp\mu'$ has a density, we can apply Lemma \ref{lemma:scaled-Brenier} stated below with $(\psi_1,\psi_2)=(\tilde\Sigma_0,\mathrm{id})$ to obtain the existence of a unique optimal transport plan between $p_\sharp\mu'$ and $p_\sharp\nu'$ for the cost $c_{\Sigma_0}$.
    
\end{proof}

\begin{theorem}[Theorem 2.4 of \cite{dumont2025existence}, truncated]
    \label{theo:fibers-main}
    Let $E_0$ be a measurable space and $B_0$ and $F$ be complete Riemannian manifolds.
    Let $\mu,\nu \in \P(E_0)$ be two probability measures with compact support.
    Assume that there exists a set $E\subset E_0$ such that $\mu(E) = 1$ and that there exists a measurable map $\Phi : E \to B_0 \times F$ that is injective and whose inverse, when restricted to its image, is measurable as well.
    Let $p_B$ and $p_F$ denote the projections of $B_0 \times F$ on $B_0$ and $F$ respectively, and
    let $\varphi\coloneqq p_ B \circ \Phi: E\to B_0$.
    Let $c: E_0 \times E_0 \to \R$ and suppose that there exists a cost function $\tilde c: B_0 \times B_0 \to \R$ satisfying the twist condition:
    \[ \forall x_0\in B_0,\quad  y\mapsto \nabla_x \tilde c(x_0,y)\in T_{x_0}B_0 \text{ is injective,}\]
    such that
    \begin{equation*}
    c(x,y)=\tilde c(\varphi(x),\varphi(y))\quad \text{ for all } (x,y)\in E_0\times E_0.
    \end{equation*}
    Assume that $\varphi_\sharp\mu\ll\mathrm{vol}_{B_0}$. Let $\mu' \coloneqq \Phi_\sharp\mu$ and $\nu' \coloneqq \Phi_\sharp\nu$.
    Suppose that there exists a disintegration $(\mu'_u)_{u\in B_0}$ of $\mu'$ by $p_B$ such that for $\varphi_\sharp\mu$-almost every $u$, $\mu'_{u}\ll\mathrm{vol}_F$.\\
    Then there exists an optimal map $T:E\to E$ between $\mu$ and $\nu$ for the cost $c$.
\end{theorem}

\begin{lemma}[Lemma 3.3 of \cite{dumont2025existence}]
    \label{lemma:reparam}
    Let $E,F$ be Polish spaces, $\mu,\nu\in \P(E)$ and let $\psi_1,\psi_2:E\to F$ be homeomorphisms. Let $\tilde c:F\times F\to\R$ and consider the cost $c(x,y)= \tilde{c}( \psi_1(x),\psi_2(y))$. Then a map is optimal for the cost $c$ between $\mu$ and $\nu$ if and only if it is of the form $\psi_2^{-1}\circ T\circ\psi_1$ with $T:F\to F$ optimal for the cost $\tilde c$ between $\psi_{1\sharp}\mu$ and $\psi_{2\sharp}\nu$.
\end{lemma}

\begin{lemma}[Lemma 3.4 of \cite{dumont2025existence}]
    \label{lemma:scaled-Brenier}
    Let $h\geq 1$ and $\mu,\nu\in \P(\R^h)$ with compact supports and such that $\mu\ll\Leb_h$. Consider the cost $c(x, y)= -\langle \psi_1(x),\psi_2(y)\rangle$ where $\psi_1,\psi_2:\R^h\to \R^h$ are diffeomorphisms.
    Then there exists a unique optimal transport plan between $\mu$ and $\nu$ for the cost $c$, and it is induced by a map $t:\R^h\to \R^h$ of the form $t=\psi_2^{-1}\circ\nabla f\circ\psi_1$, with $f:\R^h\to\R$ convex.
\end{lemma}
\subsubsection*{Proof of Lemma \ref{lem:Monge_DGW_2}}
\begin{proof}
Note that Assumptions on $\mu,\nu$ imply that they  follow assumption \ref{assump_full_dim}, for some $R,\lambda>0$. Let $\pi^*$ be an optimal transport plan for the \ref{eq:DGW} problem. The first-order optimality condition gives us that $\pi^*$ solves in particular:
\[\inf_{\pi\in\Pi(\mu,\nu)}\int-\langle x,\mathrm{com}(\Sigma^{\pi^*}) y\rangle d\pi(x,y).\]

Using the proof of Lemma \ref{lemma:P_Pi_volume} and the fact that the marginals are centered, there exists a centered non empty ball included in the projection $P_\Pi$. In particular, $\inf_{\pi\in\Pi(\mu,\nu)} -\det(\Sigma^\pi)\neq 0$.  Then $\Sigma^{\pi^*}$ has full rank. We can now follow the proof of Lemma \ref{lem:Monge_DGW}: the matrix $M$ is given by $-\text{com}(\Sigma^{\pi^*})$ (from the proof of Example \ref{ex:3_csots_poly}), which has also full rank. Thus $\text{com}(\Sigma^{\pi^*})$ has full rank and Lemma \ref{lemma:scaled-Brenier} guarantees the uniqueness of the transport plan and the existence of the Monge map.

\end{proof}
\section{Proof of Section \ref{sec:algo}}\label{sec:supp_algo} 
\subsection{Proof of Section \ref{sec:bounding_boxes}}\label{sec:supp_algo_conv}

\subsubsection*{Auxiliary results}

We first need to estimate the volume of the projection $P_\Pi$:
\begin{lemma}[Dimension of $P_\Pi$]\label{lemma:P_Pi_volume}
    Under Assumption \ref{assump_full_dim}, there exists $r_0, r_1>0$ depending only on $\lambda,d,R$ such that $\mathcal{B}^V_{\bar{\mu}\bar{\nu}^T}\left(r_0\right)\subset P_\Pi\subset \mathcal{B}^V_{0}(r_1)$.  Thus, $\mathrm{dim}(P_\Pi)= d^2$. 
\end{lemma}
Following Corollary \ref{coro:reformulation_GW_m}, we can formulate the problem on two different spaces, $P_\Pi$ and $\tilde{P_\Pi}$. Since we are following Assumption \ref{assump_full_dim}, the map $\mathcal{R}_f$ is only scaling the vectors by a factor $r_R\coloneqq\int_{\mathcal{B}^{\R^d}_0(R)}x_1^2\; dx_1\dots dx_d>0$, and the results of this section apply to both representation (with different hidden constants).
\begin{proof}[Proof of Lemma \ref{lemma:P_Pi_volume}] Let $\X= \B_0^{\R^d}(R)$. For the first assertion, we can assume that $\mu,\nu$ are centered without loss of generality. Let's  first note that $\Sigma^{\mu\times\nu}=p_0(\mu\times\nu)=0$, so $0\in P_\Pi$.
    We will prove that there exists $r_0>0$ such that for all $v\in V$ unit vector, $\max_{y\in P_\Pi}\langle v,y\rangle\geq r_0$. That implies $\mathcal{B}^V_{0}\left(r_0\right)\subset P_\Pi$, which we are proving now by contradiction. Let's assume that $\mathcal{B}^V_{0}\left(r_0\right)\not\subset P_\Pi$ and define $x_0=\argmin_{x\in \partial P_\Pi}\|x\|_2$. Then $\|x_0\|_2<r_0$ and $x_0$ is a critical point of the squared norm on $\partial P_\Pi$, which has gradient $2x_0$. By convexity of $P_\Pi$, $x_0$ is included in the supporting hyperplane of $P_\Pi$ with normal $\frac{x_0}{\|x_0\|_2}$, i.e. \[ \forall x\in P_\Pi,\; \langle x, \frac{x_0}{\|x_0\|_2}\rangle\leq\|x_0\|_2<r_0.\]

    We now choose $v\in V$ unit vector, $(e_i)_i$ an orthonormal basis of $\R^d$, and define $(f_{i,j})_{i,j}$ an orthonormal basis of $V$ by $f_{i,j}(x_1,x_2)=\frac{1}{r_R}\langle e_i,x_1\rangle\langle e_j,x_2\rangle$.   Let $A\in \R^{d\times d}$ be the coordinates of $v$ on $(f_{i,j})_{i,j}$, then $\|A\|_F=1$. Now we can define:
    \[\sup_{w\in P_\Pi}\langle v,w\rangle=\sup_{\pi\in \Pi(\mu,\nu)}\frac{1}{r_R}\int x^TAy\;d\pi\eqqcolon g_{\mu,\nu}(A).\]
    Since $\Sigma^{\mu\times\nu} = 0$, we have $g_{\mu,\nu}(A)\geq0$ and we prove by contradiction that $g_{\mu,\nu}(A)>0$. Let's assume that $g_{\mu,\nu}(A)=0$, then $\mu\times\nu$ is an optimal coupling for the linear cost $(x,y)\to -\langle x,Ay\rangle$. Thus, the support of $\mu\times\nu$ is c-cyclically monotone \cite[Theorem 1.38]{santambrogio2015optimal}, meaning that: \[\forall(x_1,y_1), (x_2,y_2)\in \mathrm{supp(\mu\times\nu)}, \quad x_1^TAy_1+x_2^TAy_2\geq x_1^TAy_2+x_2^TAy_1,\]
     which implies that \[\forall x_1,x_2\in \mathrm{supp(\mu)}, (y_1,y_2)\in \mathrm{supp(\nu)},\quad (x_1-x_2)^TA(y_1-y_2)=0.\]
     Now, the assumptions on the second-moment matrices of the marginals $\mu,\nu$ imply that their respective support have non-zero volume, thus $(x_1-x_2)^TA(y_1-y_2)=0$ holds for all vectors in a basis of $\R^d$ and $A=0$, which is absurd.

    Let $B_{R,\lambda}$ be the set of all centered probability distributions following Assumption \ref{assump_full_dim}. Let's prove a more general result:
    \[\exists r_2>0,\quad\forall\eta,\xi\in B_{R,\lambda}, \;M\in\partial\mathcal{B}^{\R^{d\times d}}_{0} (1),\quad  g_{\eta,\xi}(M)\geq r_2.\]
    The unit sphere of $\R^{d\times d}$ is compact and since $\mathcal{B}^V_{0}(R)$ is compact, so is $A_{R,\lambda}$. Now, the map \[\eta,\xi,M\to g_{\eta,\xi}(M)=\frac{1}{r_R}\sup_{\pi\in\Pi(\eta,\xi)}\int  x^T Ay\;d\pi\]
    is continuous (from Proposition \ref{prop:stability-general}) and it attains its minimum on $ B_{R,\lambda}\times B_{R,\lambda}\times\partial\mathcal{B}^{\R^{d\times d}}_{0} (1)$. Let $(\eta,\xi,M)$ be minimizer of $g$, we have seen that $r_2\coloneqq g_{\eta,\xi}(M)>0$.

    We now prove that $P_\Pi\subset\mathcal{B}^V_{0}(\frac{R^2}{r_R})$. It follows that, for a coupling $\pi\in\Pi(\mu,\nu)$, with $\mu,\nu$ following Assumption \ref{assump_full_dim}, and $v$ the aforementioned unit vector, \[\sup_{w\in P_\Pi}\langle v,w\rangle=\sup_{\pi\in \Pi(\mu,\nu)}\frac{1}{r_R}\int x^TAy\;d\pi\leq\frac{1}{r_R} \|A\|_\infty\left(\int |x|^2\;d\mu\int |y|^2\;d\nu\right)^{1/2}\leq \frac{R^2}{r_R}.\]
\end{proof}
In this proposition, we describe how to initialize the bounding boxes to ensure a minimal size to $P_\Pi^{-}$.

\begin{proposition}[Bounding box initialization]\label{prop:bounding_box}
    Let's suppose that there exist $r_0,r_1>0$ and $x_0\in V$ such that $\mathcal{B}^V_{x_0}(r_0)\subset P_\Pi\subset\mathcal{B}^V_{0}(r_1)$. Algorithm \ref{alg:bounding_box}  constructs bounding boxes $P_{\Pi,0}^-\subset P_\Pi\subset P_{\Pi,0}^+$ with the following properties: there exists $x_1\in V$ and there exists $r_2>0$ only depending on $d,r_0,r_1$ such that: \[\mathcal{B}^V_{x_1}(r_2)\subset P_{\Pi,0}^-\quad\mathrm{and}\quad P_{\Pi,0}^+\subset\mathcal{B}^V_{x_1}((2d+1)r_1).\]
\end{proposition}

\begin{proof}[Proof of Proposition \ref{prop:bounding_box}]
    For each iteration $k>1$, let's define $v^+=\max(|\langle g^*_{ k}-g^*_{+1},e_k\rangle|, |\langle g^*_{ -k}-g^*_{+1},-e_k\rangle|)$. The assumption $\mathcal{B}^V_{x_0}(r_0)\subset P_\Pi$ implies $v^+\geq\frac{r_0}{2}$.
    If we write $\mathrm{vol}_k$ the $k$ dimensional unsigned volume created by the convex hull of a set of points, then  $\mathrm{vol}_k(S^2_{k})=\frac{v^+}{k}\mathrm{vol}_{k-1}(S^2_{k-1})$ by construction. Note that $\mathrm{vol}_1(S^2_{1})\geq r_0$. Finally, $S^2_{d^2}$ has  a $d^2$ dimensional volume of at least $r_0^{d^2}/(2^{d^2}d^2!)$.
    By construction $S^+_{d^2}=P_{\Pi,0}^+$ is a rectangular bounding box contained in a ball of radius at most $dr_1$. Now, $P_{\Pi,0}^+\subset\mathcal{B}^V_{x}((2d+1)r_1)$ for all $x\in \mathcal{B}^V_{0}(r_1)$.
    Let's now prove that there exists $x_1\in \mathcal{B}^V_{0}(r_1)$ and $r_2>0$ (depending only on $r_0,d$ and $r_1$) such that $\mathcal{B}^V_{x_1}(r_2)\subset S^2_{d^2}\subset P_{\Pi,0}^-$.\newline
    Let \[A =\left\{(x_1,\dots, x_{d^2+1})\in \mathcal{B}^V_{0}(r_1)^{d^2+1}, \mathrm{vol}_{d^2}(x_1,\dots, x_{d^2+1})\geq \frac{r_0^{d^2}}{2^{d^2}d^2!}\right\},\]
    then this is  a non-empty compact set, and if we denote $(p_1,\dots, p_{d^2+1})$ the vertices of $S^2_{d^2}$, then $(p_1,\dots, p_{d^2+1})\in A$. Given $\mathbf{x}=(x_1,\dots, x_{d^2+1})\in V^{d^2+1}$, we define: \[B^-(\mathbf{x})= \max(r,(r,x)\in[0,2r_1]\times \mathcal{B}^V_{0}(r_1) \text{ and }  \mathcal{B}^V_{x}(r)\subset \mathrm{conv}(x_1,\dots, x_{d^2+1})).\]
    Here conv stands for the convex hull of the points.
    Then $B^-$ is a continuous map and $B^-(A)$ is also a non-empty compact set. Let's assume that $\min_{\mathbf{x}\in A} B^-(x)=0.$ Then there exists $(x_1,\dots, x_{d^2+1})\in \mathcal{B}^V_{0}(r_1)^{d^2+1}$, with an empty interior and satisfying $\mathrm{vol}_{d^2}(x_1,\dots, x_{d^2+1})\geq \frac{r_0^{d^2}}{2^{d^2}d^2!}$, which is absurd. Finally, we take $r_2=\min_{\mathbf{x}\in A} B^-(\mathbf{x})>0$ depending only on $r_0,d$ and $r_1$, and select $x_1$ such that $ \mathcal{B}^V_{x_1}(r_2)\subset S^2_{d^2}\subset  P_{\Pi,0}^-$.
\end{proof}

\subsubsection*{Proof of Theorem \ref{theo:convergence_rate}}
\begin{proof}
    Let $\mathcal{H}_2$ be the Hausdorff distance for polytopes in $V$ as defined in \eqref{eq:hauss_distance}. This result is a direct corollary of Lemma 4.9 of \cite{lammel2023convergence}, which states that, if there exist $r_0,r_1>0$ and $x_0\in V$ such that  $\mathcal{B}^V_{x_0}(r_0)\subset P_{\Pi,0}^-$ and $P_{\Pi,0}^+\subset\B_{x_0}(r_1)$, then 
    \[\mathcal{H}_2(P_{\Pi,k}^+,P_{\Pi,k}^-)\leq c_0k^{-1/(d^2-1)},\;\mathrm{for}\; k\geq d^2+1,\]
    with $c_0$ depending on $r_0,r_1$ and $d$ only. Now, Proposition \ref{prop:bounding_box} and Lemma \ref{lemma:P_Pi_volume} enable us to select suitable values for $r_0,r_1$ as a function of $d$, $R$ and $\lambda$ only.
\end{proof}

\subsection{Proofs of Section \ref{sec:algo_glob_concave}}\label{sec:sup_alg_theo}

\subsubsection*{Proof of Theorem \ref{theo:FPTAS}}
\begin{proof}
    We now go over the complexity of our algorithm. Let's fix $\epsilon>0$ and $d\geq2$ a fixed number (typically $d=2$ or $3$). If the marginals $\mu,\nu$ are discrete with at most $N$ points in the marginals, solving an optimal transport problem has complexity $O(N^3\ln(N)$ \cite{peyre2019computational}. From Theorem \ref{theo:convergence_rate}, there exists $c_0$ depending on $d$, $R$ and $\lambda$ only such that:
    \[\mathcal{H}_2(P_{\Pi,k}^+,P_{\Pi,k}^-)\leq c_0k^{-1/(d^2-1)},\;\mathrm{for}\; k\geq d^2+1.\]
    But  $P_{\Pi,0}^-\subset P_{\Pi}\subset P_{\Pi,0}^+\subset \mathcal{B}^V_{0}(4dR^2/r_R)$ and $Q_m$ is $\beta-$Lipschitz for some $\beta>0$ on $\mathcal{B}^V_{0}(4dR^2/r_R)$. Thus, \[|\min_{x\in P_{\Pi,k}^-}-Q_m(x)-\min_{y\in P_{\Pi,k}^+}-Q_m(y)|\leq \beta\;\mathcal{H}_2(P_{\Pi,k}^+,P_{\Pi,k}^-).\]
    This implies that the number $k_f$ of iterations in Algorithm \ref{alg:main_GW_m_concave} to reach a precision $\epsilon$ in cost is at most $k_f=O\left((\frac{1}{\epsilon})^{(d^2-1)}\right)$, where the implied constant does not depend on the marginals, only on $d$, $R$, $\beta$ and $\lambda$. Let's study the complexity of our algorithm. We can initialize the bounding boxes with a finite number of cuts (each time an optimal transport problem is solved) as seen in Algorithm \ref{alg:bounding_box_cvx}.  At each iteration $k$, we solve one optimal transport problem (line \ref{alg_cv_lin_problem}) and add a vertex to $P_\Pi^-$ and a constraint to $P_\Pi^+$, thus their number of respective vertices/constraints is $O(k)$. A consequence of the  Upper Bound Theorem is that the number of faces (of any dimension) of a $d^2$ dimensional polytope with $k$ vertices can be upper bounded by $2^{d^2+1} \binom{k}{\lfloor d^2/2\rfloor}\leq 2^{d^2+1} k^{\lfloor d^2/2\rfloor}$ \cite[Section 5.5]{matouvsek2002lectures}. Hence, the number of respective vertices/constraints in $P_\Pi^+$ and $P_\Pi^-$ has order $O(k^{\lfloor d^2/2\rfloor})$.

    At each iteration of our algorithm:
    
    \begin{itemize}
        \item The optimal direction for the new cut (line \ref{alg_cv_dir}) is selected by solving \eqref{eq:Hausdorff_direc}. This can be done by enumerating the pairs constraints/vertices in $P_\Pi^-$/$P_\Pi^+$ and has a complexity $O(k^{2\lfloor d^2/2\rfloor})$.
        \item An optimal transport problem with marginals $\mu,\nu$ is solved (line \ref{alg_cv_lin_problem}), which has complexity $O(N^3\ln(n))$ \cite{peyre2019computational}.
        \item The $V$-representation of $P_\Pi^-$ and $H$-representation of $P_\Pi^+$ are updated, with complexity $O(1)$.
        \item The $H$-representation of $P_\Pi^-$ and $V$-representation of $P_\Pi^+$ are updated, with complexity $O(k^{2\lfloor d^2/2\rfloor+1})$ (See next paragraph).
        \item A new optimal cost for $-Q_m$ on $P_\Pi^-,P_\Pi^+$ is computed. This can  be done by enumerating vertices of the polytopes, i.e. with complexity  $O(k^{\lfloor d^2/2\rfloor})$.
    \end{itemize}
    
     We now elaborate on how the $V$-representation of $P_\Pi^+$ is updated. We are using the data structure described in \cite[Appendix A3]{ryner2023globally}, and use their notation. Let $E$ be the set of extreme points of  $P_\Pi^+$ (its $V$-representation), it has size $O(k^{\lfloor d^2/2\rfloor})$. Let $A$ be the set of constraints of  $P_\Pi^+$ (its $H$-representation), it has size $O(k)$. We also need an adjacency matrix $D$ between the vertices of $P_\Pi^+$. It represents the faces of the polytope of dimension $1$ and has thus size $O(k^{\lfloor d^2/2\rfloor})$. We define a binary matrix $B$ linking vertices and constraints: if vertex $i$ satisfies constraint $j$ $B_{i,j}=1$, else $B_{i,j}=0$. $B$ has size $O(k^{\lfloor d^2/2\rfloor+1})$. When a new constraint is added:
     \begin{itemize}
         \item $A$ is updated with complexity $O(1)$.
         \item Infeasible extreme points are computed using the matrix product $A.E$, with complexity $O(k^{\lfloor d^2/2\rfloor+1})$.
         \item Their adjacent vertices are selected, with complexity $O(k^{\lfloor d^2/2\rfloor})$.
         \item Each adjacency leads to computing at most one new vertex with complexity $O(1)$.
         \item $E$ is updated with complexity $O(k^{\lfloor d^2/2\rfloor})$.
         \item $B$ is updated with complexity $O(k^{\lfloor d^2/2\rfloor+1})$.
         \item  Finally, $D$ has to be updated, which can be done by selecting pairs of  newly created vertices solving the same $d^2-1$ constraints. In other words, each pair of  rows newly added to $B$ is compared, with a total complexity of  $O(k^{2\lfloor d^2/2\rfloor+1})$.
     \end{itemize}

     The $H$-representation of $P_\Pi^-$ is updated in a similar manner, using a dual polytope $P_D^-$. Adding a point to $P_\Pi^-$ is equivalent to adding a constraint to $P_D^-$, which is handled as previously described. Going back and forth between the primal and dual representation consists in enumerating constraints or vertices, and has complexity at most $O(k^{\lfloor d^2/2\rfloor})$.

     Finally, the complexity of running iteration $k$ is $O(N^3\ln(N)+k^{2 \lfloor d^2/2\rfloor+1})$ and the final complexity is $O(N^3\ln(N)k_f+k_f^{2 \lfloor d^2/2\rfloor+2})$. In terms of memory, the largest set is the matrix $B$ and the transport plan needs to be stored, needing a space of size $O(N^2+k_f^{\lfloor d^2/2\rfloor+1})$.
\end{proof}

\newpage
\subsubsection*{Complexity for solving the 2-Gromov-Wasserstein distance}

As we have seen in Section \ref{sec:classical_gw}, the classical Gromov-Wasserstein 2 distance $\GW(\mu,\nu)$ can be computed by solving $\inf_{x\in P_\Pi} -Q(x)$, with $P_\Pi$ a compact convex set of dimension at most $d^2+1$ and $-Q$ a concave polynomial with $d^2+1$ variables. Unlike Assumption \ref{assump_full_dim}, there is no simple assumption enabling us to lower bound the volume of $P_\Pi$ (in order to apply Proposition \ref{prop:bounding_box} and the proof of Theorem \ref{theo:convergence_rate}). As an example, if $\mu,\nu$ are supported on centered unit spheres of $\R^d$, $\int gd\pi=1$ for all $\pi\in\Pi(\mu,\nu)$ and $P_\Pi$ has zero volume. To remedy this problem, one can first compute the dimension of $P_\Pi$ by iteratively computing supporting hyperplanes and extreme points of $P_\Pi$ in orthogonal directions, as represented in Algorithm \ref{alg:P_Pi_dim}.

\begin{algorithm}
\caption{Computing dimension of $P_\Pi$ }\label{alg:P_Pi_dim}
\begin{algorithmic}[1]
\State $g_0\gets p_0(\mu\otimes\nu)$
\State $\tilde{V},\hat{V}\gets \emptyset,\emptyset$
\For{$k\in[1,d^2+1]$}
\State Choose $e_k$ unit vector of $V$ orthogonal to $\tilde{V}$ and $\hat{V}$
\State Solve problem \eqref{eq:lin_problem} and compute $\hat{g}_{\pm }, g^*_{\pm }$ with direction $g=\pm e_k$
\State $\hat{V}\gets\hat{V}\cup \{g^*_{+ }-g_0,g^*_{- }-g_0\}$
\If{$\hat{g}_{+ }=-\hat{g}_{- }$ ($P_\Pi $ has zero width in direction $e_k$)}
\State $\tilde{V}\gets\tilde{V} \cup \{e_k\}$
\EndIf
\EndFor

\State\Return $(\hat{e})$ orthonormal basis of $\hat{V}$
\end{algorithmic}
\end{algorithm}

This algorithm outputs the base of a vector space $\hat{V}\subset V$ such that $P_\Pi\subset \hat{V}$ and $P_\Pi$ has non-zero volume in this space.  We can then run an algorithm similar to Algorithm \ref{alg:main_GW_m_concave} on the space $\hat{V}$. The dimension reduction from $V$ to $\hat{V}$ ensures that we can apply Lemma 4.9 of \cite{lammel2023convergence} similarly to the proof of Theorem \ref{theo:convergence_rate}, stating that  the number of  iterations $k$ needed to approximate the optimal cost  of $-Q$ on $P_\Pi$ with precision $\epsilon$ can be upper bounded as $k=O \left(\frac{1}{\epsilon^{d^2}}\right)$, with the constant depending on $d$,  and  the marginals $\mu,\nu$. Algorithm \ref{alg:classical_GW_2} represents our strategy to compute the 2-Gromov-Wasserstein distance. 

\begin{algorithm}
\caption{Algorithm to approximate the 2-$\GW$ cost in dimension $d$}\label{alg:classical_GW_2}
\begin{algorithmic}[1]
\Require $\epsilon>0, \mu,\nu$.
\State Compute $C_0=\int\|x_1-x_2\|^4d\mu^{ 2}+\int\|y_1-y_2\|^4d\nu^{ 2} - 4 \int\|y\|^2d\nu\int\|x\|^2d\mu$
\State Define $f_{i,j}\coloneqq (x,y)\in \R^d\times\R^d \mapsto x^iy^j$ and $V=\mathrm{span}((f_{i,j})_{i,j})$
\State Compute $({\hat {e}}_{i,j})_{i,j}$  orthonormal basis of $\hat{V}$ (Algorithm \ref{alg:P_Pi_dim}).
\State \textbf{We run the following steps on $\hat{V}$}
\State Define bounding boxes $P_\Pi^+\coloneqq\{\}$,$P_\Pi^-\coloneqq\{\}$, box cost $c^- = +\infty$, $c^+ = -\infty$ and optimal direction $\bar{g}=\emptyset$
 \State  \textbf{Initialization of the bounding boxes: }  $P_\Pi^-,P_\Pi^+,c^-, \bar{g}$ (Algorithm \ref{alg:bounding_box_cvx})
\State $c^+,x^+\gets \mathrm{concave\_min}(P_\Pi^+, -Q)$ 

 \State  \textbf{Updating bounding box until convergence }

\While{$|c^+ -c^-|>\epsilon$}
    \State Solve problem \eqref{eq:Hausdorff_direc} (alternatively \eqref{eq:concave_direction}) given $x^+$ and compute optimal direction $g$ 
    \State Solve problem \eqref{eq:lin_problem} and compute $\hat{g},g^*$ 
    \State $P_\Pi^-\gets \mathrm{conv} (P_\Pi^-\cup g^*)$  
    \State $P_\Pi^+\gets P_\Pi^+\cap H(g)$ 
    \State $c^+,x^+\gets \mathrm{concave\_min}(P_\Pi^+, -Q)$ 
    \If{$c^->-Q(g^*)$}
    \State $\bar{g}\gets g$
    \State $c^-\gets -Q(g^*)$
    \EndIf
\EndWhile

\State  \textbf{Optimal coupling} $\pi^*$ solving convex problem \eqref{eq:final_coupling_cvx} given $\bar{g}$
\State\Return $\pi^*,C_0+2c^-$
\end{algorithmic}
\end{algorithm}

The complexity analysis of the proof of Theorem \ref{theo:FPTAS} can also be applied, changing $d^2$ by $d^2+1$, which yields:
\begin{theorem}\label{theo:FPTAS_GW_2}
    Let $\mu,\nu\in\P(\R^d)$ be discrete measures with bounded support. Algorithm \ref{alg:classical_GW_2} enables us to approximate $\GW_2(\mu,\nu)^2$ at a precision $\epsilon>0$:

    - with complexity $O\left(\left(\frac{1}{\epsilon}\right)^{2(d^2)(\lfloor (d^2+1)/2\rfloor+1)}\right)$,

    - with  memory of size $O\left(\left(\frac{1}{\epsilon}\right)^{(d^2)(\lfloor (d^2+1)/2\rfloor+1)}\right)$,\\
    where the implied constants depend on $d$ and the distributions $\mu,\nu$.
\end{theorem}

\subsection{Proofs of Section \ref{sec:algo_loc}}\label{sec:sup_algo_loc}
\begin{proof}[Frank-Wolfe convergence rate for concave functions]
Let $P \subset \R^{n}$ be nonempty, compact and convex, and let
$f$ be concave and differentiable on an open set containing $\mathcal{D}$.
We consider the Frank-Wolfe Algorithm \cite{frank1956algorithm} to minimize $f$ on $D$.
Choose $x_0 \in \mathcal{D}$ and, for $k = 0,1,2,\dots$, set
\[
  x_{k+1} \in \argmin_{s \in \mathcal{D}} \langle \nabla f(x_k), s \rangle.
\]
We define the  Frank-Wolfe gap as
\[
  g_k \coloneqq \max_{s \in \mathcal{D}} \langle \nabla f(x_k), x_k - s \rangle.
\]
Now $g_k \ge 0$ and, $g_k = 0$ if and only if
$\langle \nabla f(x_k), s - x_k \rangle \ge 0$ for all $s \in \mathcal{D}$,
i.e.\ $x_k$ is a first-order stationary point.
For every $k \ge 0$, we have by concavity
\[
  f(x_{k+1}) \le f(x_k) + \langle \nabla f(x_k), s_k - x_k \rangle
  = f(x_k) - g_k. 
\]
Finally, for every $K \ge 0$,
\[
  \sum_{k=0}^{K} g_k \le f(x_0) - f(x_{K+1}) \le f(x_0) - \min_{s \in \mathcal{D}}f(s)
  \qquad\text{and}\qquad
  \min_{0 \le k \le K} g_k \le \frac{f(x_0) - \min_{s \in \mathcal{D}}f(s)}{K+1}.
\]
Thus the algorithm converges at rate $O\left(\frac{1}{k}\right)$.

\end{proof}

\newpage
\section{Convexity of the different costs}\label{sec:supp_cost_convexity}
In this section we go over the convexity properties of the different costs. We are looking at the polynomial cost $-Q_m$ over the domain $P_\pi$. For the Inner Gromov-Wasserstein distance, $Q_m =\|\cdot\|^2_F $, which is convex. For the Determinant Gromov-Wasserstein distance, $Q_m =\det $, which is non-convex non-concave. Let's go over the Chiral Gromov-Wasserstein distance, depending on the dimension $d$. For $f:\R^{d^2}\to\R$, we denote $H_e(f)$ its Hessian and $\eig(M)$ the set of eigenvalues of the matrix $M$. In the CGW case, $Q_m(x) = t\|x\|^2_F + (1-t) d! \det(x)$.
\begin{enumerate}
    \item For $d=1$, $x\in \R$ and $Q_m(x) = t\|x\|^2_F + (1-t) x$, which is convex for $0<t\leq1$.
    \item For $d=2$, $x\in \R^4$ and  $\eig(H_e(\|\cdot\|^2_F))=\{2\}$ and $\eig(H_e(\det(\cdot))=\{-1,1\}$. Thus, $Q_m$ is convex for $\frac{1}{2}\leq t\leq1$, else it is not.

    \item For $d=3$, we want to find a value $t$ such that $Q_m(x)=t\|x\|^2_F + 6 (1-t) \det(x)$ stays convex on $P_\pi$ i.e., such that $H_e(Q_m(x))$ is semi-positive definite. We assume that $\mathrm{supp}(\mu),\mathrm{supp}(\nu)\in \mathcal{B}_0^{\R^d}(R)$ for $R>0$. Now for $\pi\in\Pi(\mu,\nu)$,  \[\|\Sigma^{\pi}\|^2_F =\sum_{i,j}\left(\int x^iy^jd\pi\right)^2  \leq\sum_{i,j}\int |x^i|^2d\pi\int |y^i|^2d\pi=\int \|x\|^2d\mu\int \|y\|^2d\nu\leq R^4,\]
    
    Let $x,y\in\R^{3\times3}$ and $s\in\R$. Then the determinant expands as follows: 
    \[
    \det(x+sy)=\det x+s\,\mathrm{tr}(\mathrm{com}(x)^Ty)
    +s^2\langle x,\mathrm{com}(y)^T\rangle_F+s^3\det y ,
    \]
    where $\mathrm{com}(y)$ is the cofactor matrix of $y$ and $\langle \cdot,\cdot\rangle_F$ the Frobenius inner product.
    We can now compute the quadratic form associated with the Hessian as:
    \[
    |\langle y,H_e(\det(x))y\rangle|
    =2\,|\langle x,\mathrm{com}(y)\rangle_F| \leq 2 \|x\|_F\;\|\mathrm{com}(y)\|_F .
    \]
    Let's upper bound $\|\mathrm{com}(y)\|_F$ for $\|y\|_F \leq R^2$. 
    The entries of $\mathrm{com}(y)$ are, up to sign, the $2\times2$ minors of $y$, so
    $\|\mathrm{com}(y)\|_F=\|C_2(y)\|_F$, where $C_2(y)$ is the second compound matrix of $y$.
    Let $y=U\Sigma V^T$ be a singular value decomposition with singular values
    $\sigma_1,\sigma_2,\sigma_3\ge 0$. By the Cauchy--Binet formula,
    $C_2(y)=C_2(U)\,C_2(\Sigma)\,C_2(y)^T$ from the multiplicative property of the second compound matrix. Moreover, $C_2(U)$ and $C_2(V)$ are
    orthogonal, and $C_2(\Sigma)=\mathrm{diag}(\sigma_1\sigma_2,\sigma_1\sigma_3,\sigma_2\sigma_3)$.
    Hence
    \begin{align*}
        \|\mathrm{com}(y)\|_F^2=&\sigma_1^2\sigma_2^2+\sigma_1^2\sigma_3^2+\sigma_2^2\sigma_3^2 \\=&\frac{1}{3} (\sigma_1^2+\sigma_2^2+\sigma_3^2)-\frac{1}{6}( (\sigma_1^2-\sigma_2^2)^2+(\sigma_1^2-\sigma_3^2)^2+(\sigma_2^2-\sigma_3^2)^2)\\
        \leq &\frac{1}{3} (\sigma_1^2+\sigma_2^2+\sigma_3^2) = \frac{1}{3}\|y\|_F^2.
    \end{align*}
    Note that this inequality is tight as, if we note $I$ the identity matrix:
    \[\|\mathrm{com}\left(\frac
    {1}{\sqrt{3}}I\right)\|_F^2=\frac{1}{3}=\frac{1}{3}\|I\|_F^2\]
    We now go back to $Q_m$:
    \[\langle y,H_e(Q_m(x))y\rangle\geq 2t \|y\|_F^2-6(1-t)|\langle y,H_e(\det(x))y\rangle|\geq (2t-\frac{12R^2}{\sqrt{3}}(1-t))\|y\|_F^2,\]
    And this last term is non-negative for $t\geq\frac{2\sqrt{3}R^2}{1+2\sqrt{3}R^2}$.

\end{enumerate}
Here is a table recapitulating the costs properties and convexity of $-Q_m$:

\begin{table}[h!]
\centering
\begin{tabular}{|c| c| c| c|} 
 \hline
 Cost & Convexity & dimension & t \\ 
 \hline
 $\IGW$ & Concave & all &  \\ 
 \hline
 $\DGW$  & Non-convex, non-concave & all &  \\
 \hline
 $\CGW$  & Concave & 1 & $0<t<1$ \\
   & Concave & 2 & $\frac{1}{2}\leq t<1$ \\
   & Non-convex, non-concave & 2 & $0< t<\frac{1}{2} $ \\ 
  & Concave & 3 & $ \frac{2\sqrt{3}R^2}{1+2\sqrt{3}R^2}\leq t < 1$  \\
   & Non-convex, may not be concave & 3 & $0\leq t \leq \frac{2\sqrt{3}R^2}{1+2\sqrt{3}R^2}$ \\ 
 \hline
\end{tabular}
\caption{Convexity of the different costs, for marginals contained in a centered ball of radius $R$.}
\label{table:1}
\end{table}

\newpage
\section{Algorithm, concave cost}\label{sec:supp_alg_cvx_ex}

Here is a way to initialize the bounding boxes to ensure the convergence rate of Theorem \ref{theo:convergence_rate}:
\begin{algorithm}
\caption{Bounding box initialization}\label{alg:bounding_box}
\begin{algorithmic}[1]
\State Choose $e_1$ unit vector of $V$
\State Solve problem \eqref{eq:lin_problem} and compute $\hat{g}_{\pm 1}, g^*_{\pm 1}$ with direction $g=\pm e_1$ 
\State $S^1_{1} \gets \{x\in V, \langle x,  e_1\rangle\leq  \hat{g}_{1}\}\cap \{x\in V, \langle x,  -e_1\rangle\leq  \hat{g}_{-1}\}$
\State $S^2_{1}\gets \{g^*_{-1}, g^*_{+1}\} $
\State $S^3_{1} \gets \emptyset$
\For{$k\in[2,d^2]$}
   \State Choose $e_k$ unit vector of $V$ orthogonal to the affine subspace containing $S^2_{k-1}$
   \State Solve problem \eqref{eq:lin_problem} and compute $\hat{g}_{\pm k}, g^*_{\pm k}$ with direction $g=\pm e_k$ 
   \State $g^+, g^-\gets\argmax,\argmin(|\langle g^*_{ +k}-g^*_{+1},e_k\rangle|, |\langle g^*_{ -k}-g^*_{+1},-e_k\rangle|)$
   \State $S^1_{k} \gets S^1_{k-1}\cap\{x\in V, \langle x,  e_k\rangle\leq  \hat{g}_{+k}\}\cap \{x\in V, \langle x,  -e_k\rangle\leq  \hat{g}_{-k}\}$
   \State $S^2_{k} \gets S^2_{k-1}\cup g^+$
   \State $S^3_{k} \gets S^3_{k-1}\cup g^-$
\EndFor
\State $P_{\Pi,0}^+\gets S^1_{d^2}$ 
\State $P_{\Pi,0}^-\gets \mathrm{conv}(S^2_{d^2}\cup S^3_{d^2})$
\State \Return $P_{\Pi,0}^-,P_{\Pi,0}^+$
\end{algorithmic}
\end{algorithm}

Now we state the formulation of the bounding box initialization when $Q_m$ is concave, enabling one to know the optimal direction $g^-$ and cost $c^-$. 
\begin{algorithm}
\caption{Bounding box initialization}\label{alg:bounding_box_cvx}
\begin{algorithmic}[1]
\State Choose $e_1$ unit vector of $V$
\State Solve problem \eqref{eq:lin_problem} and compute $\hat{g}_{\pm 1}, g^*_{\pm 1}$ with directions $g=\pm e_1$ 
\State $c^-\gets \min (-Q_m(g^*_{ 1}),-Q_m(g^*_{ -1}))$
\If{$-Q_m(g^*_{ 1})\leq-Q_m(g^*_{ -1})$}
\State $\bar{g}\gets e_1$
\Else
 \State$\bar{g}\gets -e_1$
\EndIf
\State $S^1_{1} \gets \{x\in V, \langle x,  e_1\rangle\leq  \hat{g}_{1}\}\cap \{x\in V, \langle x,  -e_1\rangle\leq  \hat{g}_{-1}\}$
\State $S^2_{1}\gets \{g^*_{-1}, g^*_{+1}\} $
\State $S^3_{1} \gets \emptyset$
\For{$k\in[2,d^2]$}
   \State Choose $e_k$ unit vector of $V$ orthogonal to the affine subspace containing $S^2_{k-1}$
   \State Solve problem \eqref{eq:lin_problem} and compute $\hat{g}_{\pm k}, g^*_{\pm k}$ with direction $g=\pm e_k$ 
   \State $S^1_{k} \gets S^1_{k-1}\cap\{x\in V, \langle x,  e_k\rangle\leq  \hat{g}_{+k}\}\cap \{x\in V, \langle x,  -e_k\rangle\leq  \hat{g}_{-k}\}$
   \State $g^+\gets\argmax(|\langle g^*_{ +k}-g^*_{+1},e_k\rangle|, |\langle g^*_{ -k}-g^*_{+1},-e_k\rangle|)$
   \State $g^-\gets\argmin(|\langle g^*_{ +k}-g^*_{+1},e_k\rangle|, |\langle g^*_{ -k}-g^*_{+1},-e_k\rangle|)$
   \State $S^2_{k} \gets S^2_{k-1}\cup g^+$
    \State $S^3_{k} \gets S^3_{k-1}\cup g^-$
    \If{$c^-> \min (-Q_m(g^*_{ k}),-Q_m(g^*_{ -k}))$}
    \If{$-Q_m(g^*_{ k})\leq-Q_m(g^*_{ -k})$}
    \State $\bar{g}\gets e_k$
    \State $c^-\gets -Q_m(g^*_{ k})$
    \Else
    \State $\bar{g}\gets -e_k$
    \State $c^-\gets -Q_m(g^*_{ -k})$
    \EndIf
    \EndIf
\EndFor
\State $P_{\Pi,0}^+\gets S^1_{d^2}$ 
\State $P_{\Pi,0}^-\gets \mathrm{conv}(S^2_{d^2}\cup S^3_{d^2})$
\State\Return $P_{\Pi,0}^-,P_{\Pi,0}^+, c^-, \bar{g}$
\end{algorithmic}
\end{algorithm}
\newpage

We now describe a practical initialization for the bounding boxes. This method does not ensure the convergence rate of Theorem \ref{theo:FPTAS}, but it seems to yield good convergence in practical applications.
\begin{algorithm}
\caption{Bounding box initialization (alternative)}\label{alg:bounding_box_cvx_2}
\begin{algorithmic}[1]
\Require $e_1,\dots,e_{d^2}$ orthonormal basis of $V$.
\State $S_0^+\gets V$
\State $S_0^-\gets \emptyset$
\State $c^- = +\infty$
\For{$k\in[1,d^2]$}
    \State Solve problem \eqref{eq:lin_problem} and compute $\hat{g}_{ k}, g^*_{k}$ with direction $g= e_k$
    \State $S_k^+ \gets S_{k-1}^+\cap\{x\in V, \langle x,  e_k\rangle\leq  \hat{g}_{k}\}$
    \State $S_k^-\gets S_{k-1}^-\cup \{g^*_{k}\} $
    \If{$c^->-Q_m(g^*_{k})$}
        \State $\bar{g}\gets g$
        \State $c^-\gets-Q_m(g^*_{k})$
    \EndIf
\EndFor
\State Solve problem \eqref{eq:lin_problem} and compute $\hat{g}_{ d^2+1}, g^*_{d^2+1}$ with direction $g= -\frac{1}{d}\sum_{i=1}^{d^2} e_i$
\State $S_{d^2+1}^+ \gets S_{d^2}^+\cap\{x\in V, \langle x,  g\rangle\leq  \hat{g}_{d^2+1}\}$
\State $S_{d^2+1}^-\gets S_{d^2}^-\cup \{g^*_{d^2+1}\} $
 \If{$c^->-Q_m(g^*_{d^2+1})$}
    \State $\bar{g}\gets g$
    \State $c^-\gets-Q_m(g^*_{d^2+1})$
\EndIf
\State $P_{\Pi,0}^+\gets S_{d^2+1}^+$ 
\State $P_{\Pi,0}^-\gets \mathrm{conv}(S_{d^2+1}^-)$
\State\Return $P_{\Pi,0}^-,P_{\Pi,0}^+, c^-, \bar{g}$

\end{algorithmic}
\end{algorithm}

We present the algorithms used in practice to compute the $\GW_m$ cost, under the concavity assumption. The initialization of the bounding boxes and the choice of direction for the cut do not allow us to apply Theorem \ref{theo:FPTAS}, but the algorithm seems to perform similarly or better than Algorithm \ref{alg:main_GW_m} on simple examples.
\begin{algorithm}
\caption{Main algorithm for $\GW_m$ in dimension $d$, concave cost}\label{alg:parctical_concave}
\begin{algorithmic}[1]
\Require $\epsilon>0, \mu,\nu$.
\State Compute second-moment matrix $\tilde{\Sigma}^{\mu}$, $\tilde{\Sigma}^{\nu}$
\State Define:$f_{i,j}\coloneqq (x,y)\in \R^d\times \R^d\mapsto x^iy^j$
\State Define bounding boxes $P_\Pi^+\coloneqq\{\}$,$P_\Pi^-\coloneqq\{\}$, box cost $c^- = +\infty$, $c^+ = -\infty$ and optimal direction $\bar{g}=\emptyset$
 \State  \textbf{Initialization of the bounding boxes: }  $P_\Pi^-,P_\Pi^+,c^-, \bar{g}$ (Algorithm \ref{alg:bounding_box_cvx_2})
\State $c^+,x^+\gets \mathrm{concave\_min}(P_\Pi^+, -Q_m)$ 

 \State  \textbf{Updating bounding box until convergence }

\While{$|c^+ -c^-|>\epsilon$}
    \State Solve problem  \eqref{eq:concave_direction} given $x^+$ and compute optimal direction $g$ 
    \State Solve problem \eqref{eq:lin_problem} and compute $\hat{g},g^*$ 
    \State $P_\Pi^-\gets \mathrm{conv} (P_\Pi^-\cup g^*)$  
    \State $P_\Pi^+\gets P_\Pi^+\cap H(g)$ 
    \State $c^+,x^+\gets \mathrm{concave\_min}(P_\Pi^+, -Q_m)$ 
    \If{$c^->-Q_m(g^*)$}
    \State $\bar{g}\gets g$
    \State $c^-\gets -Q_m(g^*)$
    \EndIf
\EndWhile

\State  \textbf{Optimal coupling} $\pi^*$ solving convex problem \eqref{eq:final_coupling_cvx} given $\bar{g}$
\State (Optional) $\pi^*, c^-$ solving $\text{Frank-Wolfe}(m,\mu,\nu,\pi^*)$
\State\Return $\pi^*,Q_m(\tilde{\Sigma}^{\mu})+Q_m(\tilde{\Sigma}^{\nu})+2c^-$
\end{algorithmic}
\end{algorithm}

\newpage

\clearpage
\section{Frank-Wolfe line search for the IGW, DGW and CGW problems}\label{sec:supp_line_search}
In this section we explicitly describe the line search optimization step in the Frank-Wolfe algorithm \ref{alg:Frank_Wolfe}, with the $\IGW$ and $\DGW$ costs. The $\IGW$ line search can be written  \[T, \tau=\min,\argmin_{\tau\in[0,1]} -\|\tau\Sigma^{\hat{\pi}_{k+1}}+(1-\tau)\Sigma^{\pi_k}\|^2_F,\] which is concave, yielding solutions at 0 or 1. We have thus the following line search:

\begin{algorithm}
\caption{Line search for IGW}\label{alg:line_IGW}
\begin{algorithmic}[1]
\If{$\|\Sigma^{\hat{\pi}_{k+1}}\|^2_F>\|\Sigma^{\pi_k}\|^2_F$} 
\State $\tau\gets 1, \;T\gets -\|\Sigma^{\hat{\pi}_{k+1}}\|_F^2$
\Else
\State $\tau\gets 0,\; T\gets -\|\Sigma^{\pi_{k}}\|_F^2$
\EndIf
\State\Return $T,\tau$
\end{algorithmic}
\end{algorithm}

For the $\DGW$ cost, we want to solve \[T, \tau=\min,\argmin_{\tau\in[0,1]}-d!\det(\tau\Sigma^{\hat{\pi}_{k+1}}+(1-\tau)\Sigma^{\pi_{k}}).\]
In dimension 2, we use the fact that, for two square matrices of the same size,
\begin{equation*}
        \det(A+B)= \det(A)+\det(B)+[\mathrm{tr}(A)\mathrm{tr}(B)-\mathrm{tr}(AB)],
\end{equation*}
such that
\begin{equation}\label{eq:det_expansion}
    \begin{split}
        2\det(\tau\Sigma^{\hat{\pi}_{k+1}}+(1-\tau)\Sigma^{\pi_{k}})=& 2\tau^2\det(\Sigma^{\hat{\pi}_{k+1}})
        +2(1-\tau)^2\det(\Sigma^{\pi_{k}})\\ &+
        2\tau(1-\tau)[\mathrm{tr}(\Sigma^{\hat{\pi}_{k+1}})\mathrm{tr}(\Sigma^{\pi_{k}})-\mathrm{tr}(\Sigma^{\hat{\pi}_{k+1}}\Sigma^{\pi_{k}})].\\
    \end{split}
\end{equation}
We can write the objective \eqref{eq:det_expansion} as a polynomial $q(\tau)\coloneqq\alpha\tau^2+\beta\tau+\gamma$, that we want to maximize, with
\begin{equation*}
    \begin{split}
        \alpha &=2[\det(\Sigma^{\hat{\pi}_{k+1}})+\det(\Sigma^{\pi_{k}})-\mathrm{tr}(\Sigma^{\hat{\pi}_{k+1}})\mathrm{tr}(\Sigma^{\pi_{k}})+\mathrm{tr}(\Sigma^{\hat{\pi}_{k+1}}\Sigma^{\pi_{k}})],\\
        \beta  &=2[\mathrm{tr}(\Sigma^{\hat{\pi}_{k+1}})\mathrm{tr}(\Sigma^{\pi_{k}})-\mathrm{tr}(\Sigma^{\hat{\pi}_{k+1}}\Sigma^{\pi_{k}})-2\det(\Sigma^{\pi_{k}})],\\
        \gamma &=2\det(\Sigma^{\pi_{k}}).\\
    \end{split}
\end{equation*}
For the CGW cost with parameter $t$, we get $Q(\Sigma) = t\|\Sigma\|_F^2+2(1-t)\det(\Sigma)$ and similarly:

\begin{equation*}
    \begin{split}
        \alpha =&t[\|\Sigma^{\hat{\pi}_{k+1}}\|_F^2+\|\Sigma^{\pi_{k}}\|_F^2 -2\langle\Sigma^{\hat{\pi}_{k+1}},\Sigma^{\pi_{k}}\rangle_F]\\&+2(1-t)[\det(\Sigma^{\hat{\pi}_{k+1}})+\det(\Sigma^{\pi_{k}})-\mathrm{tr}(\Sigma^{\hat{\pi}_{k+1}})\mathrm{tr}(\Sigma^{\pi_{k}})+\mathrm{tr}(\Sigma^{\hat{\pi}_{k+1}}\Sigma^{\pi_{k}})],\\
        \beta  =&2t[\langle\Sigma^{\hat{\pi}_{k+1}},\Sigma^{\pi_{k}}\rangle_F-\|\Sigma^{\pi_{k}}\|_F^2]+2(1-t)[\mathrm{tr}(\Sigma^{\hat{\pi}_{k+1}})\mathrm{tr}(\Sigma^{\pi_{k}})-\mathrm{tr}(\Sigma^{\hat{\pi}_{k+1}}\Sigma^{\pi_{k}})-2\det(\Sigma^{\pi_{k}})],\\
        \gamma =&t[\|\Sigma^{\pi_{k}}\|_F^2]+2(1-t)[\det(\Sigma^{\pi_{k}})].\\
    \end{split}
\end{equation*}
Here is the code to maximize polynomial $q$ as a function of its parameters:
\begin{algorithm}
\caption{Line search for the DGW and CGW problems (2D)}\label{alg:line_DGW2}
\begin{algorithmic}[1]
\Require$q$
\If{$\alpha\geq0$}
\If{$q(1)> q(0)$}
\State $\tau = 1$
\Else 
\State $\tau = 0$
\EndIf
\Else
\State $\lambda\gets -\frac{\beta}{2\alpha}$
\State $\tau\gets \min(1,\max(0, \lambda)) $
\EndIf
\State\Return $q(\tau),\tau$
\end{algorithmic}
\end{algorithm}

Similarly, in dimension 3, we can compute the polynomial such that:
\[q(\tau)\coloneqq 6\det(\tau\Sigma^{\hat{\pi}_{k+1}}+(1-\tau)\Sigma^{\pi_{k}}) = \alpha\tau^3+\beta\tau^2+\gamma\tau+\delta  .\]
In practice, we can write:
\begin{multline*}
    \det(A+B)=\det(A)+\det(B)+\frac{1}{2}(\mathrm{tr}^2(A)-\mathrm{tr}(A^2))\mathrm{tr}(B)+\\\frac{1}{2}(\mathrm{tr}^2(B)-\mathrm{tr}(B^2))\mathrm{tr}(A)
    -(\mathrm{tr}(A)+\mathrm{tr}(B))\mathrm{tr}(AB)+\mathrm{tr}(A^2B+AB^2),
\end{multline*}
such that:
\begin{multline*}
    \det(\tau\Sigma^{\hat{\pi}_{k+1}}+(1-\tau)\Sigma^{\pi_{k}})=\tau^3\det(\Sigma^{\hat{\pi}_{k+1}})+(1-\tau)^3\det(\Sigma^{\pi_{k}})\\+
    \tau^2(1-\tau)\left[\frac{1}{2}(\mathrm{tr}^2(\Sigma^{\hat{\pi}_{k+1}})-\mathrm{tr}((\Sigma^{\hat{\pi}_{k+1}})^2))\mathrm{tr}(\Sigma^{\pi_{k}})+\mathrm{tr}((\Sigma^{\hat{\pi}_{k+1}})^2 \Sigma^{\pi_{k}})-\mathrm{tr}(\Sigma^{\hat{\pi}_{k+1}})\mathrm{tr}(\Sigma^{\hat{\pi}_{k+1}}\Sigma^{\pi_{k}})\right]\\ +\tau(1-\tau)^2\left[\frac{1}{2}(\mathrm{tr}^2(\Sigma^{\pi_{k}})-\mathrm{tr}((\Sigma^{\pi_{k}})^2))\mathrm{tr}(\Sigma^{\hat{\pi}_{k+1}})+\mathrm{tr}(\Sigma^{\hat{\pi}_{k+1}}(\Sigma^{\pi_{k}})^2) - \mathrm{tr}(\Sigma^{\pi_{k}})\mathrm{tr}(\Sigma^{\hat{\pi}_{k+1}}\Sigma^{\pi_{k}}).\right]
\end{multline*}
Finally, we have 
\begin{align*}
    \frac{\alpha}{6} =& \det(\Sigma^{\hat{\pi}_{k+1}}) - \det(\Sigma^{\pi_{k}}) \\&-\left[\frac{1}{2}(\mathrm{tr}^2(\Sigma^{\hat{\pi}_{k+1}})-\mathrm{tr}((\Sigma^{\hat{\pi}_{k+1}})^2))\mathrm{tr}(\Sigma^{\pi_{k}})+\mathrm{tr}((\Sigma^{\hat{\pi}_{k+1}})^2 \Sigma^{\pi_{k}})-\mathrm{tr}(\Sigma^{\hat{\pi}_{k+1}})\mathrm{tr}(\Sigma^{\hat{\pi}_{k+1}}\Sigma^{\pi_{k}})\right] \\&+\left[\frac{1}{2}(\mathrm{tr}^2(\Sigma^{\pi_{k}})-\mathrm{tr}((\Sigma^{\pi_{k}})^2))\mathrm{tr}(\Sigma^{\hat{\pi}_{k+1}})+\mathrm{tr}(\Sigma^{\hat{\pi}_{k+1}}(\Sigma^{\pi_{k}})^2) - \mathrm{tr}(\Sigma^{\pi_{k}})\mathrm{tr}(\Sigma^{\hat{\pi}_{k+1}}\Sigma^{\pi_{k}}).\right]
\end{align*}
\begin{align*}
    \frac{\beta}{6} =&  
    3\det(\Sigma^{\pi_{k}}) \\&
    +\left[\frac{1}{2}(\mathrm{tr}^2(\Sigma^{\hat{\pi}_{k+1}})-\mathrm{tr}((\Sigma^{\hat{\pi}_{k+1}})^2))\mathrm{tr}(\Sigma^{\pi_{k}})+\mathrm{tr}((\Sigma^{\hat{\pi}_{k+1}})^2 \Sigma^{\pi_{k}})-\mathrm{tr}(\Sigma^{\hat{\pi}_{k+1}})\mathrm{tr}(\Sigma^{\hat{\pi}_{k+1}}\Sigma^{\pi_{k}})\right] \\&
    -2\left[\frac{1}{2}(\mathrm{tr}^2(\Sigma^{\pi_{k}})-\mathrm{tr}((\Sigma^{\pi_{k}})^2))\mathrm{tr}(\Sigma^{\hat{\pi}_{k+1}})+\mathrm{tr}(\Sigma^{\hat{\pi}_{k+1}}(\Sigma^{\pi_{k}})^2) - \mathrm{tr}(\Sigma^{\pi_{k}})\mathrm{tr}(\Sigma^{\hat{\pi}_{k+1}}\Sigma^{\pi_{k}}).\right]
\end{align*}
\begin{equation*}
    \frac{\gamma}{6} = -3\det(\Sigma^{\pi_{k}}) +\left[\frac{1}{2}(\mathrm{tr}^2(\Sigma^{\pi_{k}})-\mathrm{tr}((\Sigma^{\pi_{k}})^2))\mathrm{tr}(\Sigma^{\hat{\pi}_{k+1}})+\mathrm{tr}(\Sigma^{\hat{\pi}_{k+1}}(\Sigma^{\pi_{k}})^2) - \mathrm{tr}(\Sigma^{\pi_{k}})\mathrm{tr}(\Sigma^{\hat{\pi}_{k+1}}\Sigma^{\pi_{k}}).\right]
\end{equation*}
\begin{equation*}
    \delta = 
    6\det(\Sigma^{\pi_{k}}). 
\end{equation*}
Finally, we want to maximize $q$ over [0,1], with 
\[\frac{dq}{d\tau}= 3\alpha\tau^2+2\beta\tau+\gamma.\]
The optimization is the following:

\begin{algorithm}
\caption{Line search for the DGW and CGW problems (3D)}\label{alg:FW_dgw}
\begin{algorithmic}[1]
\Require$q$
\State $A\gets\textbf{True}$
\State $\Delta\gets4\beta^2-12\alpha\gamma$
\If{$\Delta>0$}
\State $\lambda\gets\frac{-2\beta-\sqrt{\Delta}}{6\alpha}$
\If{$0\leq\lambda\leq 1$ \textbf{and }$q(\lambda)>q(0)$  \textbf{and }$q(\lambda)>q(1)$}
\State $\tau\gets \lambda$
\State $A\gets\textbf{False}$
\EndIf
\EndIf
\If{A}
\If{$q(1)\geq q(0)$}
\State $\tau = 1$
\Else
\State $\tau = 0$
\EndIf
\EndIf
\State\Return $q(\tau),\tau$
\end{algorithmic}
\end{algorithm}

Similarly for $\CGW(t)$:
\begin{multline*}
    \alpha = 6(1-t) [\det(\Sigma^{\hat{\pi}_{k+1}}) - \det(\Sigma^{\pi_{k}})] \\-6(1-t)\left[\frac{1}{2}(\mathrm{tr}^2(\Sigma^{\hat{\pi}_{k+1}})-\mathrm{tr}((\Sigma^{\hat{\pi}_{k+1}})^2))\mathrm{tr}(\Sigma^{\pi_{k}})+\mathrm{tr}((\Sigma^{\hat{\pi}_{k+1}})^2 \Sigma^{\pi_{k}})-\mathrm{tr}(\Sigma^{\hat{\pi}_{k+1}})\mathrm{tr}(\Sigma^{\hat{\pi}_{k+1}}\Sigma^{\pi_{k}})\right] \\+6(1-t)\left[\frac{1}{2}(\mathrm{tr}^2(\Sigma^{\pi_{k}})-\mathrm{tr}((\Sigma^{\pi_{k}})^2))\mathrm{tr}(\Sigma^{\hat{\pi}_{k+1}})+\mathrm{tr}(\Sigma^{\hat{\pi}_{k+1}}(\Sigma^{\pi_{k}})^2) - \mathrm{tr}(\Sigma^{\pi_{k}})\mathrm{tr}(\Sigma^{\hat{\pi}_{k+1}}\Sigma^{\pi_{k}}).\right]
\end{multline*}
\begin{multline*}
    \beta =  t[\|\Sigma^{\hat{\pi}_{k+1}}\|_F^2+\|\Sigma^{\pi_{k}}\|_F^2-2\langle\Sigma^{\hat{\pi}_{k+1}},\Sigma^{\pi_{k}}\rangle_F]+18(1-t)
    \det(\Sigma^{\pi_{k}}) \\
    +6(1-t)\left[\frac{1}{2}(\mathrm{tr}^2(\Sigma^{\hat{\pi}_{k+1}})-\mathrm{tr}((\Sigma^{\hat{\pi}_{k+1}})^2))\mathrm{tr}(\Sigma^{\pi_{k}})+\mathrm{tr}((\Sigma^{\hat{\pi}_{k+1}})^2 \Sigma^{\pi_{k}})-\mathrm{tr}(\Sigma^{\hat{\pi}_{k+1}})\mathrm{tr}(\Sigma^{\hat{\pi}_{k+1}}\Sigma^{\pi_{k}})\right] \\
    -12(1-t)\left[\frac{1}{2}(\mathrm{tr}^2(\Sigma^{\pi_{k}})-\mathrm{tr}((\Sigma^{\pi_{k}})^2))\mathrm{tr}(\Sigma^{\hat{\pi}_{k+1}})+\mathrm{tr}(\Sigma^{\hat{\pi}_{k+1}}(\Sigma^{\pi_{k}})^2) - \mathrm{tr}(\Sigma^{\pi_{k}})\mathrm{tr}(\Sigma^{\hat{\pi}_{k+1}}\Sigma^{\pi_{k}}).\right]
\end{multline*}
\begin{multline*}
    \gamma = 2t[\langle\Sigma^{\hat{\pi}_{k+1}},\Sigma^{\pi_{k}}\rangle_F-\|\Sigma^{\pi_{k}}\|_F^2]-18(1-t)
    \det(\Sigma^{\pi_{k}}) \\+6(1-t)
    \left[\frac{1}{2}(\mathrm{tr}^2(\Sigma^{\pi_{k}})-\mathrm{tr}((\Sigma^{\pi_{k}})^2))\mathrm{tr}(\Sigma^{\hat{\pi}_{k+1}})+\mathrm{tr}(\Sigma^{\hat{\pi}_{k+1}}(\Sigma^{\pi_{k}})^2) - \mathrm{tr}(\Sigma^{\pi_{k}})\mathrm{tr}(\Sigma^{\hat{\pi}_{k+1}}\Sigma^{\pi_{k}}).\right]
\end{multline*}
\begin{equation*}
    \delta = 
    t\|\Sigma^{\pi_{k}}\|_F^2+6(1-t)\det(\Sigma^{\pi_{k}}).
\end{equation*}

In this case we can also use Algorithm \ref{alg:FW_dgw}.
\newpage

\section{Supplementary Numerical Experiments}\label{sec:supp_num_exp}

In this section we present results from numerical experiments studying the convergence rate of our algorithms. In all Figures \ref{fig:conv_glo_2}, \ref{fig:conv_FW} and \ref{fig:bench_ryner}, marginal points were randomly drawn from a standard Gaussian distribution and assigned with the same weight.

In Figure \ref{fig:conv_glo_2}, we evaluated the complexity in the number of vertices and constraints in respectively the upper and lower bounding box as well as the certificate of global optimality $|c^+-c^-|$. Our results suggest, in both dimension 2 and 3 that the complexities in the size of the bounding box are below the theoretical complexity evaluated in Theorem \ref{theo:FPTAS}.

In Figure \ref{fig:bench_ryner}, we compared two methods (ours and Ryner's \cite{ryner2023globally}) for selecting the new bounding box cut. Both strategies have similar convergence rates, and their relative performances depend on the number of points and the desired precision. Typically, Ryner's method tends to alternate between plateaus and large improvements of the cost, but the algorithm may stagnate indefinitely, probably due to machine imprecision. In contrast, our method tends to have a marginally slower convergence, with a continuous decrease in cost. In the experiments we performed, the number of vertices in the outer bounding box behaved differently, between Ryner's and our method. In the first case this number is quadratic as a function of the iterations, whereas it is linear in our case, allowing for a larger number of iterations for a given memory.  Furthermore, we note that our method converges exactly in finite (exponential) time for discrete measures with a finite number of points, since the inner bounding box recovers all vertices of $P_\Pi$, whereas such a guarantee does not hold using Ryner's method. This is especially useful for marginals with a low number of points.

Figure \ref{fig:conv_FW} represents the convergence of Algorithm \ref{alg:Frank_Wolfe} in terms of cost (squared cost error) or of transport plan (plan error). In both cases of the \ref{eq:IGW} and the \ref{eq:DGW} costs, the algorithm converges in less than 30 iterations.

Figure \ref{fig:conv_FW_entrop} represents the convergence of Algorithm \ref{alg:Frank_Wolfe} with entropy regularized optimal transport, for the \ref{eq:IGW} and the \ref{eq:DGW}  case, in both 2D and 3D.

\newpage

\section{Supplementary Figures}\label{sec:supp_figures}

\begin{figure}[ht!]
    \centering
    \includegraphics[width=0.4\linewidth]{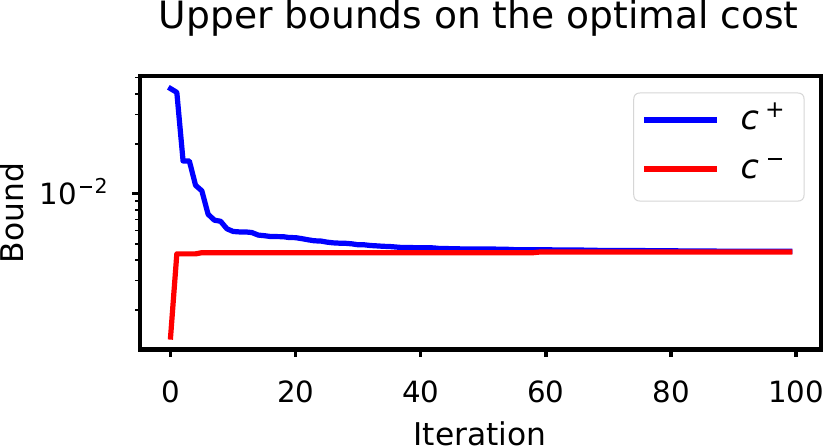}
    \caption{ Lower and upper bound convergence for the global optimization concave algorithm, when computing the $\CGW$ cost in Figure \ref{fig:hand}A. More precisely, $c^-$ and $c^+$ bound the optimal value of $Q_m$ on $P_\Pi$. The lower bound $c^-$ converges faster than the upper bound $c^+$, which is a recurrent behavior in our simulations.  }
    \label{fig:supp_conv_certif}
\end{figure}

\newpage

\begin{figure}[ht!]
    \centering
    \includegraphics[width=\linewidth]{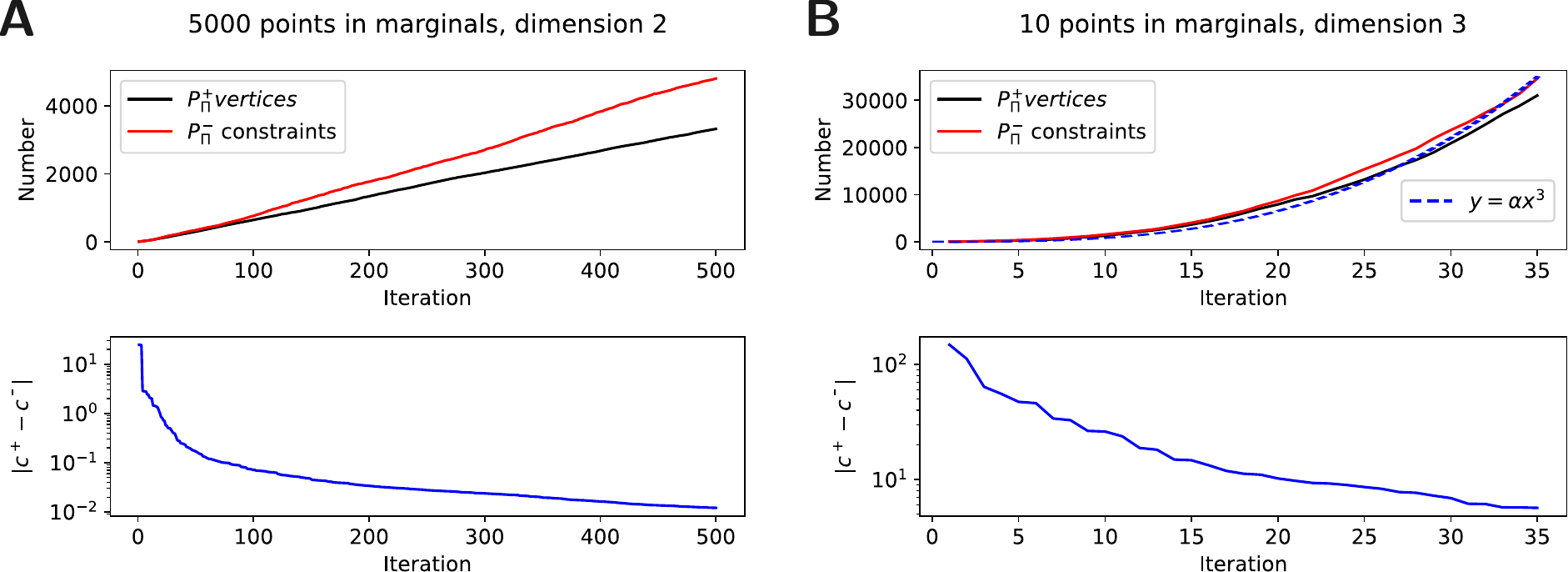}
    \caption{Convergence of the global optimization concave algorithm for the \ref{eq:IGW} cost. A) Convergence in dimension 2 with 5000 points in each marginal. An upper bound on the number of vertices in $P_\Pi^+$ and the number of constraints in $P_\Pi^-$ is $O(k^2)$, with $k$ the number of iterations. In practice, those numbers have a lower complexity, with a linear behavior in  this case. $|c^+-c^-|$ is an upper bound of the precision at which the global solution is approximated (see algorithm \ref{alg:main_GW_m_concave}). In dimension 2, an ordinary laptop can run several thousand iterations of the algorithm without suffering memory issues. B) Convergence in dimension 3, with 10 points in each marginal.  An upper bound on the number of vertices in $P_\Pi^+$ and the number of constraints in $P_\Pi^-$ is $O(k^4)$, with $k$ the number of iterations. The blue dashed line represents the function $y=\alpha x^3$ with $\alpha=1.3$, suggesting that the number of vertices and constraints has size $O(k^3)$ in practice. }
    \label{fig:conv_glo_2}
\end{figure}

\newpage
\begin{figure}[ht!]
    \centering
    \includegraphics[width=\linewidth]{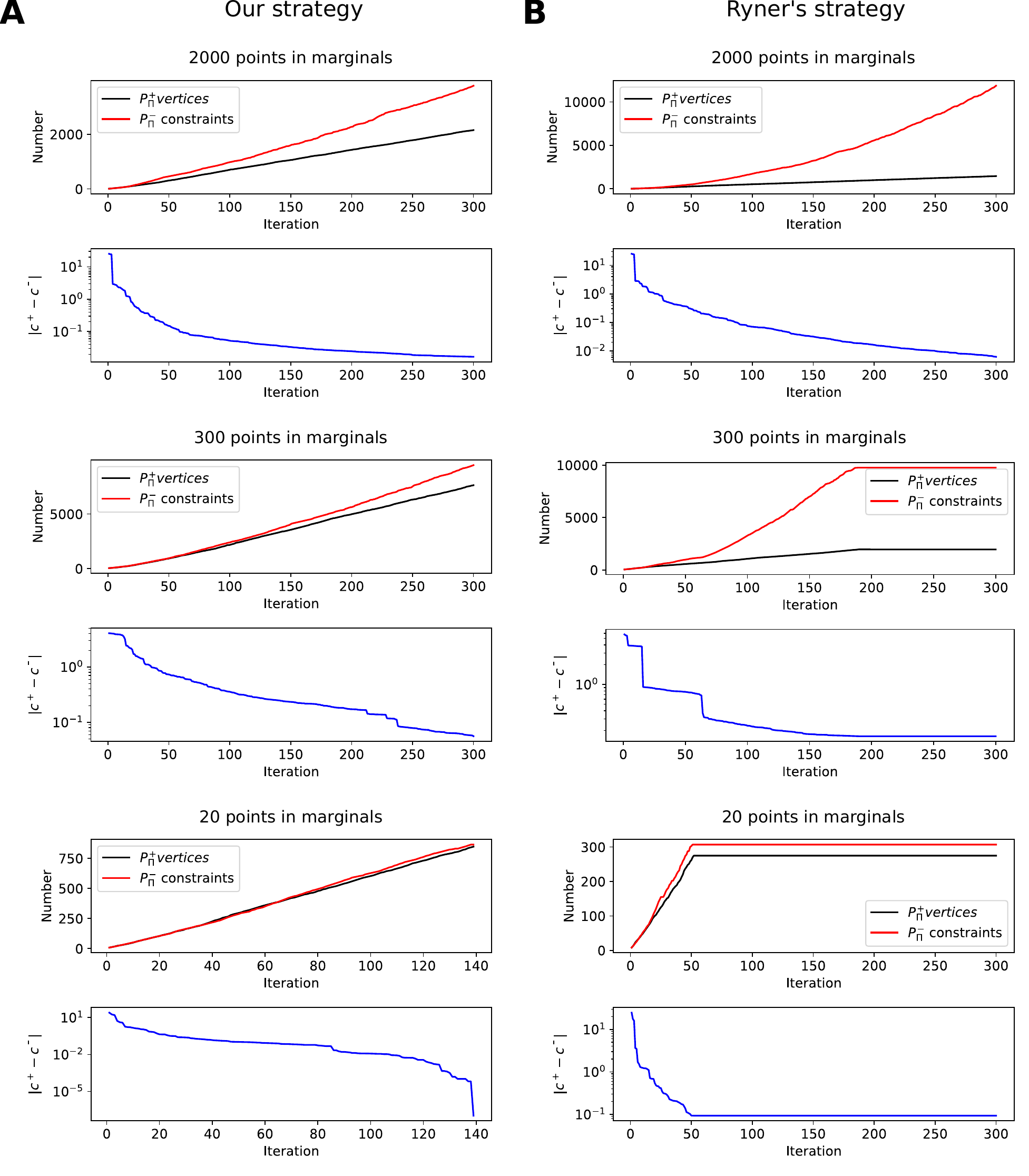}
    \caption{Convergence of the global optimization concave algorithm for the \ref{eq:IGW} cost, using two different strategies for selecting the search direction at each iteration. A) Convergence using our strategy, using the $P_\Pi^-$ facet facing $x^+$. B) Convergence of Ryner's strategy \cite{ryner2023globally}, using the gradient of $Q_m$ at $x^+$.}
    \label{fig:bench_ryner}
\end{figure}

\newpage
\begin{figure}[ht!]
    \centering
    \includegraphics[width=\linewidth]{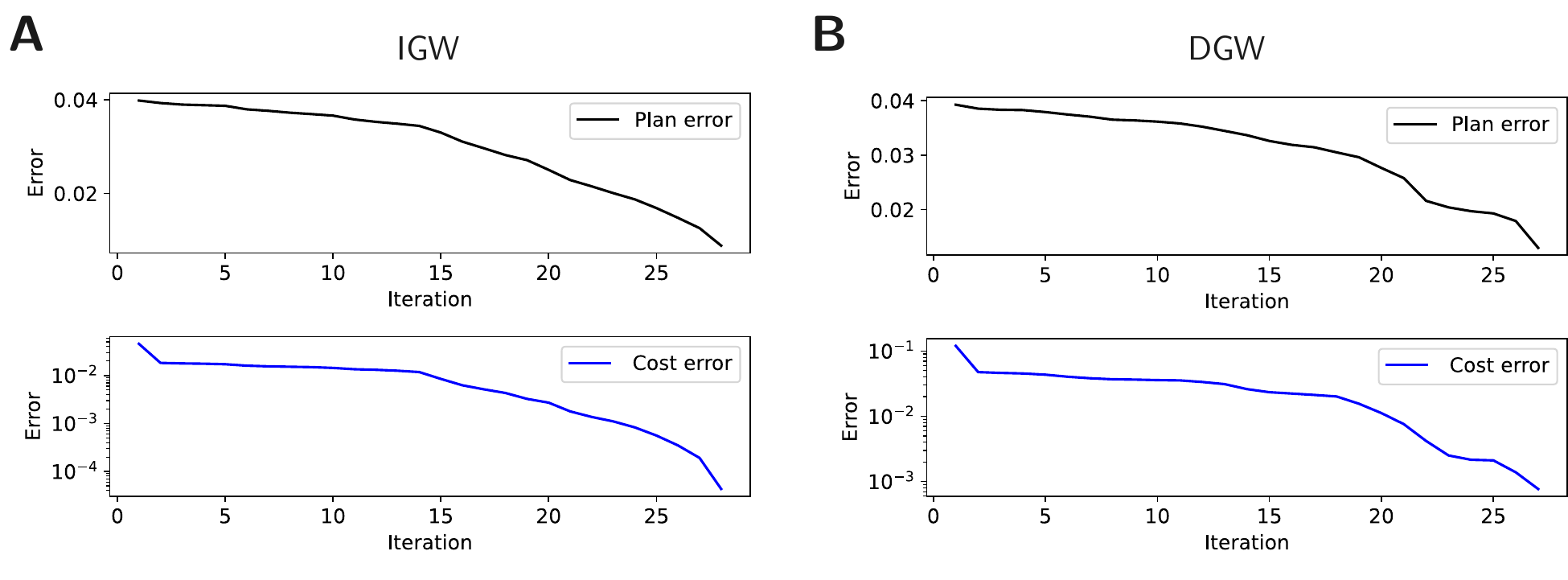}
    \caption{Convergence of the Frank-Wolfe scheme (Algorithm \ref{alg:Frank_Wolfe}) in dimension $d=3$, with 1000 points in the marginals. Plan error corresponds to the Euclidean 2-norm between the final transport plan, and the transport plan at each iteration. The cost error similarly corresponds to the difference between the final squared cost and the squared cost at each iteration. A) Algorithm running with the concave \ref{eq:IGW} cost, which simplifies the line search (see section \ref{sec:supp_line_search}). B) Algorithm running with the \ref{eq:DGW} cost. In practice, the algorithm converges quickly with both costs, typically under 50 iterations.}
    \label{fig:conv_FW}
\end{figure}

\newpage
\begin{figure}[ht!]
    \centering
    \includegraphics[width=0.8\linewidth]{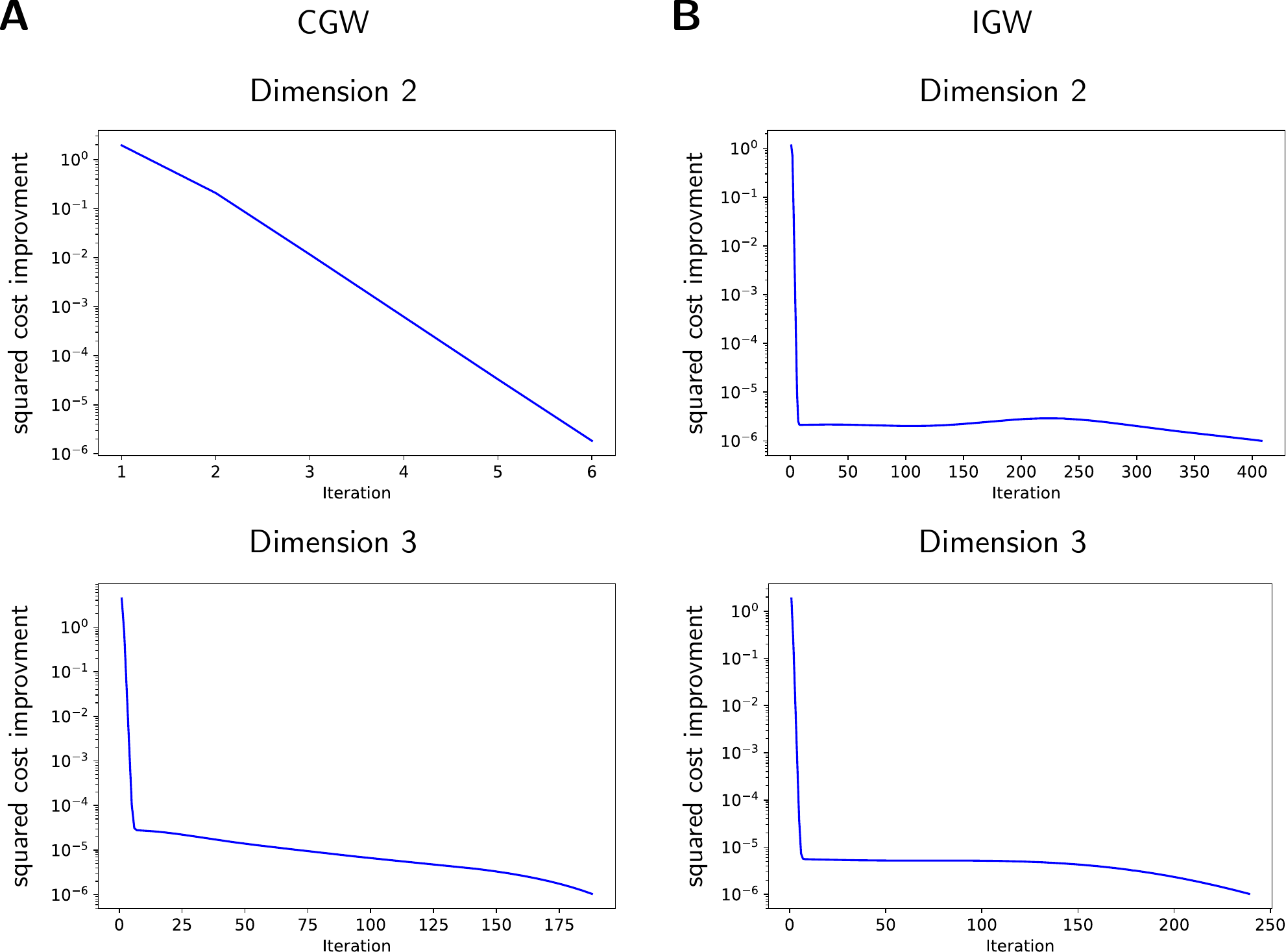}
    \caption{Convergence of the Frank-Wolfe scheme with entropy regularized optimal transport, for 10000 point in the marginals. The improvement in (squared) cost between each iteration is represented. The initial search direction is sampled uniformly on a $d$-dimensional sphere. The regularization parameter is 0.2, with a convergence threshold of  $10^{-6}$ and a maximum of 500 iterations. The algorithm may need more iterations to converge, but it allows for larger data.}
    \label{fig:conv_FW_entrop}
\end{figure}

\end{appendices}
\end{document}